\documentclass{article}
\usepackage{graphicx} 
\usepackage[a4paper, margin=2.5cm]{geometry}
\usepackage{pgfplots}
\usepgfplotslibrary{groupplots}
\pgfplotsset{compat=1.17}
\usepackage{tikz}
\usepackage{tikz-3dplot}
\usepackage{ifthen}
\usepackage{caption}
\usepackage{subcaption}
\usepackage{float}
\usepackage{hyperref}
\usepackage[normalem]{ulem}

\usepackage[thinlines]{easytable}

\newcommand\xm{%
  X\kern-.0em%
  --%
   \hbox{M}\kern+.16em%
   \raise-0.65ex\hbox{--}\kern-.7em%
  \raise0.34ex\hbox{es}%
  \hbox{{h}}%
\@}

\usepackage{amsfonts,amssymb}
\usepackage{amsmath,amsthm}
\usepackage{graphicx}

\newtheorem{lemma}{Lemma}
\newtheorem{theorem}[lemma]{Theorem}

\theoremstyle{remark}
\newtheorem{remark}{Remark}

\def \IR {\mathbb{R}}
\def \Th {{\mathcal T}_h}
\def \d {{\rm d}}
\def\dx{\,\d x }
\def\dS{\,\d S }

\def\B{\mathcal{B}}
\def\WIi{W^{1,\infty}}
\def\WIIi{W^{2,\infty}}

\def\tu{\tilde{u}}
\def\opt{\mathrm{opt}}
\def\gradn{\nabla_{\!\!\perp}}
\def\gradt{\nabla_{\!\!{\scriptscriptstyle\parallel}}}
\def\tD{\widetilde{D}}
\def\width{\bar{h}}

\newcommand{\J}{J}
\newcommand{\Jmin}{J_{\text{min}}}
\newcommand{\dpartial}[2]{\frac{\partial #1}{\partial #2}}

\pgfplotsset{compat=1.18}

\usepackage{authblk}
\title{The tempered finite element method}
\author[1]{Antoine {Quiriny}}
\author[2]{Václav {Kučera}}
\author[1]{Jonathan {Lambrechts}}
\author[1]{Nicolas {Moës}}
\author[1]{Jean-François {Remacle}}
\affil[1]{Institute of Mechanics, Materials and Civil Engineering (iMMC), Avenue Georges Lemaître 4, 1348 Louvain-la-Neuve, Belgium}
\affil[2]{Faculty of Mathematics and Physics, Charles University, Sokolovská 83, 186 75 Praha 8, Czech Republic}
\date{November 2024}    
\begin{document}
\maketitle
\noindent

\section*{Abstract}
In this paper, we propose a new approach -- the Tempered Finite Element Method (TFEM) -- that extends the Finite Element Method (FEM) to classes of meshes that include zero-measure or nearly degenerate elements for which standard FEM approaches do not allow convergence. First, we review why the maximum angle condition \cite{babuvska1976angle} is not necessary for FEM convergence and what are the real limitations in terms of meshes. Next, we propose a simple modification of the classical FEM for elliptic problems that provably allows convergence for a wider class of meshes including bands of caps that cause locking of the solution in standard FEM formulations. The proposed method is trivial to implement in an existing FEM code and can be theoretically analyzed. We prove that in the case of exactly zero-measure elements it corresponds to mortaring. We show numerically and theoretically that what we propose is functional and sound. The remainder of the paper is devoted to extensions of the TFEM method to linear elasticity, mortaring of non-conforming meshes,  high-order elements, and advection.


\section{Introduction}

Since the early 1970s, the finite element method has been used extensively in  computational physics and engineering.
The fact that there exists a structure of proof for finite element algorithms \cite{ern2004theory} supplemented by the fact that finite elements are able to solve non-linear problems \cite{wriggers2008nonlinear} on possibly complex geometries using unstructured meshes \cite{geuzaine2009gmsh}
make finite elements the method that prevails in most branches of manufacturing industry -- aerospace, automotive, shipbuilding, electrical machines, ...  

Engineering designs are encapsulated in CAD systems. By extension, the word CAD also designates the geometry of the object/system to be analyzed. The analysis process begins with the CAD geometry. Then, the currently prominent method of analysis, finite elements, requires a mesh, which is an alternative (discrete) representation of a geometry. 
Finite elements (unlike certain classes of finite volumes) are relatively robust when it comes to mesh quality. In the previous sentence, we have deliberately used the word \emph{relatively}. Finite elements are resilient to the poor quality of some meshes, but not all. There exist some specific mesh configurations that make the finite element method to become \emph{brittle} i.e. that forbid finite element convergence. In a classical paper \cite{babuvska1976angle}, Babu{\v{s}}ka and Aziz claim that \emph{what is essential is the fact that no angle is close to $180^o$}. Once again, words are everything: Babu{\v{s}}ka and Aziz use the  word \emph{essential} but not the word \emph{necessary}, which has a more precise mathematical definition. Much later, Hannukainen et. al. \cite{hannukainen2012maximum} showed through a very simple geometrical construction that the maximum angle condition is not necessary. In \cite{hannukainen2012maximum}, the authors do not specify whether there are mesh configurations that render the finite element method inoperative. In \cite{babuvska1976angle}, the authors already show that degenerate triangles with two angles close to $90^o$ -- which we'll call needles in the following -- don't cause any problems. Ku{\v{c}}era shows in \cite{kuvcera2016necessary} that a triangle with an angle close to $180^o$ -- which we'll call a cap in the following -- poses no problem \emph{when isolated}. We  can explain using basic linear algebra why isolated caps (or needles) causes no problem in the finite element method and why
only specific patterns should be avoided. We start by studying the stiffness matrix related to the Laplace operator on
a triangle (see Figure \ref{fig:cap}):
$$K_{ij} = \int_{T} \nabla \phi_i \cdot \nabla \phi_j \;
dx~~~,~~~i,j=1,2,3.$$
Assume three scalar values $u_i$, $i=1,2,3$, associated to nodes located at 
$x_i$, $i=1,2,3$. The stiffness matrix, $[K]$, allows one to write the quadratic expression
of the energy:
\begin{equation}
f(u_1,u_2,u_3) = \frac12 \sum_{i=1}^3\sum_{j=1}^3K_{ij} u_i u_j.
\label{eq:quadf}
\end{equation}
\begin{figure}[h!]
\begin{center}
        \begin{tikzpicture}
            \draw (0,0) node {$\bullet$} ;
            \draw (5,1.75) node {$\bullet$} ;
            \draw (8,0) node {$\bullet$} ;
            \draw (5,0) node {$\times$} ;
            \draw[] (0,0)  node[left]{${x}_1$} --  (5,1.75) node[above]{${x}_3$} --  (8,0)  node[right]{${x}_2$} -- (0,0);
            \draw[->] (-0.1,0.75)--(0.4,0.75) node[right]{$x$};
            \draw[->] (-0.1,0.75)--(-0.1,1.25) node[above]{$y$};
            \draw[dashed] (5,1.75)--(5,0.) ;
            \draw[<->,color=gray] (0,-0.3)--(5,-0.3) ;
            \draw (2.5,-0.3)node[below]{$l_1$} ;
            \draw[<->,color=gray] (5,-0.3)--(8,-0.3) ;
            \draw (6.5,-0.3)node[below]{$l_2$} ;
            \draw[<->,color=gray] (8.7,0.)--(8.7,1.75) ;
            \draw (8.7,0.875)node[right]{$h$} ;
            \draw[] (1.5,0) arc (0:9:1.5) node[right]{$\theta_1$} ;
            \draw[] (1.5,0) arc (0:19:1.5) ;
            \draw[] (6.5,0) arc (180:165:1.5) node[left]{$\theta_2$} ;
            \draw[] (6.5,0) arc (180:150:1.5);
            \draw[] (5,1.25) arc (-90:-30:0.5);
            \draw[] (5,1.25) arc (-90:-160:0.5);
            \draw[] (4.5, 1.25) node{$\theta_3$};
        \end{tikzpicture}
    \end{center}
    \caption{A triangle $T$. The angle 
    $\theta_i$ is associated to the vertex $x_i$.\label{fig:cap}}
\end{figure}
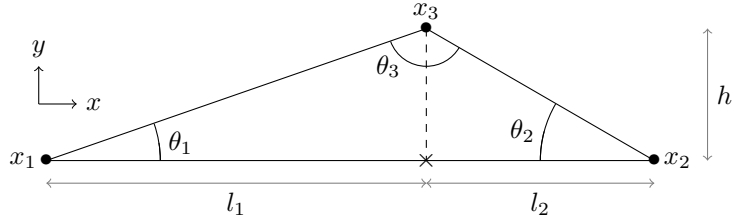
Without loss of generality, we  assume  that $\theta_3 \geq \theta_2 \geq \theta_1$. 
Edge $({x}_1, {x}_2)$ is thus the longest edge of T (see
Figure \ref{fig:cap}) with $\|{x}_2 - {x}_1\| = l_1+l_2$. 
Defining the element flatness $f$ and a symmetry parameter $s$, 
\begin{equation}
    f = \frac{h} {l_1+l_2}, \quad s = \frac{l_1}{l_1+l_2}, \quad f \in\left[0,
      {\sqrt{3}/2}\right], \quad s \in [ 0, 1 ].
    \label{eq:f}
\end{equation}
This local matrix has three eigenvalues $\lambda_1, \lambda_2$ and $\lambda_3$ and three corresponding eigenvectors ${\bf v}_1, {\bf v}_2$ and ${\bf v}_3$ that relate to different modes of the solution. The computation of the eigenvalues, eigenvectors and energies associated are described in appendix \ref{apx:stiff}. The first mode is the rigid
 mode with zero energy.
We are interested in the limit case  
$f \rightarrow
0$ for which the element reaches zero measure. 
It is possible to show that as $f \rightarrow 0$, we have
\begin{equation}{\lambda_2 \sim \frac{3}{2} \frac{f}{1-s+s^2}
    \quad \text{and} \quad \lambda_3 \sim  2 \frac{ 1 - s + s^2}{f}} 
\label{eq:eig}
  \end{equation}
As $f
  \rightarrow 0$, the energy of the second mode ${\bf v}_2$ goes to zero linearly in
  $f$ and the energy of the third mode ${\bf v}_3$ goes to infinity as $1/f$.
The energy and therefore the gradients of $u(x,y)$ tend towards 
infinity if the ${\bf v}_3$ mode corresponding to $\lambda_3$ is active. 
It is possible to show that $$\lim_{f\rightarrow 0} {\bf v}_2 = (1-s, s, -1)$$ and ${\bf v}_2$ is always orthogonal to ${\bf v}_3$.
Therefore, any $u$ with finite energy can be written as
linear combinations of ${\bf v}_1$ and ${\bf v}_2$:
  \begin{equation}u_3 = (1-s) u_1 + s u_2.\label{eq:lin}\end{equation}
 In a degenerated element, values of ${u}_1$ and ${u}_2$ can be chosen arbitrarily but $u_3$ {\bf must be the linear 
interpolation between $u_1$ and $u_2$ on edge ${ x}_1 {x}_2$.}

In the limit $f\rightarrow 0$, any $u$ field that does not satisfy \eqref{eq:lin} has an infinite energy. When solving for the global finite element system $Ax=b$, there exists in the limit $f\rightarrow 0$ an infinite eigenvalue in $A$ corresponding to an eigenmode ${\bf x}_{\infty}$. 
Since the solver minimizes the energy, it will choose a solution ${\bf x}$ that does not contain ${\bf x}_{\infty}$. The degenerate cap will simply act as a mortar between its three neighbors, imposing  condition \eqref{eq:lin}. 

Now what if we consider a band of caps like the one in Figure \ref{fig:2D_band}.
At the $\bar{h} \rightarrow 0$  limit, the solution $u_2$ at point $x_2$ is constrained by the solution at points $x_1$ and $x_3$. The solution at $x_3$ is itself constrained by the solution at $x_2$ and $x_4$, and so on. So the $u$ solution is forced to be linear along the band. This locking phenomenon leads to the fact that $h$ is no longer the mesh size in the band but $L$ is.
And, if $L$ is held fixed, convergence is lost. By making such an argument more rigorous, Ku{\v{c}}era \cite{kuvcera2016necessary} was able to demonstrate that for a band of caps of fixed length $L$, classical finite elements have $O(h)$-convergence if and only if $\bar{h} \ge C h^2$ is satisfied, where $C$ is a constant independent of the mesh (cf. also Remark \ref{rem:FEM}). Finite elements have in fact a greater tolerance to mesh ``quality" than expected.

Note that for a needle we have either $s=0$ or $s=1$. The condition for finite energy forces both nodes of the collapsed edge to have the same value.
In a needle, the acceptable mode ${\bf v}_2$ couples only two nodes, and a band of needles causes no concern.

The fact that bands of needles are not a problem for finite element convergence is in agreement with the historical development of the method. The first mathematical analysis of the finite element method was performed independently by Zl\'{a}mal and \v{Z}en\'{i}\v{s}ek in the papers \cite{Zlamal} and \cite{Zenisek}. In these papers optimal convergence is proved under the so-called \emph{minimum angle condition} which assumes that all angles in the mesh are uniformly bounded away from zero. Such a condition precludes the presence of needles in the mesh. However, these results were later improved independently by several authors to the so-called \emph{maximum angle condition}, cf. e.g. \cite{babuvska1976angle}. This condition assumes that the angles in the mesh are uniformly bounded away from $180^o$, hence needles are allowed but not caps. The most general form of the maximum angle condition can be formulated using the \emph{circumradius} of the triangles in the mesh \cite{kobayashi2015circumradius}. The maximum angle condition can be shown to be optimal when estimating Lagrange interpolation, \cite{Kucera_circumradius}. However, for finite elements it is known that the circumradius condition is not necessary for convergence, cf. \cite{hannukainen2012maximum}. The special case of the band of caps which is relevant to the present paper is thoroughly analyzed in \cite{kuvcera2016necessary}. We note that we have only discussed the situation in 2D. Much less is known about the 3D case, where even a necessary and sufficient condition for Lagrange interpolation is not known. There are many partial results concerning such quantities as vertex and dihedral angles, projected circumradii, etc.

In the paper \cite{kuvcera2016necessary}, Ku{\v{c}}era showed that very strange triangulations allows finite element convergence, stating that \emph{such meshes perhaps do not have any value from the practical point
  of view}.
Recently, we developed a new interface tracking method, called \xm\, \cite{degrooff2024simulation,moes2023extreme,quiriny2024x}, which takes advantage of meshes with zero or close to zero-measure elements. The principle is simple: track the interface by mesh relaying. The interface is handed over from layers of nodes to new layers of nodes across elements which can be rather thin in the direction of the front propagation. These elements may even possibly form bands of caps. 

This is the motivation for this paper: 
be able from an \xm\  perspective
to compute easily with bands of caps.
This requires a more robust finite element approach that we have named 
the \emph{tempered finite element method}.
 This name is inspired by an analogy with metallurgy. The problem with finite elements on a band of caps is that the numerical solution undergoes a locking phenomenon which prevents proper functioning of the method and convergence. In a material-science analogy, after discretization the `material' becomes extremely hard on a band of caps and the numerical solution can no longer conform to the exact solution. The method we propose  alleviates this by locally `softening' the material on the band. In Poisson's problem this will correspond to locally decreasing the diffusion parameter in a mesh-dependent way. This allows the numerical solution to conform to the continuous problem. This process is analogous to that of tempering in metallurgy. Steel can be made very hard by heating it to a high temperature and then rapidly cooling it in water or oil. This process is called \emph{quenching} and it makes the steel extremely hard, but also very brittle. \emph{Tempering} is then performed, when the steel is slowly reheated to an intermediate temperature and slowly cooled, allowing the internal stresses of the material to be relieved. This makes the alloy more flexible and ductile, while preserving most of the hardness. This is analogous to our approach in the FEM, where we decrease the diffusion coefficient (`soften') to increase flexibility, while preserving other desired qualities (e.g. good approximation properties). For this reason, we call the method \textbf{tempered finite elements (TFEM)}.

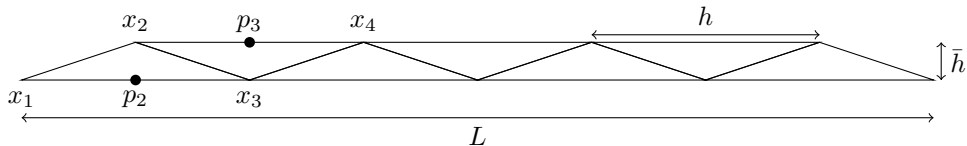
\begin{figure}[!ht]
\begin{center}
    \begin{tikzpicture}
        \draw (0,0) -- (3,0) -- (1.5,0.5) -- cycle;
        \draw (3,0) -- (1.5,0.5) -- (4.5,0.5) -- cycle;
        \draw (3,0) -- (4.5,0.5) -- (6,0) -- cycle;
        \draw (4.5,0.5) -- (6,0) -- (7.5,0.5) -- cycle;
        \draw (6,0) -- (7.5,0.5) -- (9,0) -- cycle;
        \draw (7.5,0.5) -- (9,0) -- (10.5,0.5) -- cycle;
        \draw (9,0) -- (10.5,0.5) -- (12,0) -- cycle;
        \draw[<->] (12.1,0) -- (12.1,0.5) node[midway, right] {$\bar{h}$};
        \draw[<->] (7.5,0.6) -- (10.5,0.6) node[midway, above] {$h$};
        \node at (0,-0.25) {$x_1$};
        \node at (3,-0.25) {$x_3$};
        \node at (1.5, 0.75) {$x_2$};
        \node at (4.5, 0.75) {$x_4$};
        \node at (1.5, -0.25) {$p_{2}$};
        \fill (1.5,0) circle (2pt);
        \node at (3, 0.75) {$p_{3}$};
        \fill (3,0.5) circle (2pt);
        \draw[<->] (0,-0.5) -- (12,-0.5) node[midway, below] {$L$};
    \end{tikzpicture}
\end{center}
\caption{Band of caps triangles in 2D.} \label{fig:2D_band}
\end{figure}
The paper is structured as follows.
First we perform a simple numerical experiment on a manufactured Laplacian to show that, using a simple numerical trick, we are able to recover finite element convergence. The TFEM solution on a mesh containing a band of caps is compared with the finite element solution on a regular mesh. Section 3 is dedicated to a rigorous mathematical analysis, confirming the numerical experiment. We demonstrate that the TFEM convergence has the same rate as the one of the traditional finite element method on regular meshes. In the next section, we present extensions to the TFEM method and its usefulness in various situations: linear elasticity, non-conformal mesh mortaring, high-order elements, and advection. 
\section{A simple numerical experiment}
\label{sec:experiment}
In this section, we solve a classical finite element problem on a series of meshes exhibiting a band of degenerate caps and observe 
that a simple \emph{numerical trick} is enough to avoid the locking phenomenon. 

\subsection{Finite element formulation of the 2D Poisson problem}
\label{section:FEM}
Let $\Omega \subset \mathbb{R}^2$ be a planar domain with its boundary $\partial \Omega$.
We seek $u$, the solution of:
\begin{align} \label{eq:problem}
    -\Delta u &= f \hspace{1.25cm}  \text{on } \Omega, \\
    u &= u_D \hspace{1cm} \text{on } \partial \Omega. \nonumber
\end{align}
Defining the space $V = \{u \in H^1(\Omega)~,~ u = u_D~\text{on}~\partial \Omega\}$, the corresponding weak form  is to find $u \in V$, such that 
\begin{equation} \label{eq:weak}
    \int_{\Omega} \nabla u \cdot \nabla v \dx = \int_{\Omega} f v \dx \, , \hspace{1cm} \forall v \in V^0,
\end{equation} \\
where $V^0 = H^1_0(\Omega) = \{v \in H^1(\Omega)~,~ v = 0~\text{on}~\partial \Omega\}.$
This problem is solved on a mesh denoted $\mathcal{T}_{h}$ composed of triangular elements of size $h$. The finite element method constructs a space ${X_h} \subset H^1(\Omega)$ of continuous and piecewise linear function over the mesh along with the space $V_h=\{v_h\in X_h~,~ v_h =\tilde{u}_D\text{ on }\partial\Omega\}$, where $\tilde{u}_D$ is a piecewise linear interpolation of $u_D$ on $\partial\Omega$. By defining $V^0_h = X_h \cap V^0$, we obtain the discrete version of (\ref{eq:weak}). 
\vspace{0.2cm} \\
Find $u_h \in V_h$ such that 
\begin{equation} \label{eq:discrete}
    \int_{\Omega} \nabla u_h \cdot \nabla v_h \dx = \int_{\Omega} f v_h \dx \, , \hspace{1cm} \forall v_h \in V^0_h.
\end{equation}
Following the classical finite element approach the matrix is obtained as an assembly of local matrices $K_{ij}$ evaluated on a single element $E$, $i$ and $j$ being two nodes of the element $E$:
\[
  K_{ij}
  = \int_E \nabla \phi_i \cdot \nabla \phi_j \d x. \nonumber
\]
A usual way of obtaining the gradients of the shape functions and integrating them is via a parent element. The coordinates of our element $E$ are expressed in $\mathbf{x} = [x,y]$ while for the parent element they are expressed in $\boldsymbol{\xi} = [\xi, \eta]$. The mapping between both coordinate systems  is represented in Figure \ref{fig:mapping}.
\begin{figure}[!ht]
\begin{center}
    \begin{tikzpicture}
        \draw[] (0,0)  --  (2,0.4)  --  (4,0)  -- (0,0);
        \draw[->] (-0.1,0.5)--(0.4,0.5) node[right]{$x$};
        \draw[->] (-0.1,0.5)--(-0.1,1.0) node[above]{$y$};
        \draw[] (6,0)
        -- (7.5,0) node[below]{$(1,0)$}
        -- (6,1.5) node[left]{$(0,1)$}
        -- (6,0) node[below]{$(0,0)$};
        \draw[->] (6.0,0)--(8.0,0) node[right]{$\xi$};
        \draw[->] (6.0,0)--(6.0,2.0) node[above]{$\eta$};
        \draw[-latex] (2.9,  -0.5)  arc  (-130:-50:2)  node [midway, above] {$\boldsymbol{\xi}(\boldsymbol{x})$};
        \draw[-latex] (5.5, 2)  arc  (50:130:2) node[midway, below] {$\boldsymbol{x}(\boldsymbol{\xi})$};
        \node at (3,0.6) {$E$};
        \node at (7,1) {$\hat{E}$};
    \end{tikzpicture}
\end{center}
\caption{Mapping between the physical and reference domains.} \label{fig:mapping}
\end{figure}
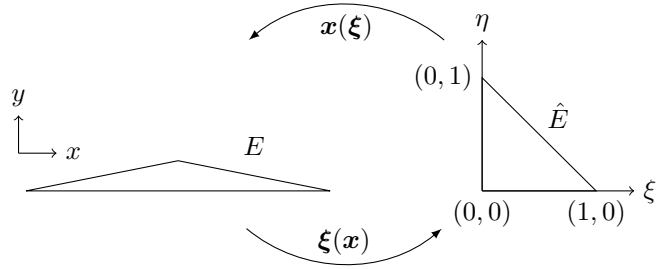

\noindent The gradient of 
$\phi$ in the element $E$ can then be obtained via the chain-rule $\dpartial{\phi}{\mathbf{x}} = \dpartial{\phi}{\boldsymbol{\xi}}\dpartial{\boldsymbol{\xi}}{\mathbf{x}}$. The calculation of the gradient of the transform $\dpartial {\boldsymbol{\xi}}{\mathbf{x}} $ is given by:
\[
  \dpartial{\boldsymbol{\xi}}{\boldsymbol{x}} = \left( \dpartial{\boldsymbol{x}}{\boldsymbol{\xi}} \right)^{-1} = \frac{1}{\J} \begin{bmatrix}
    \dpartial{y}{\eta} & -\dpartial{x}{\eta} \\
    - \dpartial{y}{\xi} & \dpartial{x}{\xi} 
  \end{bmatrix}
\]
where $\J = \dpartial{x}{\xi}\dpartial{y}{\eta}-\dpartial{x}{\eta}\dpartial{y}{\xi}$ is the mapping determinant.
For a triangular element, $\J$ is constant and equals to two times the area of the element.
$\J$ also appears in the elementary surface of the integral $\d \boldsymbol{x} = J \d \boldsymbol{\xi}$. Finally, the local matrix $K_{ij}$ can be expressed as:
\begin{equation}
  K_{ij}
  =
  \int_{\hat{E}}
 \frac{1}{\J}
  \begin{bmatrix}
    \dpartial {\phi_i}{\xi} & \dpartial{\phi_i}{\eta}
  \end{bmatrix} 
  \begin{bmatrix} 
    \dpartial y\eta & -\dpartial x\eta\\
    -\dpartial y\xi & \dpartial x\xi
  \end{bmatrix} 
  \begin{bmatrix} 
    \dpartial y\eta & -\dpartial y\xi\\
    -\dpartial x\eta & \dpartial x\xi
  \end{bmatrix} 
    \begin{bmatrix}
      \dpartial {\phi_j} \xi \\ \dpartial {\phi_j} \eta
    \end{bmatrix}
  \, \d \boldsymbol{\xi}.
  \label{eq:local_matrix}
\end{equation}

\subsection{Tempered finite element on a degenerate mesh}
\begin{figure}
    \centering
  \begin{tikzpicture}[scale=0.95]
    \node at (0,0) {\includegraphics[width=5cm]{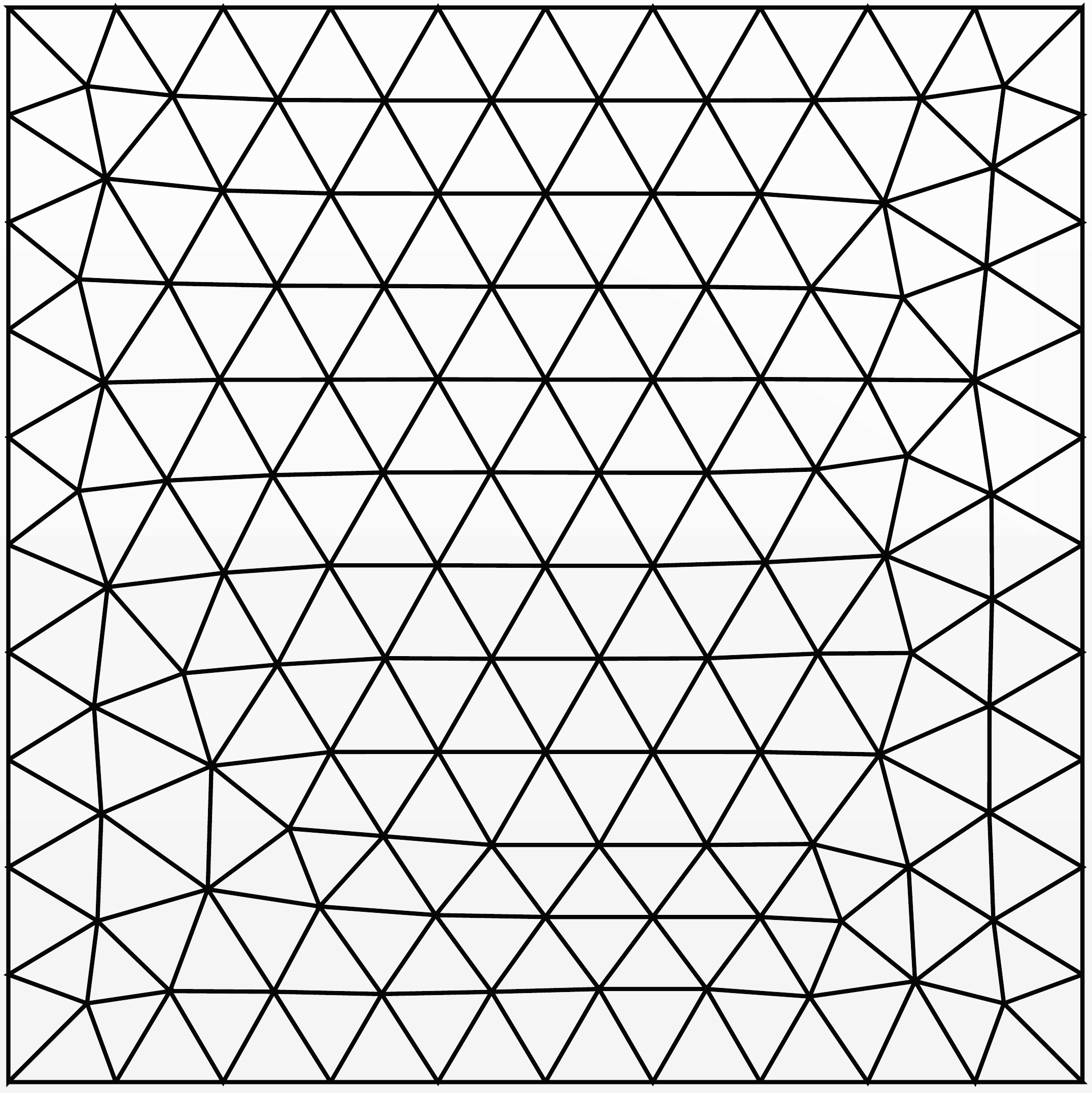}};
    \node at (5.5,0) {\includegraphics[width=5cm]{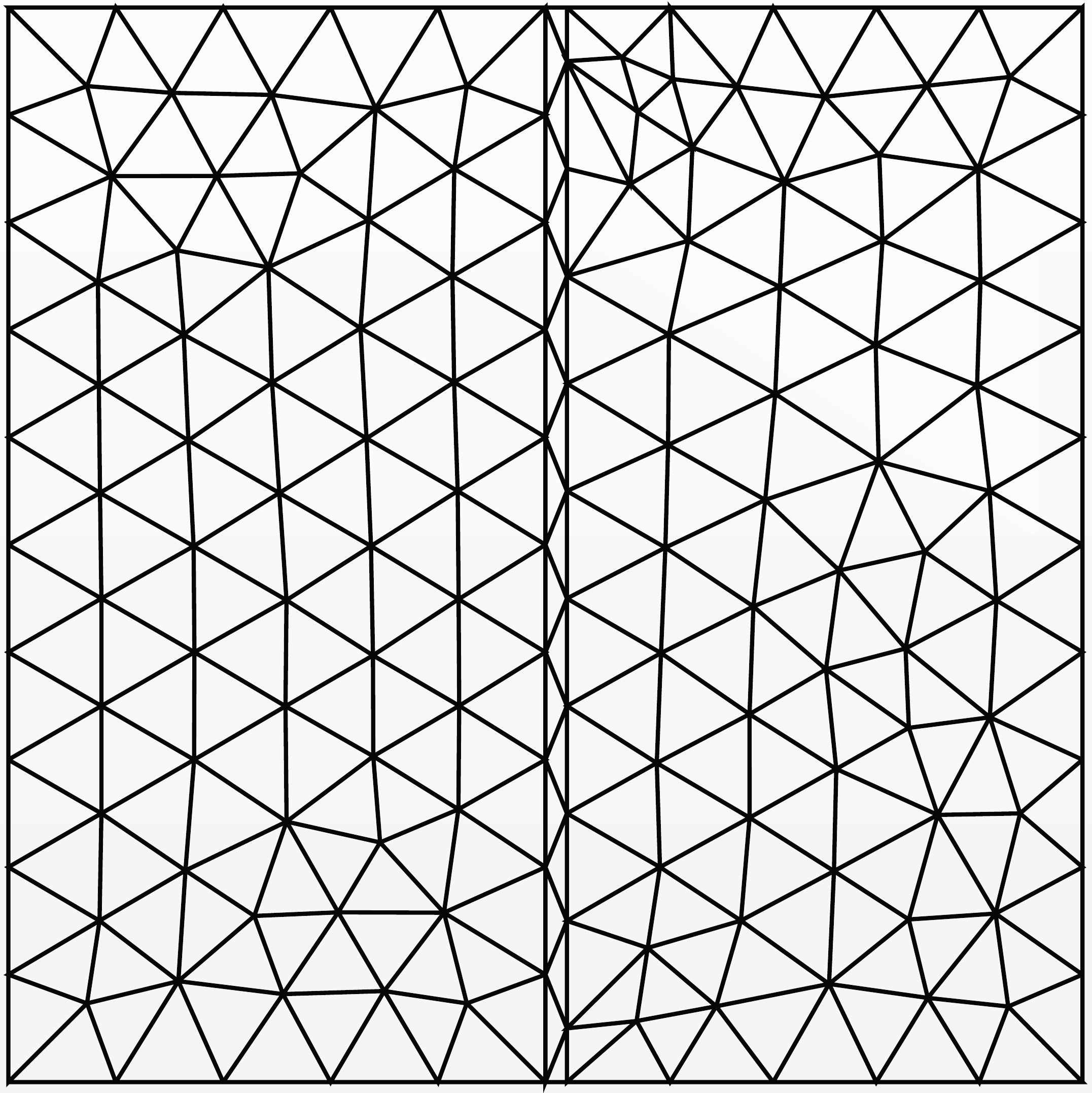}};
    \node at (11,0) {\includegraphics[width=5cm]{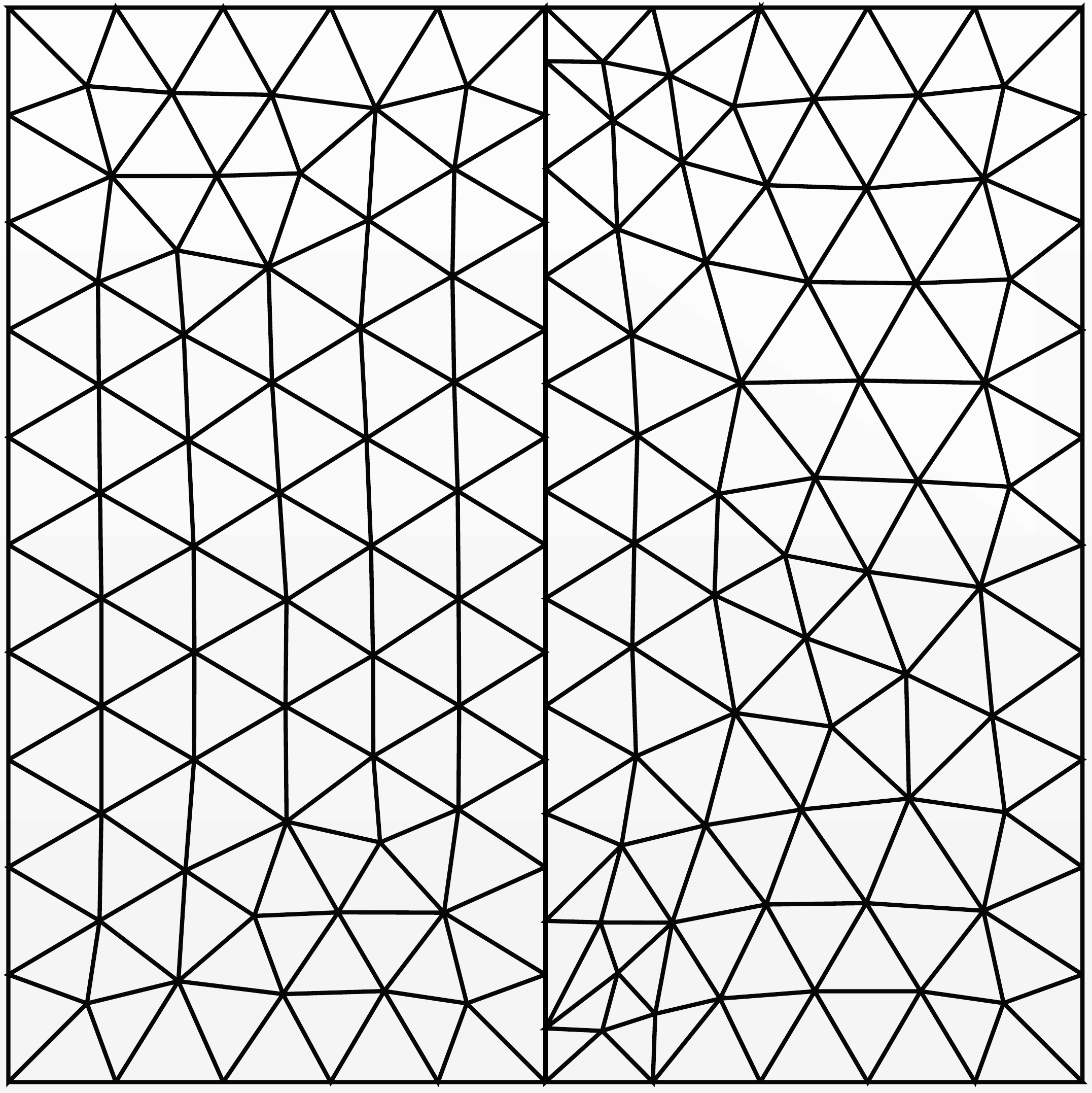}};
    \draw[gray, fill=gray, nearly transparent] (5.5,-2.57) rectangle (5.6,2.6);
    \node at (5.65, -2.85) {$\mathcal{B}$};
    \draw [<->] (5.45,2.8) -- node[above=1mm] {$\width$} (5.65,2.8);
    \draw [<->] (4.47,-2.7) -- node[below] {$h$} (4.98,-2.7);
  \end{tikzpicture}
    \caption{Three 2D meshes: regular (left), with a band of caps (center) and with a fully degenerated band of caps (right).}
    \label{fig:meshes}
\end{figure}
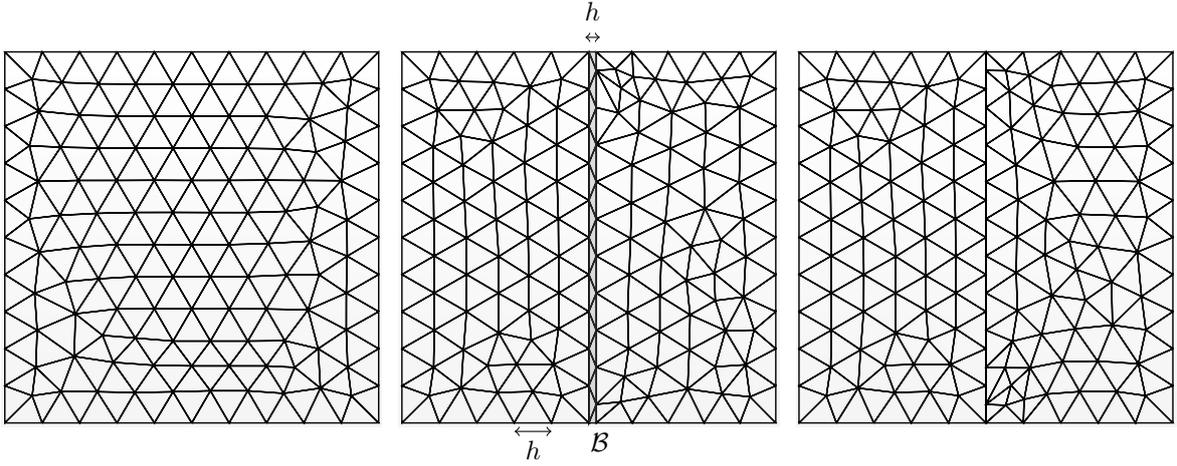
\begin{figure}[!ht]
  \centering
  \begin{tikzpicture}[scale=0.95]
    \node at (0,0) {\includegraphics[width=3.7cm]{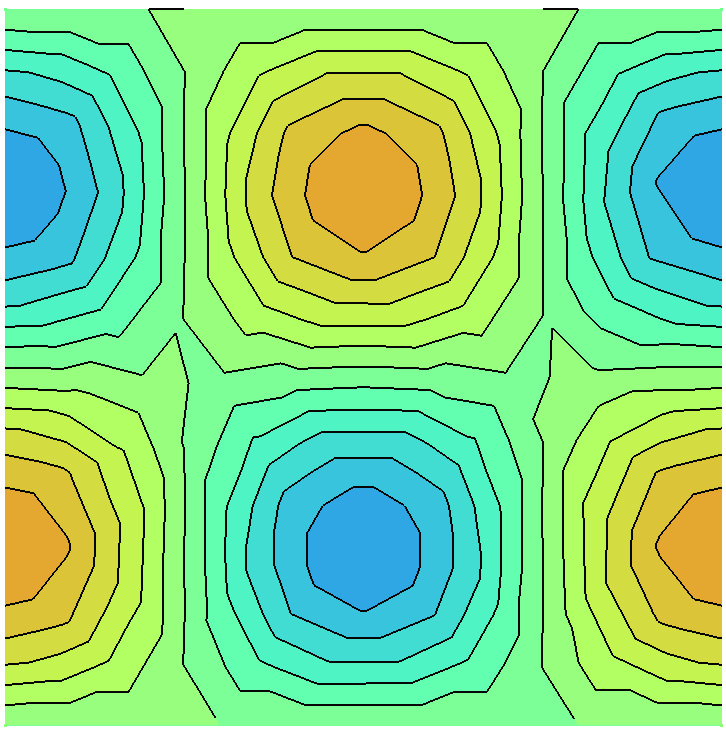}};
    \node at (0,2.3) {a) Regular};
    \node at (4,0) {\includegraphics[width=3.7cm]{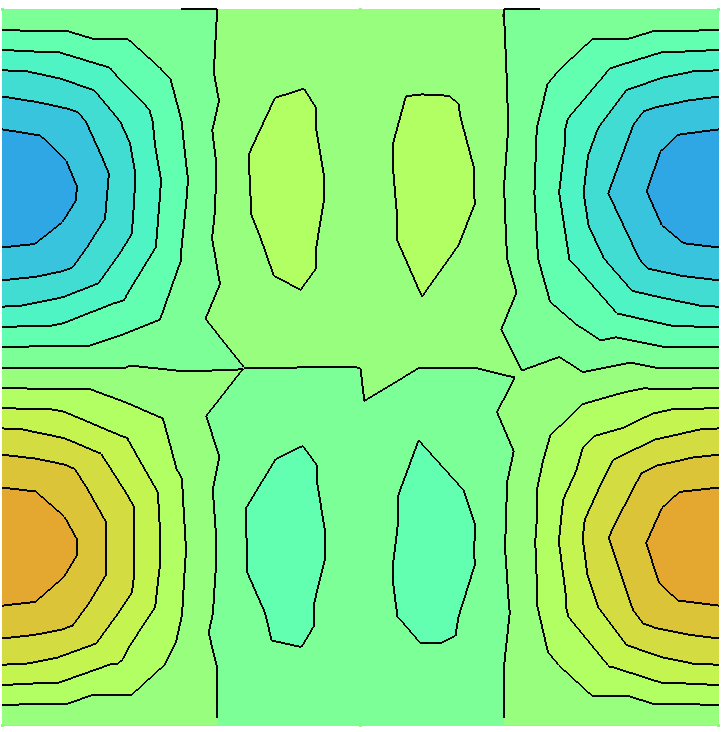}};
    \node at (4,2.3) {b) Band $\Jmin = 10^{-8}$};
    \node at (8,0) {\includegraphics[width=3.7cm]{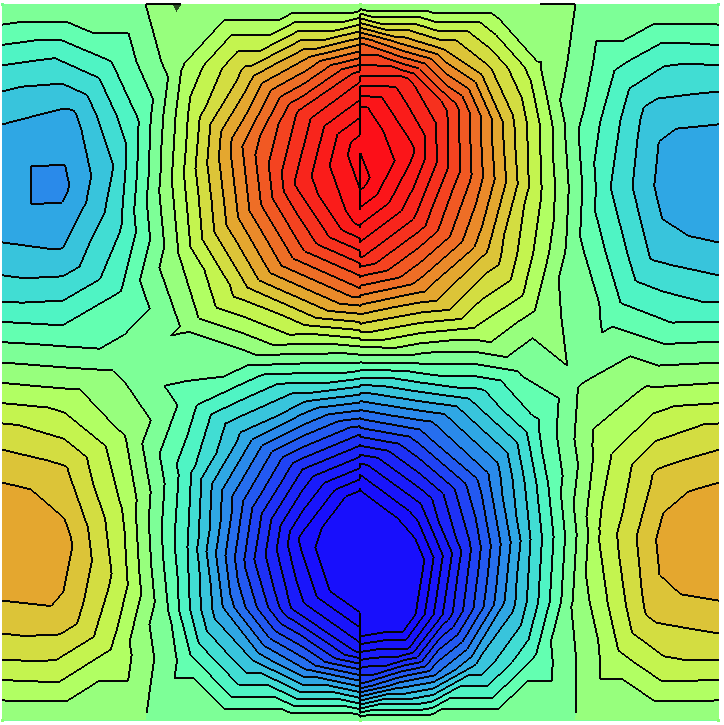}};
    \node at (8,2.3) {c) Band $\Jmin = 10^{-2}$};
    \node at (12,0) {\includegraphics[width=3.7cm]{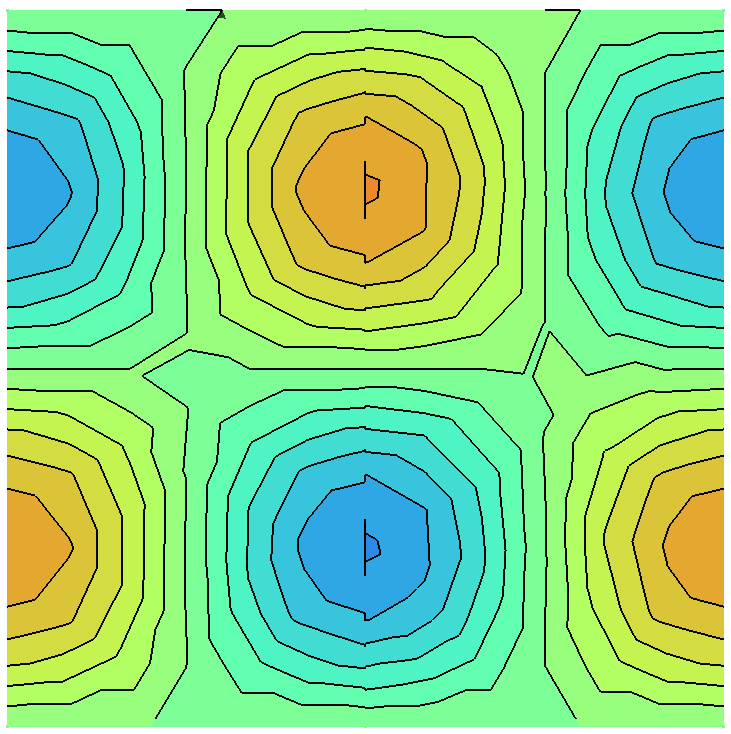}};
    \node at (12,2.3) {d) Band $\Jmin = 10^{-3}$};
    \node at (6,-2.3) {\includegraphics[width=10cm]{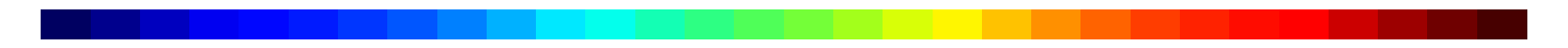}};
    \node at (0.5,-2.3) {$-2.5$};
    \node at (11.5,-2.3) {$2.5$};
    \draw [draw=black, thick] (-1.93,-1.93) rectangle (1.93,1.93);
    \draw [draw=black, thick] (2.07,-1.93) rectangle (5.93,1.93);
    \draw [draw=black, thick] (6.07,-1.93) rectangle (9.93,1.93);
    \draw [draw=black, thick] (10.07,-1.93) rectangle (13.93,1.93);
  \end{tikzpicture}
    \caption{Solution of the manufactured solution on the regular mesh and on a mesh with a degenerated band with $h=\frac{1}{10}$ for different values of $\Jmin$.}
    \label{fig:sol_manufactured}
\end{figure}
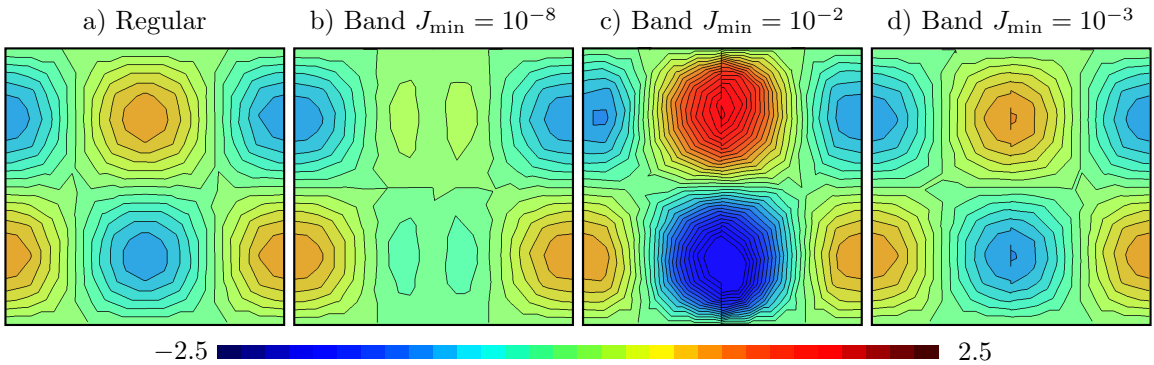
The rightmost mesh of Figure \ref{fig:meshes} 
looks non-conforming with T-junctions. This is in fact not the case. The mesh has been progressively deformed, starting from Figure \ref{fig:meshes} (center), to bring two sub-meshes into contact by "crushing" 
the intermediate \emph{band of caps} denoted $\mathcal{B}$. 
As the band reaches a zero thickness, the finite element is forced to
be linear (or constant) along the band -- this is precisely the locking phenomenon described in the Introduction.
In practice however, it is not possible to assemble the stiffness matrix \eqref{eq:local_matrix} for a zero thickness band since 
it results in a division by zero.
A handy way to avoid this division is to modify $K \rightarrow \tilde K$ by
bounding the value of $\J$ away from zero
with some value denoted $\Jmin$.
Equation \eqref{eq:local_matrix} becomes:
\begin{equation}
  \tilde K_{ij}
  = 
  \int_{\hat{E}}
\frac{1}{\text{max}(\J, \Jmin)}
  \begin{bmatrix}
    \dpartial {\phi_i}{\xi} & \dpartial{\phi_i}{\eta}
  \end{bmatrix} 
  \begin{bmatrix} 
    \dpartial y\eta & -\dpartial x\eta\\
    -\dpartial y\xi & \dpartial x\xi
  \end{bmatrix} 
  \begin{bmatrix} 
    \dpartial y\eta & -\dpartial y\xi\\
    -\dpartial x\eta & \dpartial x\xi
  \end{bmatrix} 
    \begin{bmatrix}
      \dpartial {\phi_j} \xi \\ \dpartial {\phi_j} \eta
    \end{bmatrix}
  \, \d \boldsymbol{\xi}.
  \label{eq:local_matrix_mod}
\end{equation}
In the rest of this paper we will note $\tilde{u}_h$ the solution obtained using this modified matrix.
Choosing the smallest possible nonzero value for $\Jmin$ allows one to achieve optimum finite element convergence for isolated caps or needles and even for bands of needles.

In the case of the band of caps, the concern is that we are not interested in finding the zero-measure limit solution because of locking. The question is then: is there a value for $\Jmin$, possibly dependent on $h$, which is small enough to allow optimal convergence but large 
enough to avoid locking? The answer is yes, and that's what this paper is all about. We shall call the resulting numerical method corresponding to \eqref{eq:local_matrix_mod} the \emph{Tempered Finite Element Method} (TFEM).

In the case when the band elements are not exactly degenerate ($\J > 0$), we can express the modified local matrix from equation \eqref{eq:local_matrix_mod} as
\begin{equation} \label{eq:Ktilde_K}
     \tilde K_{ij}
  = \frac{\J}{\max(\J, \Jmin)} K_{ij}.
\end{equation}
We call this coefficient $d(x) = \frac{\J}{\max(\J, \Jmin)}$. This corresponds to modifying the basic FEM scheme \eqref{eq:discrete} by introducing a variable diffusion parameter $0<d(x)\leq 1$ which equals to a carefully chosen value $0<D = \frac{\J}{\Jmin}<1$ on $\B$. Define 
\begin{equation}
d(x)=\begin{cases}
1,& x\in\Omega\setminus\B,\\
D,& x\in\B.
\end{cases}
\label{eq:def_D}
\end{equation}
The Tempered FEM problem can then be reformulated as finding $\tu_h\in V_h$ such that
\begin{equation}
\int_\Omega d(x)\nabla \tu_h\cdotp\nabla v_h \dx=(f,v_h),\quad\forall v_h\in V_h^0.
\label{eq:FEM_mod}
\end{equation}

Equation \eqref{eq:FEM_mod} motivates the name `tempered' FEM. We effectively lower the diffusion coefficient locally to recover good approximation properties of the original FEM. As explained in more detail in the Introduction, we view this as local `softening' of the material which we regard as analogous to \emph{tempering} in metallurgy.

In the case when the elements in the band $\B$ are exactly degenerate ($J=0$), deriving a formulation similar to \eqref{eq:FEM_mod} is not as straightforward. One must consider the limit of \eqref{eq:FEM_mod} as $\width\to 0$. This is done in detail in Section \ref{section:zero-measure}. Surprisingly, the resulting scheme \eqref{eq:FEM_zero} is a penalization (or mortaring) method.

What remains is to determine an appropriate value of $\Jmin$. We start with a numerical experiment. Let's consider two manufactured problems, one in 2D and the other one in 3D. 
The domains $\Omega$ are the unit square in 2D and the unit cube in 3D.
The exact solutions are respectively 
$u=\sin (2\pi x) \cos (2 \pi y)$ in 2D and $u= \cos (2\pi x) \sin (2 \pi y) \cos (2 \pi z)$ in 3D.
The source term in \eqref{eq:problem} is 
$f = 8 \pi^2 \sin (2\pi x) \cos (2 \pi y)$ in 2D and $f = 12 \pi^2 \cos (2\pi x) \sin (2 \pi y) \cos (2 \pi z)$ in 3D.

As illustrated in Figure \ref{fig:sol_manufactured}, a small $\Jmin=10^{-8}$ which would have worked perfectly for isolated caps or needles leads to locking.
On the other hand, a large $\Jmin=10^{-2}$ leads to a band that is, say, ``too elastic" and solution jumps of too large a magnitude are observed across the band.
However, an appropriate choice of $\Jmin=10^{-3}$ leads to a solution that is visually close to the reference solution on a standard mesh. The solution with $\Jmin=10^{-3}$ still contains jumps but these are limited.

\subsection{Convergence}
An appropriate choice of $\Jmin$ seems to provide a qualitatively acceptable solution. How does it compare to the solution on a regular mesh?
Is it possible to choose $\Jmin$ so that the convergence rate is the same as the one of the classical finite elements?
To answer these questions, we compare the $L^2$ error and the $H^1$ error outside the band  on two series of meshes: a series of regular successively refined meshes with different element size $h$ and a series of meshes with the same mesh size and a band of cap in the middle.  
We refine the mesh and observe how $\Jmin$ must vary in order to make the error converge. We assume that $\Jmin$ evolves as a power of $h$ such that :
\begin{align*}
    \Jmin = C h^k,
\end{align*}
where $C$ is a parameter that will be analyzed later and $k$ a coefficient that must be adequately chosen. 
Since we have observed that the choice of $\Jmin=10^{-3}$ was judicious for a mesh of size $h=\frac{1}{10}$, we choose $C = 10^{k-3}$ so that we get this value for any choice of $k$. We can see in Figure \ref{fig:conv_2D} (left) that the curves (except the reference) starts from the same point for this reason.
The 2D and 3D convergence studies for the $L^2(\Omega)$ and $H^1(\Omega \setminus \B)$ errors are shown  in Figures \ref{fig:conv_2D} and \ref{fig:conv_3d}. These error were computed by comparing to the analytical solution and integrated by a 4th order quadrature rule. Different choices of $k$ are studied. It can be seen that in the 2D case, the choice of $k=3$ gives the optimal rate of convergence of the error in the $L^2$ norm. For the $H^1$ semi-norm, the optimal choice of $k$ lies  in the range $[2,4]$. 
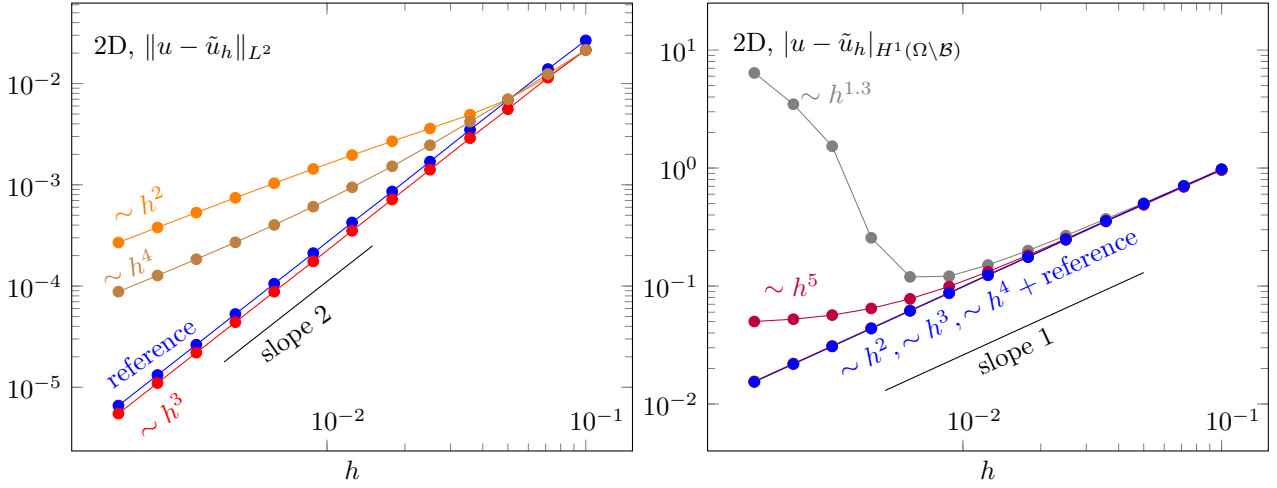
\begin{figure}[!ht]
    \hspace{-0.75cm}%
    \begin{tikzpicture}
        \begin{axis}[
            xlabel={$h$},
            ylabel={2D, $\|u - \tilde{u}_h\|_{L^2}$},
            xtick={0.01, 0.1},
            xticklabel style={above, yshift=1ex, xshift=1ex},
            ytick={0.00001, 0.0001, 0.001, 0.01, 0.1},
            legend pos=north west,
            ymajorgrids=true,
            grid=none,
            xmode=log,
            ymode=log,
            width=0.56\textwidth,
            height=0.47\textwidth,
            x label style={at={(axis description cs:0.5,-0.0)},anchor=north},
            y label style={at={(axis description cs:0.2,0.85)},rotate=-90,anchor=south},
        ]
        
        \addplot[
            color=blue,
            mark=*,
            ]
            coordinates {
                (0.1000000000, 0.0266839393) (0.0714285714, 0.0139108849) (0.0500000000, 0.0068525003) (0.0357142857, 0.0034963940) (0.0250000000, 0.0016951800) (0.0178571429, 0.0008593616) (0.0125000000, 0.0004247803) (0.0088495575, 0.0002114534) (0.0062500000, 0.0001055128) (0.0044247788, 0.0000529551) (0.0031250000, 0.0000263535) (0.0022123894, 0.0000132260) (0.0015625000, 0.0000066000)
            };
            
        \addplot[
            color=orange,
            mark=*,
            ]
            coordinates {
                (0.1000000000, 0.0214721462) (0.0714285714, 0.0116559864) (0.0500000000, 0.0069565155) (0.0357142857, 0.0049252109) (0.0250000000, 0.0036027890) (0.0178571429, 0.0026950458) (0.0125000000, 0.0019649096) (0.0088495575, 0.0014348451) (0.0062500000, 0.0010351698) (0.0044247788, 0.0007443398) (0.0031250000, 0.0005316409) (0.0022123894, 0.0003793773) (0.0015625000, 0.0002694593)
            };
        
        \addplot[
            color=red,
            mark=*,
            ]
            coordinates {
                (0.1000000000, 0.0214721462) (0.0714285714, 0.0114294197) (0.0500000000, 0.0055882920) (0.0357142857, 0.0028900441) (0.0250000000, 0.0014094087) (0.0178571429, 0.0007171503) (0.0125000000, 0.0003514386) (0.0088495575, 0.0001752714) (0.0062500000, 0.0000879023) (0.0044247788, 0.0000439796) (0.0031250000, 0.0000219740) (0.0022123894, 0.0000110109) (0.0015625000, 0.0000054950)
            };

        \addplot[
            color=brown,
            mark=*,
            ]
            coordinates {
                (0.1000000000, 0.0214721462) (0.0714285714, 0.0124283812) (0.0500000000, 0.0069988804) (0.0357142857, 0.0042042558) (0.0250000000, 0.0024595430) (0.0178571429, 0.0015239162) (0.0125000000, 0.0009445083) (0.0088495575, 0.0006083028) (0.0062500000, 0.0004015884) (0.0044247788, 0.0002703863) (0.0031250000, 0.0001841136) (0.0022123894, 0.0001270218) (0.0015625000, 0.0000880670) 
            };

        \addplot[color=black, mark=none] 
            coordinates {(0.015, 0.00025) (0.004, 0.00001778)};

            \node [orange, rotate=20] at (rel axis cs: 0.12,0.55) { $\sim h^2$ };
            \node [brown, rotate=20] at (rel axis cs: 0.1,0.41) { $\sim h^4$ };
            \node [blue, rotate=35] at (rel axis cs: 0.14, 0.22) { reference };
            \node [red, rotate=35] at (rel axis cs: 0.16,0.08) { $\sim h^3$ };
            \node [black, rotate=35] at (rel axis cs: 0.40,0.26) { slope 2 };


        \end{axis}
        \end{tikzpicture}%
    \begin{tikzpicture}%
        \begin{axis}[
            xlabel={$h$},
            ylabel={2D, $|u - \tilde{u}_h|_{H^1(\Omega \setminus \B)}$},
            xtick={0.01, 0.1, 1},
            ytick={0.01, 0.1, 1, 10},
            xticklabel style={above, yshift=1ex, xshift=1ex},
            ymin=0.004, 
            ymax=25,
            legend pos=north west,
            ymajorgrids=true,
            grid=none,
            xmode=log,
            ymode=log,
            width=0.56\textwidth,
            height=0.47\textwidth,
            x label style={at={(axis description cs:0.5,-0.0)},anchor=north},
            y label style={at={(axis description cs:0.25,0.85)},rotate=-90,anchor=south},
        ]

        \addplot[
            color=gray,
            mark=*,
            ]
            coordinates {
                (0.1000000000, 0.9636902479) (0.0714285714, 0.7048311021) (0.0500000000, 0.5039307080) (0.0357142857, 0.3706453608) (0.0250000000, 0.2672360391) (0.0178571429, 0.1996628427) (0.0125000000, 0.1506109771) (0.0088495575, 0.1211425292) (0.0062500000, 0.1194474119) (0.0044247788, 0.2561343607) (0.0031250000, 1.5315294912) (0.0022123894, 3.4743380084) (0.0015625000, 6.4004642796)
            };
        
        \addplot[
            color=orange,
            mark=*,
            ]
            coordinates {
                (0.1000000000, 0.9636902479) (0.0714285714, 0.6996743790) (0.0500000000, 0.4939414661) (0.0357142857, 0.3567864380) (0.0250000000, 0.2494729485) (0.0178571429, 0.1782310473) (0.0125000000, 0.1248540975) (0.0088495575, 0.0880640429) (0.0062500000, 0.0624322144) (0.0044247788, 0.0442116142) (0.0031250000, 0.0312244389) (0.0022123894, 0.0221074313) (0.0015625000, 0.0156134177)
            };

        \addplot[
            color=red,
            mark=*,
            ]
            coordinates {
                (0.1000000000, 0.9636902479) (0.0714285714, 0.6956093638) (0.0500000000, 0.4894145218) (0.0357142857, 0.3530243653) (0.0250000000, 0.2466069404) (0.0178571429, 0.1761063878) (0.0125000000, 0.1233299333) (0.0088495575, 0.0869707695) (0.0062500000, 0.0616584973) (0.0044247788, 0.0436625698) (0.0031250000, 0.0308362919) (0.0022123894, 0.0218324634) (0.0015625000, 0.0154191468)
            };
        
        \addplot[
            color=brown,
            mark=*,
            ]
            coordinates {
                (0.1000000000, 0.9636902479) (0.0714285714, 0.6938467007) (0.0500000000, 0.4887001077) (0.0357142857, 0.3529291614) (0.0250000000, 0.2467172649) (0.0178571429, 0.1762515174) (0.0125000000, 0.1234534191) (0.0088495575, 0.0870677041) (0.0062500000, 0.0617295801) (0.0044247788, 0.0437140305) (0.0031250000, 0.0308730793) (0.0022123894, 0.0218586442) (0.0015625000, 0.0154376911)
            };

        \addplot[
            color=purple,
            mark=*,
            ]
            coordinates {
                (0.1000000000, 0.9636902479) (0.0714285714, 0.6935450541) (0.0500000000, 0.4901410911) (0.0357142857, 0.3556585977) (0.0250000000, 0.2509610416) (0.0178571429, 0.1823488201) (0.0125000000, 0.1321581445) (0.0088495575, 0.0991187614) (0.0062500000, 0.0778697720) (0.0044247788, 0.0645660706) (0.0031250000, 0.0566866895) (0.0022123894, 0.0523374927) (0.0015625000, 0.0500033085)
            };

        \addplot[
            color=blue,
            mark=*,
            ]
            coordinates {
                (0.1000000000, 0.9796381199) (0.0714285714, 0.7060265326) (0.0500000000, 0.4975812972) (0.0357142857, 0.3554175512) (0.0250000000, 0.2474763333) (0.0178571429, 0.1758301014) (0.0125000000, 0.1238719022) (0.0088495575, 0.0872497875) (0.0062500000, 0.0616525345) (0.0044247788, 0.0436945648) (0.0031250000, 0.0308110241) (0.0022123894, 0.0218319842) (0.0015625000, 0.0154226963) 
            };

        \addplot[color=black, mark=none] 
        coordinates {(0.05, 0.13) (0.005, 0.013)};
        
        \node [gray, rotate=0] at (rel axis cs: 0.23,0.80) { $\sim h^{1.3}$ };
        \node [purple, rotate=0] at (rel axis cs: 0.15,0.37) { $\sim h^5$ };
        \node [blue, rotate=25] at (rel axis cs: 0.50,0.33) { $\sim h^2, \sim h^3, \sim h^4 + \text{reference}$ };
        \node [black, rotate=25] at (rel axis cs: 0.55,0.22) { slope 1 };

        \end{axis}
        \end{tikzpicture}
        \caption{Convergence of the error in the $L^2$ norm (left) and $H^1$ semi-norm (right) in 2D.}
        \label{fig:conv_2D}
\end{figure}\\
\begin{figure}[!ht]
    \hspace{-0.75cm}%
    \begin{tikzpicture}
        \begin{axis}[
            xlabel={$h$},
            ylabel={3D, $\|u - \tilde{u}_h\|_{L^2}$},
            xtick={0.01, 0.1},
            ytick={0.001, 0.01, 0.1},
            xmin=0.005, xmax=0.12,
            ymin=0.00015, ymax=0.1,
            legend pos=north west,
            grid=none,
            xmode=log,
            ymode=log,
            width=0.56\textwidth,
            height=0.47\textwidth,
            x label style={at={(axis description cs:0.5,-0.0)},anchor=north},
            y label style={at={(axis description cs:0.2,0.85)},rotate=-90,anchor=south},
            xticklabel style={above, yshift=1ex, xshift=0.5ex},
        ]
        
        \addplot[
            color=blue,
            mark=*,
            ]
            coordinates {
                (0.1, 0.04643290078593369) (0.08333333333333333, 0.03276972684947128) (0.06666666666666667, 0.02144780899583039) (0.05, 0.012177305154034883) (0.04, 0.007737666542277612) (0.03225806451612903, 0.005013909681721832) (0.025, 0.0030221421146074613) (0.02, 0.0019384668179921961) (0.015873015873015872, 0.0012151597313434857) (0.0125, 0.0007537859673007643) (0.01, 0.00048093514839835257)
            }; 
            
        \addplot[
            color=orange,
            mark=*,
            ]
            coordinates {
                (0.1, 0.04674452663530721) (0.08333333333333333, 0.029943849787643518) (0.06666666666666667, 0.019383324121442735) (0.05, 0.010214433869905332) (0.04, 0.006314695298849563) (0.03225806451612903, 0.003971014369691118) (0.025, 0.002372243211772429) (0.02, 0.001604038775607533) (0.015873015873015872, 0.0011483254393485855) (0.0125, 0.0008750649003314969) (0.01, 0.0007067839194954369) 
            }; 
        
        \addplot[
            color=red,
            mark=*,
            ]
            coordinates {
                (0.1, 0.0467445266353072) (0.08333333333333333, 0.030892058233764143) (0.06666666666666667, 0.020854910642475504) (0.05, 0.011652736450799806) (0.04, 0.007533739402746608) (0.03225806451612903, 0.004878560920464494) (0.025, 0.00292165458002299) (0.02, 0.0018686894412387669) (0.015873015873015872, 0.0011746920059553033) (0.0125, 0.000727422647729413) (0.01, 0.0004644053472442571)  
            }; 

        \addplot[
            color=brown,
            mark=*,
            ]
            coordinates {
                (0.1, 0.0467445266353072) (0.08333333333333333, 0.03183110479452874) (0.06666666666666667, 0.0223134521136697) (0.05, 0.013181527812446242) (0.04, 0.008945845664417473) (0.03225806451612903, 0.00611957388953101) (0.025, 0.003959497878568212) (0.02, 0.002742782462583584) (0.015873015873015872, 0.0019030624975992492) (0.0125, 0.0013272252537637403) (0.01, 0.0009636781145703335)
            }; 

            \addplot[color=black, mark=none] 
            coordinates {(0.05, 0.007) (0.02, 0.00112)};
            \node [orange, rotate=15] at (rel axis cs: 0.15,0.24) { $\sim h^3$ };
            \node [brown, rotate=25] at (rel axis cs: 0.27,0.38) { $\sim h^5$ };
            \node [blue, rotate=15] at (rel axis cs: 0.1, 0.15) { reference };
            \node [red, rotate=32] at (rel axis cs: 0.3,0.2) { $\sim h^4$ };
            \node [black, rotate=40] at (rel axis cs: 0.60,0.42) { slope 2 };

        \end{axis}
        \end{tikzpicture}%
        \begin{tikzpicture}
            \begin{axis}[
                xlabel={$h$},
                ylabel={3D, $|u - \tilde{u}_h|_{H^1(\Omega \setminus \B)}$},
                xtick={0.01, 0.1, 1},
                ytick={0.1, 1, 10},
                xmin = 0.006,
                ymin = 0.08,
                legend pos=north west,
                ymajorgrids=true,
                grid=none,
                xmode=log,
                ymode=log,
                width=0.56\textwidth,
                height=0.47\textwidth,
                x label style={at={(axis description cs:0.5,-0.0)},anchor=north},
                y label style={at={(axis description cs:0.25,0.85)},rotate=-90,anchor=south},
                xticklabel style={above, yshift=1ex, xshift=1ex},
            ]
            
            \addplot[
                color=blue,
                mark=*,
                ]
                coordinates {
                    (0.1, 1.4665157106670539) (0.08333333333333333, 1.233431263629783) (0.06666666666666667, 0.999709472927631) (0.05, 0.7503042080600089) (0.04, 0.5966226213279184) (0.03225806451612903, 0.4807871415512738) (0.025, 0.373954414990329) (0.02, 0.29978451644199555) (0.015873015873015872, 0.23765029296406195) (0.0125, 0.18737203104359024) (0.01, 0.14983889263098324)
                }; 
                
            \addplot[
                color=orange,
                mark=*,
                ]
                coordinates {
                    (0.1, 1.5959521441991298) (0.08333333333333333, 1.314234103215576) (0.06666666666666667, 1.0810465790862822) (0.05, 0.8251269969944133) (0.04, 0.6724382024990099) (0.03225806451612903, 0.551850102951522) (0.025, 0.4404175171122717) (0.02, 0.36294766502684433) (0.015873015873015872, 0.29886036134985267) (0.0125, 0.2460511613097628) (0.01, 0.20650606469701394) 
                }; 
            
            \addplot[
                color=red,
                mark=*,
                ]
                coordinates {
                    (0.1, 1.59595214419913) (0.08333333333333333, 1.3074460633620961) (0.06666666666666667, 1.0669748250907625) (0.05, 0.8025508886124623) (0.04, 0.6440980490851304) (0.03225806451612903, 0.5186544091407033) (0.025, 0.4026394736358069) (0.02, 0.32205377600465435) (0.015873015873015872, 0.25563954209727996) (0.0125, 0.2012637423970228) (0.01, 0.16096429348362132)
                }; 
    
            \addplot[
                color=brown,
                mark=*,
                ]
                coordinates {
                    (0.1, 1.59595214419913) (0.08333333333333333, 1.3036674484933357) (0.06666666666666667, 1.0646077449921765) (0.05, 0.8060390121813185) (0.04, 0.652857556090589) (0.03225806451612903, 0.5323866738826363) (0.025, 0.42153607871481247) (0.02, 0.34473720668017227) (0.015873015873015872, 0.2814823954515597) (0.0125, 0.22959268788436457) (0.01, 0.19094830066488644)
                }; 
    

            \addplot[color=black, mark=none] 
            coordinates {(0.07, 0.8) (0.035, 0.4)};
            \node [orange, rotate=32] at (rel axis cs: 0.25,0.41) { $\sim h^3$ };
            \node [brown, rotate=0] at (rel axis cs: 0.1,0.27) { $\sim h^5$ };
            \node [blue, rotate=38] at (rel axis cs: 0.35, 0.32) { reference };
            \node [red, rotate=0] at (rel axis cs: 0.1,0.21) { $\sim h^4$ };
            \node [black, rotate=38] at (rel axis cs: 0.7,0.55) { slope 1 };

            \end{axis}
            \end{tikzpicture}
    \caption{Convergence of the error in the $L^2$ norm (left) and $H^1$ semi-norm (right) in 3D.}
    \label{fig:conv_3d} 
\end{figure}
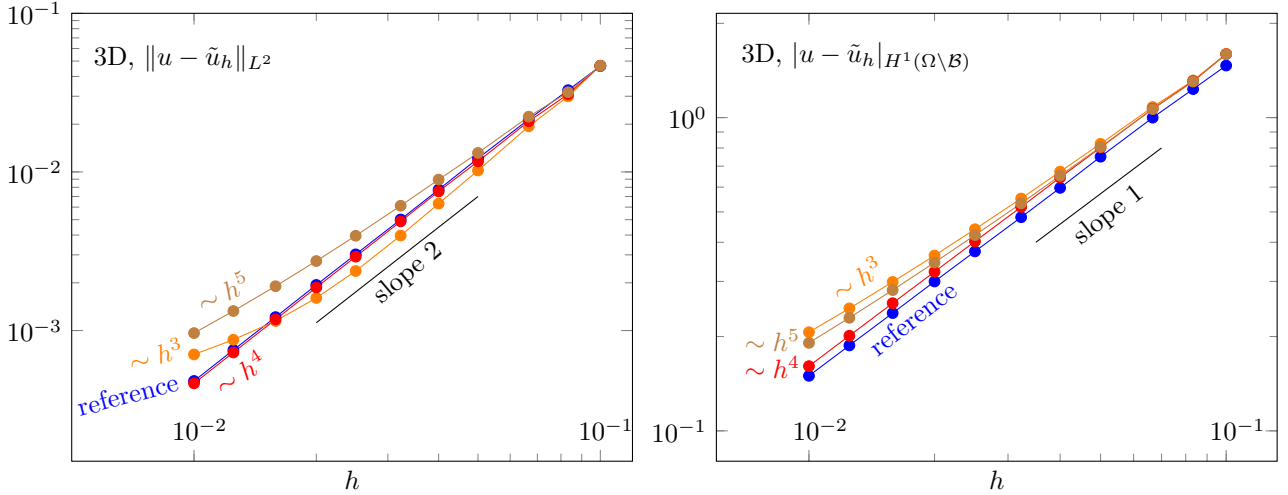\\
In 3D, where an analogue of a band of caps is a thin slab or `surface' of degenerate tetrahedra (spades and slivers in our case), the proper rate of convergence  is obtained for the $L^2$ and $H^1$ error with the choice of $k=4$. The choice of a higher exponent doesn't bring a too large error but will eventually create a locking effect of the solution as $h$ is reduced.
It may come as a surprise that the error of the TFEM method is lower than that of the classical FEM on a regular mesh, while the latter gives the best approximation of the solution in the discrete space $V_h$. This is because it is the best approximation in the energy norm of the problem, which is in our case the $H^1$ semi-norm. In the $H^1$ semi-norm, the finite element error on a regular mesh is lower, but this is not the case for the $L^2$ norm as we can see in the Figure \ref{fig:conv_2D}.
While the finite element method minimizes the error for the $H^1$ semi-norm, the tempered FEM cannot offer a better solution than the classical FEM, as observed in Figures \ref{fig:conv_2D} and \ref{fig:conv_3d}. Nevertheless, in 2D, the errors are almost identical, and in 3D, the TFEM errors are quite close to those of the classical FEM.

\subsection{Error sensitivity and  choice of \texorpdfstring{$\Jmin$}{}} \label{sensitivity}
To assess the sensitivity of the solution to the choice of $\Jmin$, we have conducted the same test case across a broad spectrum of C values using the formula $\Jmin = Ch^{d+1}$, $d=2,3$ being the problem dimension.
The $H^1$ error is less sensitive than the $L^2$ error, as illustrated in Figure \ref{fig:coeff_3D} and \ref{fig:coeff_2D}. This explains why, in the convergence graphs (Figures \ref{fig:conv_2D} and \ref{fig:conv_3d}), the convergence rate deteriorates sooner for the $L^2$ error compared to the $H^1$ error. The two dashed lines represent the error for the FEM on a regular mesh. For a specific range of $C$ values, the $L^2$ error of the TFEM is lower than that for the FEM. This range, spanning almost an order of magnitude in 3D, indicates a degree of robustness in the method regarding the choice of $C$.
\input{coeff_2D.tex}\\
The results in 2D and 3D are similar, with the notable exception that the \(H^1\) error is even less sensitive to the choice of \(J_{\text{min}}\)  as shown in Figure \ref{fig:coeff_2D}. Indeed, the error measured for the gradients exhibits a plateau, remaining nearly constant over a wide range of $C$ values. As the mesh is refined, this plateau seems to grow at a rate of $\sim h^{-1}$, explaining the range of exponent values for $h$ that allow gradient convergence. Choosing $\Jmin = Ch^3$ with a constant $C$ is effectively equivalent to choosing $\Jmin = Ch^2$ with a $C$ that varies as $\sim h$. The dotted line in Figure \ref{fig:coeff_2D} illustrates this scenario, showing it remains well within the plateau, ensuring the error converges at the expected rate.
\begin{figure}[!ht]
    \hspace{-0.75cm}
    \begin{tikzpicture}
        \begin{axis}[
            xlabel={$C$},
            ylabel={3D, $\|u - \tilde{u}_h\|_{L^2}$},
            xtick={1e-6, 1e-4, 1e-2, 1},
            ytick={1e-3, 1e-2, 1e-1},
            legend pos=north west,
            ymajorgrids=true,
            grid=none,
            xmin=5e-5,  xmax=200,
            ymin = 7e-4,
            ymax = 1,
            xmode=log,
            ymode=log,
            width=0.56\textwidth,
            height=0.47\textwidth,
            x label style={at={(axis description cs:0.5,0)},anchor=north},
            y label style={at={(axis description cs:0.2,0.85)},rotate=-90,anchor=south},
            xticklabel style={above, yshift=1ex, xshift=1.25ex},
        ]
        
        \addplot[
            color=blue,
            mark=*,
            mark options={scale=0.6},
            line width=1.pt
            ]
            coordinates {
                (0.0001000000, 0.3095094205) (0.0003162278, 0.2713403482) (0.0010000000, 0.1988105012) (0.0031622777, 0.1178097704) (0.0100000000, 0.0650798152) (0.0316227766, 0.0361357008) (0.1000000000, 0.0208831697) (0.3162277660, 0.0144928350) (1.0000000000, 0.0116527365) (3.1622776602, 0.0094734532) (10.0000000000, 0.0187749512) 
            }; 
            
        \addplot[
            color=orange,
            mark=*,
            mark options={scale=0.6},
            line width=1.pt
            ]
            coordinates {
                (0.0001000000, 0.2957523808) (0.0003162278, 0.2449711597) (0.0010000000, 0.1639255279) (0.0031622777, 0.0912050792) (0.0100000000, 0.0490240171) (0.0316227766, 0.0256047887) (0.1000000000, 0.0139544869) (0.3162277660, 0.0094361565) (1.0000000000, 0.0075337394) (3.1622776602, 0.0061018106) (10.0000000000, 0.0117126337) 
            }; 
        
        \addplot[
            color=red,
            mark=*,
            mark options={scale=0.6},
            line width=1.pt
            ]
            coordinates {
                (0.0001000000, 0.2794835541) (0.0003162278, 0.2157572905) (0.0010000000, 0.1322409756) (0.0031622777, 0.0707212660) (0.0100000000, 0.0368006632) (0.0316227766, 0.0179996954) (0.1000000000, 0.0093116505) (0.3162277660, 0.0061501840) (1.0000000000, 0.0048785609) (3.1622776602, 0.0039635924) (10.0000000000, 0.0075405178) 
            }; 

        \addplot[
            color=brown,
            mark=*,
            mark options={scale=0.6},
            line width=1.pt
            ]
            coordinates {
                (0.0001000000, 0.2549345095) (0.0003162278, 0.1780059500) (0.0010000000, 0.0999368014) (0.0031622777, 0.0520582627) (0.0100000000, 0.0256603701) (0.0316227766, 0.0116425859) (0.1000000000, 0.0057325934) (0.3162277660, 0.0037056295) (1.0000000000, 0.0029216546) (3.1622776602, 0.0023726302) (10.0000000000, 0.0044774369) 
            }; 

        \addplot[
            color=blue,
            mark=none,
            dashed
            ]
            coordinates {
                (0.0001000000, 0.0121773052) (10, 0.0121773052)
            };
        \addplot[
            color=brown,
            mark=none,
            dashed
            ]
            coordinates {
                (0.0001000000, 0.0030221421) (10, 0.0030221421)
            };

        \node [blue, rotate=0] at (rel axis cs: 0.9, 0.48) { $h=\frac{1}{10}$ };
        \node [orange, rotate=0] at (rel axis cs: 0.9,0.40) { $h=\frac{1}{15}$ };
        \node [red, rotate=0] at (rel axis cs: 0.9,0.32) { $h=\frac{1}{20}$ };
        \node [brown, rotate=0] at (rel axis cs: 0.9,0.24) { $h=\frac{1}{25}$ };

        \node [blue, align=center] at (rel axis cs: 0.15,0.47) { reference \\ $h=\frac{1}{10}$ };
        \node [brown] at (rel axis cs: 0.22,0.25) { reference $h = \frac{1}{25}$ };


        \end{axis}
        \end{tikzpicture}%
        \begin{tikzpicture}
            \begin{axis}[
                xlabel={$C$},
                ylabel={3D, $|u - \tilde{u}_h|_{H^1(\Omega \setminus \B)}$},
                xtick={1, 0.01, 0.0001, 1e-06},
                ytick={0.01, 0.1, 1, 10},
                xmin=5e-5,  xmax=200,
                ymin=0.25,
                ymax = 5.5,
                legend pos=north west,
                ymajorgrids=true,
                grid=none,
                xmode=log,
                ymode=log,
                width=0.56\textwidth,
                height=0.47\textwidth,
                x label style={at={(axis description cs:0.5,-0.0)},anchor=north},
                y label style={at={(axis description cs:0.25,0.85)},rotate=-90,anchor=south},
                xticklabel style={above, yshift=1ex, xshift=1ex},
            ]
            
            \addplot[
                color=blue,
                mark=*,
                mark options={scale=0.6},
                line width=1.pt
                ]
                coordinates {
                    (0.0001000000, 3.2865509926) (0.0003162278, 3.0900254557) (0.0010000000, 2.6697956857) (0.0031622777, 2.0742797428) (0.0100000000, 1.5300362552) (0.0316227766, 1.1484518535) (0.1000000000, 0.9231073031) (0.3162277660, 0.8219058938) (1.0000000000, 0.8025508886) (3.1622776602, 0.8593950703) (10.0000000000, 1.0757805021) 
                }; 
                
            \addplot[
                color=orange,
                mark=*,
                mark options={scale=0.6},
                line width=1.pt
                ]
                coordinates {
                    (0.0001000000, 3.1952200753) (0.0003162278, 2.9223229361) (0.0010000000, 2.4110123993) (0.0031622777, 1.7961620992) (0.0100000000, 1.2902215978) (0.0316227766, 0.9450455092) (0.1000000000, 0.7469141039) (0.3162277660, 0.6609103942) (1.0000000000, 0.6440980491) (3.1622776602, 0.6877635898) (10.0000000000, 0.8494745354)
                }; 
            
            \addplot[
                color=red,
                mark=*,
                mark options={scale=0.6},
                line width=1.pt
                ]
                coordinates {
                    (0.0001000000, 3.0963607273) (0.0003162278, 2.7357584645) (0.0010000000, 2.1534598419) (0.0031622777, 1.5505249092) (0.0100000000, 1.0876607857) (0.0316227766, 0.7779154890) (0.1000000000, 0.6057156400) (0.3162277660, 0.5330442055) (1.0000000000, 0.5186544091) (3.1622776602, 0.5529982550) (10.0000000000, 0.6778724288) 
                }; 
    
            \addplot[
                color=brown,
                mark=*,
                mark options={scale=0.6},
                line width=1.pt
                ]
                coordinates {
                    (0.0001000000, 2.9524442035) (0.0003162278, 2.4810781540) (0.0010000000, 1.8529729778) (0.0031622777, 1.2925123917) (0.0100000000, 0.8823587010) (0.0316227766, 0.6152982514) (0.1000000000, 0.4727683084) (0.3162277660, 0.4142957341) (1.0000000000, 0.4026394736) (3.1622776602, 0.4286515803) (10.0000000000, 0.5219284287) 
                }; 

            \addplot[
                color=blue,
                mark=none,
                dashed
                ]
                coordinates {
                    (0.0001000000, 0.7503042081) (10, 0.7503042081)
                };
            \addplot[
                color=brown,
                mark=none,
                dashed
                ]
                coordinates {
                    (0.0001000000, 0.3739544150) (10, 0.3739544150)
                };

            \node [blue, rotate=0] at (rel axis cs: 0.9, 0.5) { $h=\frac{1}{10}$ };
            \node [orange, rotate=0] at (rel axis cs: 0.9,0.42) { $h=\frac{1}{15}$ };
            \node [red, rotate=0] at (rel axis cs: 0.9,0.31) { $h=\frac{1}{20}$ };
            \node [brown, rotate=0] at (rel axis cs: 0.9,0.23) { $h=\frac{1}{25}$ };

            \node [blue, align=center] at (rel axis cs: 0.15,0.34) { reference \\ $h=\frac{1}{10}$ };
            \node [brown] at (rel axis cs: 0.22,0.18) { reference $h = \frac{1}{25}$ };

    
            \end{axis}
            \end{tikzpicture}
            \caption{Effect of $C$ on the error in the $L^2$ norm (left) and the $H^1$ semi-norm (right) in 3D.}
            \label{fig:coeff_3D}
\end{figure}
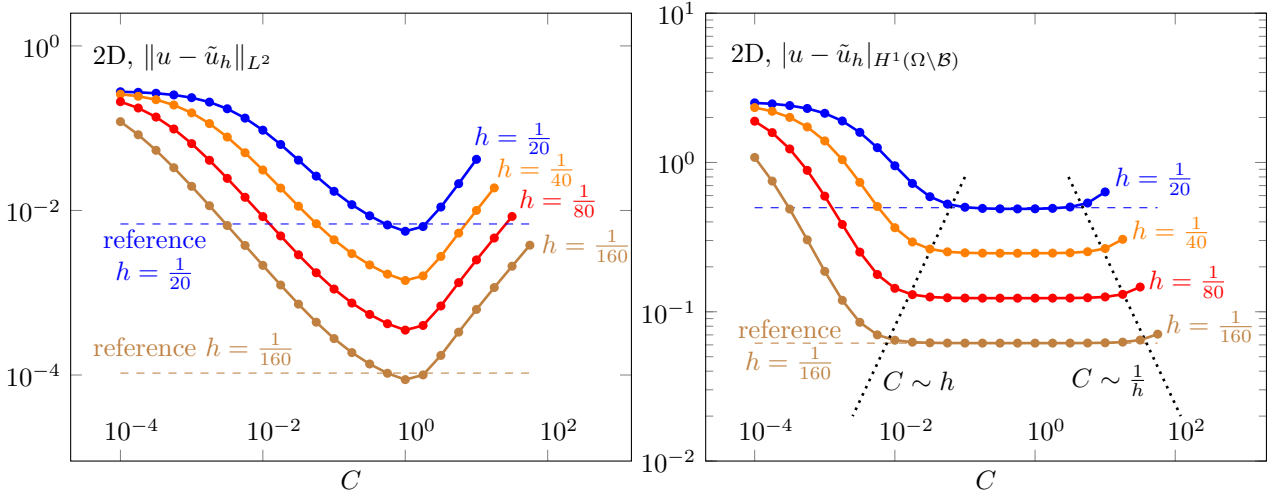
\\
\subsection{Size effect related to the degenerate band/surface} \label{sec:L_effect}
To verify the dependence on $L$, the length of the band of caps as represented in Figure \ref{fig:2D_band}, we employed a manufactured solution with several periods represented within the domain. By selecting lengths $L$ that are multiples of this period, we tested the dependence of the optimal choice of $C$ and the error on $L$ without altering the function represented on the band of caps. 

This strategy allows one to discriminate effects related to the first and second derivatives of the solution.\\
\input{L_effect.tex}
As we can see on the Figure \ref{fig:L_effect}, the length of the band in 2D does not influence the value of the optimal $C$ for both the $L^2$ and $H^1$ errors. This is in agreement with Theorem \ref{thm:FEM_modified_band_opt}. Additionally, if $C$ is chosen optimally, the error doesn't vary much with the band length. When $C$ is chosen too small and the solution presents locking, we observe that the error varies as $\sim \sqrt{L}$. Moreover, a shorter band results in a flatter profile of the curves around the minimum of the curve, indicating a plateau. Again, this plateau provides greater robustness in the choice of $C$, allowing for a wider range of acceptable values.\vspace{0.4cm}\\
\input{S_effect.tex}
In the 3D case, the area $S$ of the degenerate surface affects the choice of the optimal $C$, as illustrated in Figure \ref{fig:S_effect}. However, when $C$ is chosen optimally, the error exhibits minimal variation with respect to $S$.

\subsection{Conditioning} \label{sec:cond}
\begin{figure}[!ht]
    \hspace{-0.75cm}
\begin{tikzpicture}

\begin{axis}[
    width=10cm,
    height=5cm,
    thick,
    legend pos=north east,
    ymode=log,
    xlabel={Iteration},
    ylabel={Residual $\|r_k\|/ \|r_0\|$},
    grid=major,
    ymin=1e-11,
    ymax=1e1,
    xmin=-50,
    xmax=1500,
    width=0.75\textwidth,
    height=0.47\textwidth,
    y label style={at={(axis description cs:0.25,0.85)},rotate=-90,anchor=south},
]
\addplot[blue, thick] table[x index=0, y index=1] {data/CG_alone_D100.tex};
\addplot[orange, thick] table[x index=0, y index=1] {data/CG_alone_D1000.tex};
\addplot[red, thick] table[x index=0, y index=1] {data/CG_alone_D10000.tex};
\addplot[brown, thick] table[x index=0, y index=1] {data/CG_alone_D100000.tex};
\addplot[purple, thick] table[x index=0, y index=1] {data/CG_alone_D1000000.tex};

\node [blue, rotate=0] at (rel axis cs: 0.1, 0.15) { $D=10^{-2}$ };
\node [orange, rotate=0] at (rel axis cs: 0.45, 0.15) { $D=10^{-3}$ };
\node [red, rotate=0] at (rel axis cs: 0.7, 0.15) { $D=10^{-4}$ };
\node [purple, rotate=0] at (rel axis cs: 0.82, 0.45) { $D=10^{-5}$, $D=10^{-6}$};

\end{axis}
\node at (13,3) {\begin{TAB}(r,0.5cm,0.8cm)[5pt]{c|c|c}{c|c|c|c|c|c}
    $D$ & $\kappa$ & $H_1$ error \\ 
    $10^{-2}$ & 3096   & 12\%   \\ 
    $10^{-3}$ & 10027  & 12\%   \\ 
    $10^{-4}$ & 96255  & 13\%    \\ 
    $10^{-5}$ & 120200 & 14\%    \\ 
    $10^{-6}$ & 120201 & 14\%    \\
\end{TAB}};

\node at (7.5, 6.25) {Conjugate Gradient without preconditioner};

\end{tikzpicture}
\caption{This Figure shows iteration history of conjugate gradients to solve a problem with TFEM. The conditioning $\kappa$ varies as $1/D$ for $10^{-4} \leq D \leq 10^{-2}$ (no TFEM). TFEM is triggered for $D < 10^{-4}$ and conditioning remains stable. \label{fig:cond_no_prec}}
\end{figure}
\begin{figure}[!ht]
    \hspace{-0.75cm}
\begin{tikzpicture}

\begin{axis}[
    width=10cm,
    height=5cm,
    thick,
    legend pos=north east,
    ymode=log,
    xlabel={Iteration},
    ylabel={Residual $\|r_k\|/ \|r_0\|$},
    grid=major,
    ymin=1e-12,
    ymax=1e1,
    xmin=-5,
    xmax=125,
    width=0.75\textwidth,
    height=0.47\textwidth,
    y label style={at={(axis description cs:0.25,0.85)},rotate=-90,anchor=south},
]
\addplot[blue, thick] table[x index=0, y index=1] {data/CG_ic0_D100.tex};
\addplot[orange, thick] table[x index=0, y index=1] {data/CG_ic0_D1000.tex};
\addplot[red, thick] table[x index=0, y index=1] {data/CG_ic0_D10000.tex};
\addplot[brown, thick] table[x index=0, y index=1] {data/CG_ic0_D100000.tex};
\addplot[purple, thick] table[x index=0, y index=1] {data/CG_ic0_D1000000.tex};

\node [blue, rotate=0] at (rel axis cs: 0.62, 0.25) { $D=10^{-2}$ };
\node [orange, rotate=0] at (rel axis cs: 0.90, 0.07) { $D=10^{-3}$ };
\node [purple, rotate=0] at (rel axis cs: 0.85, 0.6) { $D=10^{-4}$,};
\node [purple, rotate=0] at (rel axis cs: 0.85, 0.5) { $D=10^{-5}$,};
\node [purple, rotate=0] at (rel axis cs: 0.85, 0.4) { $D=10^{-6}$\,};

\end{axis}
\node at (13,3) {\begin{TAB}(r,0.5cm,0.8cm)[5pt]{c|c|c}{c|c|c|c|c|c}
    $D$ & $\kappa$ & $H_1$ error \\ 
    $10^{-2}$ & 3096   & 12\%   \\ 
    $10^{-3}$ & 10027  & 12\%   \\ 
    $10^{-4}$ & 96255  & 13\%    \\ 
    $10^{-5}$ & 120200 & 14\%    \\ 
    $10^{-6}$ & 120201 & 14\%    \\
\end{TAB}};

\node at (7.5, 6.25) {Conjugate Gradient with Incomplete Cholesky preconditioner};

\end{tikzpicture}
\caption{This Figure shows the same 
    sequence of meshes with a decreasing band of cap width as Figure \ref{fig:cond_no_prec} but using preconditioned conjugate gradients. 
    There, TFEM has no influence on the iterative solver's convergence. \label{fig:cond_ic0}}
\end{figure}

It is perfectly legitimate to question the possibility of a deterioration in the performance of linear solvers when introducing elements of quasi-zero measure, which necessarily lead to the appearance of very large eigenvalues in $K$. More precisely, in the case of a second-order elliptic operator such as the Laplacian, one expects that the largest eigenvalue of $K$, and thus its conditioning $\kappa$, grows as 
\[
\kappa = \mathcal{O}(1/D),
\] 
where $D$ denotes the thickness of the band of caps. Figure~\ref{fig:cond_no_prec} illustrates the evolution of the residual during the iterations of an unpreconditioned Conjugate Gradient applied to the matrix $K$ with TFEM. In this particular case, TFEM is triggered for the two smallest values of $D$, i.e., $D=10^{-5}$ and $D=10^{-6}$. The number of iterations is approximately multiplied by $\sqrt{10}$ when decreasing $D$ from $10^{-2}$ to $10^{-3}$, and from $10^{-3}$ to $10^{-4}$, which corresponds to the theoretical worst-case scenario. On the other hand, a simple preconditioner of the incomplete Cholesky type, which is commonly used to solve this class of problems, proves to be perfectly sufficient to achieve convergence independent of $\kappa$. This is excellent news: the use of TFEM does not introduce any specific difficulty regarding the solution of linear systems.

\section{Theory behind the TFEM scheme} \label{section:theoretical}
In the previous section we showed that a simple numerical trick allowed us to recover the convergence rate of the finite element method on meshes with a band of caps. In this section we demonstrate this convergence for the choice of $\Jmin \sim h^{d+1}$ in $\mathbb{R}^d$.  To do so, we need to consider two cases: when the band of caps is not exactly degenerate but tends to a zero width, and the exactly degenerated case. For the first case, we do a classic finite element error analysis where we separate the error made on the band and the rest of the domain and we estimate and balance these carefully. In the second case for exactly degenerate bands of caps, we demonstrate that our scheme corresponds to an under-integrated mortar scheme with penalization of the jump across the band. In order to make this jump converge to zero as the mesh is refined, we obtain an expression the choice of $\Jmin$. 

In order to simplify the analysis, we consider the case of homogeneous Dirichlet boundary conditions ($u_D=0$). Considering a nonhomogeneous Dirichlet condition would only add a level of technicalities unrelated to the TFEM method itself (Dirichlet lift, error of the interpolation of $u_D$ on the boundary, etc.) which would make the analysis less clear and readable. In the notation of Section \ref{section:FEM} we therefore have $V_h=V_h^0=X_h\cap H^1_0(\Omega)$ in the following analysis.

\subsection{TFEM for a band of arbitrarily small but non-zero caps}
\label{section:theoretical1}

By applying the circumradius estimate of the Lagrange interpolation error and using the scaling of the band, we are able to bound the error $|u-\tu_h|_{H^1}$ in the $H^1$ semi-norm. This requires a careful decomposition and balancing of the error. The main difficulty is that the method (\ref{eq:FEM_mod}) is not consistent in the sense that we are essentially solving a modified continuous problem (with diffusion coefficient $d(x)$ instead of 1), hence we lack so-called Galerkin orthogonality. Nevertheless this obstacle can be overcome. The analysis requires three main working tools/lemmas.

\begin{lemma}[Circumradius estimate \cite{kobayashi2015circumradius}]
\label{lem:circum}
Let $K\subset\IR^2$ be an arbitrary triangle. Let $u\in \WIIi(K)$ and let $\Pi_K u$ be the linear Lagrange interpolation of $u$ on $K$. Then there exists a constant $C_c$ independent  of $u$ and $K$ such that
\begin{equation}
|u-\Pi_K u|_{\WIi(K)}\leq C_c R_K|u|_{\WIIi(K)} \leq C_c \frac{h^2}{\width}|u|_{\WIIi(K)}.
\label{lem:circum:est}
\end{equation}
where $R_K$ is the circumradius of the $K$ triangle.
\end{lemma}

\begin{lemma}[Maximum-angle condition \cite{babuvska1976angle}]
\label{lem:Max_angle_cond}
Let $K\subset\IR^2$ be a triangle satisfying the maximum angle condition: $\alpha_K\leq\alpha_0<\pi$ for some fixed $\alpha_0$. Let $u\in H^2(K)$ and let $\Pi_K u$ be the linear Lagrange interpolation of $u$ on $K$. Then there exists a constant $C_I$ independent of $u$ (but depending on the maximum angle $\alpha_0$) such that
\begin{equation}
|u-\Pi_K u|_{H^1(K)}\leq C_I h|u|_{H^2(K)}.
\label{lem:Max_angle_cond:est}
\end{equation}
\end{lemma}
\begin{lemma}[Young's inequality]
\label{lem:Young} Let $a,b\in\IR$ and $r>0$. Then
\begin{equation}
\label{eq:Young}
ab\leq \frac{a^2}{2r}+\frac{rb^2}{2}.
\end{equation}
\end{lemma}

Our main theoretical result is the TFEM error estimate \eqref{thm:FEM_modified:est} with a general value of $D$. We will then optimize this estimate with respect to $D$. 
\begin{theorem}
    \label{thm:FEM_modified}
    Let $u\in \WIIi(\Omega)$. Let $0<D\leq 1$ in \eqref{eq:def_D}. Then
    \begin{equation}
    \begin{split}
    \label{thm:FEM_modified:est}
    &|u-\tu_h|^2_{H^1(\Omega\setminus\B)} +D|u-\tu_h|^2_{H^1(\B)}\\
    &\quad\leq 4C_I^2h^2|u|^2_{H^2(\Omega\setminus\B)} + \frac{4}{D}|\B| |u|_{\WIi(\B)}^2 +6D|\B|C_c^2 \frac{h^4}{\width^2} |u|^2_{\WIIi(\B)}.
    \end{split}
    \end{equation}
    \end{theorem}
    \begin{proof}
    We subtract \eqref{eq:FEM_mod} and \eqref{eq:weak} with a test function $v=v_h\in V_h$:
    \begin{equation}
    \label{thm:FEM_modified:1}
    \int_\Omega\nabla u\cdotp\nabla v_h \dx- \int_\Omega d(x)\nabla \tu_h\cdotp\nabla v_h \dx =0.
    \end{equation}
    We define $\xi_h=\Pi_h u-\tu_h\in V_h$ and set $v_h=\xi_h$ in (\ref{thm:FEM_modified:1}):
    \begin{equation*}
    \label{thm:FEM_modified:2}
    \int_\Omega\Big(\nabla u -d(x)\nabla\Pi_h u +d(x)\nabla \Pi_h u -d(x)\nabla \tu_h\Big)\cdotp\nabla \xi_h \dx =0,
    \end{equation*}
    where we have  introduced the terms $\pm d(x)\nabla\Pi_h u$. Rearranging gives us
    \begin{equation*}
    \label{thm:FEM_modified:3}
    \int_\Omega d(x) |\nabla \xi_h|^2\dx = -\int_\Omega \Big(\nabla u -d(x)\nabla\Pi_h u\Big)\cdotp\nabla \xi_h \dx.
    \end{equation*}
    We split the integrals over $\B$ and $\Omega\setminus\B$ and take \eqref{eq:def_D} into account: $d(x) = 1$ over $\Omega \setminus \B$ and $d(x) = D$ over $\B$. This gives us:
    \begin{equation}
    \label{thm:FEM_modified:4}
    \begin{split}
    &\int_{\Omega\setminus\B} |\nabla \xi_h|^2\dx +D\int_{\B} |\nabla \xi_h|^2\dx\\
    &\quad= -\underbrace{\int_{\Omega\setminus\B} \Big(\nabla u -\nabla\Pi_h u\Big)\cdotp\nabla \xi_h \dx}_{(\star)} -\underbrace{\int_\B \Big(\nabla u -D\nabla\Pi_h u\Big)\cdotp\nabla \xi_h \dx}_{(\star\star)}.
    \end{split}
    \end{equation}
    Estimating $(\star)$ is standard, because we have `nice' elements on $\Omega\setminus\B$:
    \begin{equation}
    \begin{split}
    \label{thm:FEM_modified:5}
    \big|(\star)\big| &\leq |u-\Pi_h u|_{H^1(\Omega\setminus\B)}|\xi_h|_{H^1(\Omega\setminus\B)} \leq \tfrac{1}{2}|u-\Pi_h u|_{H^1(\Omega\setminus\B)}^2+\tfrac{1}{2}|\xi_h|_{H^1(\Omega\setminus\B)}^2\\
    &\leq \tfrac{1}{2}C_I^2h^2|u|_{H^2(\Omega\setminus\B)}^2 +\tfrac{1}{2}|\xi_h|_{H^1(\Omega\setminus\B)}^2,
    \end{split}
    \end{equation}
    where we have used Young's inequality \eqref{eq:Young} with $r=1$ and then the Maximum Angle Condition from lemma \ref{lem:circum}. We note that the last term in (\ref{thm:FEM_modified:5}) will be `hidden' under the corresponding left-hand side term of (\ref{thm:FEM_modified:4}).
    
    To estimate $(\star\star)$, we proceed more carefully since this integral is evaluated on the band. First, we have
    \begin{equation}
    \label{thm:FEM_modified:6}
    (\star\star) =\underbrace{\int_\B \Big(\nabla u -D\nabla u \Big)\cdotp\nabla \xi_h \dx}_{(A)} +\underbrace{\int_\B \Big(D\nabla u -D\nabla\Pi_h u\Big)\cdotp\nabla \xi_h \dx}_{(B)},
    \end{equation}
    where we have introduced $\pm D\nabla u$. Since $D<1$,  we can estimate $(A)$:
    \begin{equation}
    \label{thm:FEM_modified:7}
    \begin{split}
    \big|(A)\big|&= (1-D)\bigg|\int_\B \nabla u \cdotp\nabla \xi_h \dx\bigg|\leq \int_\B \big|\nabla u\big|\big|\nabla \xi_h \big|\dx\\ &\leq \frac{1}{D} \int_\B \big|\nabla u\big|^2\dx +\frac{D}{4} \int_\B \big|\nabla \xi_h \big|^2\dx\\ &\leq \frac{1}{D}|\B| |u|_{\WIi(\B)}^2 +\frac{D}{4} \int_\B \big|\nabla \xi_h \big|^2\dx
    \end{split}
    \end{equation}
    where we have used Young's inequality \eqref{eq:Young} with $r=\tfrac{D}{2}$ and taken the maximum of the first derivative on the band $|u|_{\WIi(\B)}$ to obtain an upper bound. Again, the last term in (\ref{thm:FEM_modified:7}) will be `hidden' under the corresponding left-hand side term of (\ref{thm:FEM_modified:4}).
    
    Finally, it remains to estimate term $(B)$ from (\ref{thm:FEM_modified:6}):
    \begin{equation}
    \begin{split}
    \label{thm:FEM_modified:8}
    \big|(B)\big| &\leq D\int_\B \big|\nabla u -\nabla\Pi_h u\big|\big|\nabla \xi_h\big| \dx\\ &\leq D\int_\B \big|\nabla u -\nabla\Pi_h u\big|^2\dx +\frac{D}{4}\int_\B\big|\nabla \xi_h\big|^2 \dx\\
    &\leq D|\B||u-\Pi_h u|^2_{\WIi(\B)} +\frac{D}{4}|\xi_h|_{H^1(\B)}^2\\ 
    &\leq D|\B|C_c^2 \bigg(\frac{h^2}{\width}\bigg)^2 |u|^2_{\WIIi(\B)} +\frac{D}{4}|\xi_h|_{H^1(\B)}^2
    \end{split}
    \end{equation}
    where we have used Young's inequality \eqref{eq:Young} with $r=\tfrac{1}{2}$ and the circumradius estimate (\ref{lem:circum:est}) in which the maximum of the second derivative on the band, $|u|_{\WIIi(\B)}$,  appears. As usual, the last term in (\ref{thm:FEM_modified:8}) will be `hidden' under the left-hand side of (\ref{thm:FEM_modified:4}).
    
    Combining estimates (\ref{thm:FEM_modified:4}), (\ref{thm:FEM_modified:5}), and (\ref{thm:FEM_modified:6})--(\ref{thm:FEM_modified:8}) gives us 
    \begin{equation}
    \label{thm:FEM_modified:9}
    \begin{split}
    &\frac{1}{2}|\xi_h|^2_{H^1(\Omega\setminus\B)} +\frac{D}{2}|\xi_h|^2_{H^1(\B)}\\
    &\quad\leq  \frac{1}{2}C_I^2h^2|u|_{H^2(\Omega\setminus\B)}^2 +\frac{1}{D}|\B| |u|_{\WIi(\B)}^2 +D|\B|C_c^2 \frac{h^4}{\width^2} |u|^2_{\WIIi(\B)}.
    \end{split}
    \end{equation}
    This gives us an estimate for $\xi_h = \Pi_h u - \tu_h$. Now we estimate the error $u-\tu_h$ using the fact that $u-\tu_h =(u-\Pi_h u)+\xi_h$:
    \begin{equation}
    \begin{split}
    \label{thm:FEM_modified:11}
    &|u-\tu_h|^2_{H^1(\Omega\setminus\B)} +D|u-\tu_h|^2_{H^1(\B)}\\ &\leq \big(|u-\Pi_h u|_{H^1(\Omega\setminus\B)} + |\xi_h|_{H^1(\Omega\setminus\B)}\big)^2 +D\big(|u-\Pi_h u|_{H^1(\B)} +|\xi_h|_{H^1(\B)}\big)^2\\
    &\leq 2|u-\Pi_h u|_{H^1(\Omega\setminus\B)}^2 + 2|\xi_h|_{H^1(\Omega\setminus\B)}^2 +2D|u-\Pi_h u|_{H^1(\B)}^2 +2D|\xi_h|_{H^1(\B)}^2,
    \end{split}
    \end{equation}
    where we have used the fact that $(a+b)^2\leq 2a^2+2b^2$. The $\xi_h$-terms in (\ref{thm:FEM_modified:11}) are estimated using (\ref{thm:FEM_modified:9}). The remaining terms have already been dealt with in (\ref{thm:FEM_modified:5}) and (\ref{thm:FEM_modified:8}), respectively:
    \begin{equation}
    \label{thm:FEM_modified:12}
    |u-\Pi_h u|_{H^1(\Omega\setminus\B)}^2\leq C_I^2h^2|u|^2_{H^2(\Omega\setminus\B)},
    \end{equation}
    since $\Omega\setminus\B$ contains only `nice' elements, and
    \begin{equation}
    \label{thm:FEM_modified:13}
    D|u-\Pi_h u|^2_{H^1(\B)} \leq D|\B||u-\Pi_h u|^2_{\WIi(\B)}\leq D|\B| C_c^2 \frac{h^4}{\width^2}|u|^2_{\WIIi(\B)},
    \end{equation}
    where we took the maximum of the gradient of $u - \Pi_h u$ over the band $\B$ and used the circumradius estimate from lemma \ref{lem:circum}.
    Using (\ref{thm:FEM_modified:9}), (\ref{thm:FEM_modified:12}), and (\ref{thm:FEM_modified:13}) in (\ref{thm:FEM_modified:11}) gives us the error estimate 
    \begin{equation*}
    \begin{split}
    \label{thm:FEM_modified:14}
    &|u-\tu_h|^2_{H^1(\Omega\setminus\B)} +D|u-\tu_h|^2_{H^1(\B)}\\ &\leq 4C_I^2h^2|u|^2_{H^2(\Omega\setminus\B)} + \frac{4}{D}|\B| |u|_{\WIi(\B)}^2 +6D|\B|C_c^2 \frac{h^4}{\width^2} |u|^2_{\WIIi(\B)},
    \end{split}
    \end{equation*}
    which is our desired estimate.
    \end{proof}
    
    \begin{remark}
    \label{rem:FEM}
        We note that for $D=1$ (i.e. standard FEM), the second right-hand side term in (\ref{thm:FEM_modified:est}) vanishes, since this comes from the consistency term $(A)$ which is zero in this case, cf. (\ref{thm:FEM_modified:7}). On the other hand, the last right-hand side term in (\ref{thm:FEM_modified:est}) is $O(h^2)$ if and only if $\width\ge Ch^2$ for some constant $C$. This is due to the fact that $|\B|\frac{h^4}{\width^2}=\width L\frac{h^4}{\width^2}=L\frac{h^2}{\width}h^2$. This means that if $\width\ge Ch^2$ then standard finite elements have optimal $O(h)$ convergence in the $H^1(\Omega)$-seminorm.
    \end{remark}
    
    Now we will optimize the value of $D$ in the case of a band of caps. The value of $D$ appears both in the numerator and the denominator of the error estimate of Theorem \ref{thm:FEM_modified}. The following lemma optimizes the value of such a function:
    
    \begin{lemma}
    \label{lem:opt}
    Let $a,b>0$ and
    \begin{equation}
    \label{lem:opt_g}
    g(x)=\frac{a}{x}+bx.
    \end{equation}
    Then $g$ attains its minimum at a unique $x_{\min}$, where
    \begin{equation}
    \label{lem:opt_eq}
    x_{\min}=\sqrt{\frac{a}{b}},\quad g(x_{\min})=2\sqrt{ab}.
    \end{equation}
    \end{lemma}
    \begin{proof}
    We have $g^\prime(x)=b-a/x^2$ which is zero for $x=\sqrt{a/b}$. Then $g(\sqrt{a/b})=a\sqrt{b/a}+b\sqrt{a/b}=2\sqrt{ab}$.
    \end{proof}
    
    We can now apply the result of this lemma to the error estimate of Theorem \ref{thm:FEM_modified}. This gives us an optimal choice for $D$: 

    \begin{theorem}
    \label{thm:FEM_modified_band_opt}
    Let $u\in \WIIi(\Omega)$. Let $\Th$ contain a band of caps of length $L\sim 1$ and width $\width$. In \eqref{eq:def_D}, set 
    \begin{equation}
    \label{thm:FEM_modified_band_opt:D}
    D=\min\left(\frac{\width}{h^2}\frac{|u|_{\WIi(\B)}}{|u|_{\WIIi(\B)}}\sqrt{\frac{2}{3C_c}} \, , \, \, 1\right),
    \end{equation}
    then the tempered FEM solution $\tu_h$ satisfies
    \begin{equation}
    \begin{split}
    \label{thm:FEM_modified_band_opt:est}
    &|u-\tu_h|_{H^1(\Omega\setminus\B)} +\sqrt{D}|u-\tu_h|_{H^1(\B)}\leq C(u,L)h,
    \end{split}
    \end{equation}
    where
    \begin{equation}
    \begin{split}
    \label{thm:FEM_modified_band_opt:est2}
    C(u,L)=2\sqrt{2}C_I|u|_{H^2(\Omega\setminus\B)} + 2\sqrt{2}\sqrt[4]{6}\sqrt{C_c L} \sqrt{|u|_{\WIi(\B)}|u|_{\WIIi(\B)}}
    \end{split}
    \end{equation}
    is a constant independent of $\width,h$.
    \end{theorem}
    \begin{proof}
    We optimize the last two terms in (\ref{thm:FEM_modified:est}) with respect to $D$. To this end, we denote them as
    \begin{equation*}
    \label{thm:FEM_modified_band_opt:1}
    g(D)= \frac{4}{D}|\B| |u|_{\WIi(\B)}^2 +6D|\B|C_c^2 \frac{h^4}{\width^2} |u|^2_{\WIIi(\B)},
    \end{equation*}
    which is a function of the form (\ref{lem:opt_g}) with respect to $D$. The application of Lemma \ref{lem:opt} gives us the optimal choice of $D$:
    \begin{equation*}
    \label{thm:FEM_modified_band_opt:2}
    D_\opt=\sqrt{\frac{4|\B| |u|_{\WIi(\B)}^2}{6|\B|C_c^2 \frac{h^4}{\width^2} |u|^2_{\WIIi(\B)}}} =\frac{\width}{h^2}\frac{|u|_{\WIi(\B)}}{|u|_{\WIIi(\B)}}\sqrt{\frac{2}{3C_c}},
    \end{equation*}
    which is (\ref{thm:FEM_modified_band_opt:D}). Moreover,
    \begin{equation*}
    \begin{split}
    \label{thm:FEM_modified_band_opt:3}
    g(D_\opt)&= 2\sqrt{4|\B| |u|_{\WIi(\B)}^26|\B|C_c^2 \frac{h^4}{\width^2} |u|^2_{\WIIi(\B)}}\\ &=4\sqrt{6}C_c|\B|\frac{h^2}{\width} |u|_{\WIi(\B)}|u|_{\WIIi(\B)}\\
    &\leq 4\sqrt{6}C_c L h^2 |u|_{\WIi(\B)}|u|_{\WIIi(\B)},
    \end{split}
    \end{equation*}
    since $|\B|\leq \width L$ (inequality here since $\B$ is not a rectangle). With the choice $D=D_\opt$, estimate (\ref{thm:FEM_modified:est}) therefore becomes
    \begin{equation}
    \begin{split}
    \label{thm:FEM_modified_band_opt:4}
    &|u-\tu_h|^2_{H^1(\Omega\setminus\B)} +D|u-\tu_h|^2_{H^1(\B)}\\
    &\quad\leq 4C_I^2h^2|u|^2_{H^2(\Omega\setminus\B)} + 4\sqrt{6}C_c L h^2 |u|_{\WIi(\B)}|u|_{\WIIi(\B)}.
    \end{split}
    \end{equation}
    Now we wish to take the square root of (\ref{thm:FEM_modified_band_opt:4}). This is done by considering the equivalence of norms in $\IR^2$, which for $a,b,c,d\ge 0$ gives us 
    \begin{equation}
    \label{thm:FEM_modified_band_opt:5}
    a^2+b^2\leq c^2+d^2 \quad\Longrightarrow\quad a+b\leq\sqrt{2}(c+d).
    \end{equation}
    Applying this to (\ref{thm:FEM_modified_band_opt:4}) gives
    \begin{equation*}
    \begin{split}
    \label{thm:FEM_modified_band_opt:6}
    &|u-\tu_h|_{H^1(\Omega\setminus\B)} +\sqrt{D}|u-\tu_h|_{H^1(\B)}\\
    &\quad\leq 2\sqrt{2}C_I h|u|_{H^2(\Omega\setminus\B)} + 2\sqrt{2}\sqrt[4]{6}\sqrt{C_c L} h \sqrt{|u|_{\WIi(\B)}|u|_{\WIIi(\B)}}.
    \end{split}
    \end{equation*}
    This is the desired estimate (\ref{thm:FEM_modified_band_opt:est}) with the constant $C(u,L)$ defined by (\ref{thm:FEM_modified_band_opt:est2}). Since $D$ needs to be smaller than 1, we take the minimum between $D_\opt$ and $1$.
    \end{proof} \noindent
    
    We note that this choice of $D_\opt$ can be transformed back to a choice of $\Jmin$ by using \eqref{eq:Ktilde_K} and the definition of $d(x)$ from \eqref{eq:def_D}:
    \begin{align} \nonumber
        J_{\text{min,opt}} &= \frac{\J}{D_{\text{opt}}} \\ \nonumber
              &= \frac{h \width}{2} \, \frac{h^2}{\width} \, \frac{|u|_{\WIIi(\B)}}{|u|_{\WIi(\B)}} \, \sqrt{\frac{3C_c}{2}}\\\label{eq:Jminopt}
              &= h^3 \frac{|u|_{\WIIi(\B)}}{|u|_{\WIi(\B)}} \, \sqrt{\frac{3C_c}{8}}
    \end{align}
Thus we have proved that the error converges in 2D when we choose $\Jmin \sim h^3$ as we can see in the convergences plots of the numerical experiment in Figure \ref{fig:conv_2D}.

Contrary to what one might think, the choice of $\Jmin$ is independent of the length $L$. This corresponds well to what was observed during the numerical experiments in section \ref{sec:L_effect}. The error estimate provided by the theorem \ref{thm:FEM_modified_band_opt} predicts an error that varies as $ \sim \sqrt{L}$ when $\Jmin$ is chosen optimally. In the numerical experiment we observed this dependence in $\sim \sqrt{L}$ when $\Jmin$ was chosen too small (locking), as we can see in figure \ref{fig:L_effect}.

\noindent\textbf{Remarks:}
\begin{itemize}
    \item The expression for $D$ in (\ref{thm:FEM_modified_band_opt:D}) guarantees that  $D\leq 1$. One might ask what exactly happens when the expression in the first argument of the minimum is greater than 1. In this case $\width\ge Ch^2$, which means that standard finite elements (i.e. $D=1$) have optimal $O(h)$ convergence rate (cf. Remark \ref{rem:FEM}), and therefore no tempering is necessary in the finite element method, i.e. $D=1$ is the best choice in this case. This justifies the fact that the diffusion coefficient is always decreased in TFEM, never increased.
    \item We note that the very regular structure of elements in the band $\B$ as presented in Figure \ref{fig:2D_band} is not needed in the analysis, the band is only assumed to be somehow composed of `degenerating' elements.
    \item There need not be a singe band of caps, there can be several layers of these bands next to each other or several separate bands of caps in the domain $\Omega$. The only condition for the analysis to remain the same is that their number does not grow as $h\to 0$.
    \item Poincar\'{e}'s inequality immediately gives us the estimate
    \begin{equation}
    \label{est:Poincare}
    |u-\tu_h|_{L^2(\Omega)}\leq C(u,L)h
    \end{equation}
    for some constant $C(u,L)$ independent of $h$. This can be easily seen by applying Poincar\'{e}'s inequality to (\ref{thm:FEM_modified_band_opt:est}) on $\Omega\setminus\B$ and a simple continuity argument on $\B$. This $L^2$-error estimate is subomptimal. However, we were unable to prove optimal $L^2$-error estimates, as the standard Aubin-Nitsche duality technique runs into technical obstacles. Nevertheless numerical experiments indicate that the method has optimal convergence rates in $L^2(\Omega)$. 
    \item The constant $C_c$ from Lemma \ref{lem:circum} is very close to 1. The paper \cite{Kucera_circumradius} gives the value $2\sqrt{2}+2$, which can most likely be improved. In the end, the factor $\sqrt{\frac{2}{3C_c}}$ from (\ref{thm:FEM_modified_band_opt:D}) can simply be neglected for simplicity.
\end{itemize}

\subsubsection*{Extension to 3D}
The results of Theorem \ref{thm:FEM_modified_band_opt} can be straightforwardly extended to 3D. In this case, the thin band $\B$ with width $\bar{h}$ and length $L$ becomes a thin slab $\mathcal{S}$ of width $\bar{h}$ and area $\mathcal{A}$ of its base. We assume that the slab 
is partitioned into degenerating tetrahedra. The situation in 3D is much more complicated than 2D, as there are 9 basic types of degenerate tetrahedra, cf. \cite{Krizek}. However what they have in common is that at least one of their four altitudes is always small (similarly to one of the altitudes of a degenerate triangle). We assume the partition of $S$ into tetrahedra is such that the minimal altitude of each tetrahedron is $\bar{h}$ and its longest edge is $h$. Then the proof of Theorem \ref{thm:FEM_modified_band_opt} can be trivially extended to 3D provided we have an estimate for the error of linear Lagrange interpolation on a tetrahedron $T$ of the same form as in Lemma \ref{lem:circum}. We then obtain an $O(h)$ estimate of the TFEM error in the $H^1(\Omega\setminus \mathcal{S})$ seminorm provided $D\sim \frac{\bar{h}}{h^2}$, exactly as in Theorem \ref{thm:FEM_modified_band_opt}. We note that in 3D, this choice corresponds to $\Jmin\sim h^4$, which agrees with the numerical experiments from Figure \ref{fig:conv_3d}.

The necessary interpolation error estimate analogous to Lemma \ref{lem:circum} can be easily obtained from the estimate of Baidakova, \cite{Baidakova}, as in the following lemma.

\begin{lemma}[Interpolation error estimate in 3D]
\label{lem:circum_3D}
Let $T\subset\IR^3$ be an arbitrary tetrahedron with longest edge length $h$ and minimal altitude $\bar{h}$. Let $u\in \WIIi(T)$ and let $\Pi_T u$ be the linear Lagrange interpolation of $u$ on $T$. Then there exists a constant $C_c$ independent of $u$ and $T$ such that
\begin{equation}
|u-\Pi_T u|_{\WIi(T)}\leq C_c \frac{h^2}{\bar{h}}|u|_{\WIIi(T)}.
\label{lem:circum_3D:est}
\end{equation}
\end{lemma}
\begin{proof}
Denote the faces of $T$ as $T_i, i=1,\ldots, 4$. Let $R_i$ be the radius of the
circle circumscribed about $T_i$, and let $\varphi_{ij}$ be the angle between faces $T_i, T_j$. Let $R=\max_i R_i$ and $\sin \varphi=\max_{ij}\sin\varphi_{ij}$. Then in \cite{Baidakova} it is proven that for some constant $C$,
\begin{equation}
\label{lem:circum_3D:1}
|u-\Pi_T u|_{\WIi(T)}\leq C \frac{R}{\sin\varphi}|u|_{\WIIi(T)}.
\end{equation}
Now we estimate the factor $\frac{R}{\sin\varphi}$. Let $R=R_i$ for some $i$. Then the corresponding face $T_i$ is a triangle with vertices $A,B,C$. Assume that $AB$ is the longest edge. The law of sines then states that $2R_i = |AC||AB|/h_c$, where $h_c$ is the altitude in ABC corresponding to vertex $C$. Moreover, let $\bar{h}_c$ be the height of the tetrahedron $T$ corresponding to the apex $C$. Then trivially, $\sin \varphi_{ij} =\bar{h}_c/h_c$, combining these expressions, we get
\begin{equation}
\label{lem:circum_3D:2}
\frac{R}{\sin\varphi}\leq \frac{R}{\sin\varphi_{ij}}= \frac{ |AC||AB|/(2h_c)}{\bar{h}_c/h_c} = \frac{ |AC||AB|}{2\bar{h}_c} \leq \frac{h^2}{2\bar{h}},
\end{equation}
since $\bar{h}\leq \bar{h}_c$. Substituting (\ref{lem:circum_3D:2}) into  (\ref{lem:circum_3D:1}) gives the desired estimate.
\end{proof}

\subsection{Zero-measure elements -- TFEM and Mortaring}
\label{section:zero-measure}
In \S\ref{section:theoretical1}, we proved a sufficient condition for the convergence of the tempered finite element method for any band of degenerate elements with nonzero width. It is important to realize that this proof 
is not valid if the band width is exactly zero. 
It is however valid 
 when the band degenerates to zero width whatever the rate. 
 In order to generalize the proof of section 3.1  to the case where $\width = 0$ (truly or up to machine precision), we show that our approach is equivalent to a mortar method penalizing the jump of the solution on the band.\\ 
\begin{figure}[!ht]
    \centering
    \includegraphics[width=0.9\linewidth]{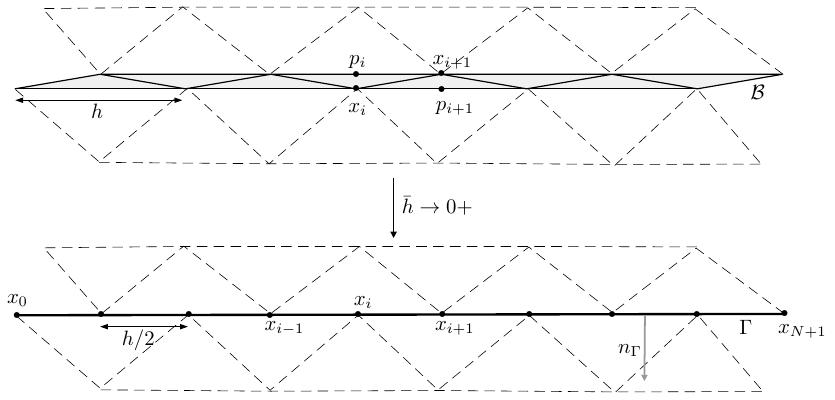}
    \caption{Degeneration of a band of caps $\mathcal{B}$ to a line $\Gamma$ as the width of the band $\width$ goes to zero.}
    \label{fig:Band_zero_measure}
\end{figure}\\
Again we assume a regular structure of the band of caps $\B$ with length $L$ and width $\width$. The goal of this section is to formulate the limiting scheme of \eqref{eq:FEM_mod} for $\width =0$, which we will interpret as the limit for $\width\to 0^+$. 

As $\width\to 0^+$, the band $\B$ degenerates to a line $\Gamma$ and the resulting triangulation $\Th^\Gamma$ is not conforming on $\Gamma$, cf. Figure \ref{fig:Band_zero_measure}. We will also call $\Gamma$ a zero-measure band. We will label the elements in the band as $K_i, i=1,\ldots,N$, ordered from one end of $\B$ to the other. For an element $K_i\in\B$, let $x_i$ denote its maximum angle vertex and $p_i$ be the foot of the altitude corresponding to $x_i$, cf. Figure \ref{fig:Band_zero_measure}. We also label the endpoints of $\Gamma$ as $x_0$ and $x_{N+1}$,respectively. As we send $\width\to 0^+$ and $\B$ degenerates to a line $\Gamma$, each $x_i$ merges with its corresponding $p_i$. We note that each $x_i$ is the vertex of a triangle adjacent to $\B$. Let $n_\Gamma$ be the unit normal to $\Gamma$ (arbitrarily oriented). We can then define the two one-sided limits of a function $f:\Omega\setminus\Gamma\to\IR$  at $x\in\Gamma$:
\begin{equation}
f^+(x)=\lim_{t\to 0^+}f(x+t n_\Gamma),\quad f^-(x)=\lim_{t\to 0^-}f(x+t n_\Gamma)
\end{equation} 
and the jump
\begin{equation}
[f]_x = f^-(x)-f^+(x),
\end{equation} 
where we will sometimes omit the subscript $x$ for simplicity. 

On $\B$ we will need to split the gradient of a function into two parts: $\gradn u$ will be the component of $\nabla u$ in the direction perpendicular to $\Gamma$, i.e. $\gradn u=(\nabla u\cdotp n_\Gamma)n_\Gamma$, while $\gradt u$ denotes the component of $\nabla u$ in the direction parallel to $\Gamma$. Using this notation, we can rewrite the scheme \eqref{eq:FEM_mod}, i.e.
\begin{equation}
\label{eq:FEM_zero:1}
\int_{\Omega\setminus\B} \nabla \tu_h\cdotp\nabla v_h \dx +D\int_{\B} \nabla \tu_h\cdotp\nabla v_h \dx=(f,v_h),\quad\forall v_h\in V_h.
\end{equation}
On $\B$ we write
\begin{equation*}
\label{eq:FEM_zero:2}
\nabla \tu_h\cdotp\nabla v_h =(\gradn \tu_h+\gradt \tu_h)\cdot(\gradn v_h +\gradt v_h) = \gradn \tu_h \cdotp\gradn v_h +\gradt \tu_h\cdotp\gradt v_h,
\end{equation*}
since $\gradn \tu_h\cdotp\gradt v_h =\gradt \tu_h\cdotp\gradn v_h =0$ by definition. We can therefore write \eqref{eq:FEM_zero:1} as
\begin{equation}
\label{eq:FEM_zero:3}
\int_{\Omega\setminus\B} \nabla \tu_h\cdotp\nabla v_h \dx +D\int_{\B} \gradn \tu_h \cdotp\gradn v_h +\gradt \tu_h\cdotp\gradt v_h \dx=(f,v_h).
\end{equation}
Now we derive the form of \eqref{eq:FEM_zero:3} for $\width=0$ by sending $\width\to 0^+$. Eventually, we will have $D\sim\width/h^2$, where the numerator $\width$ will play an important role in the derivation. For this reason, we write $D(\width, h)=\width\tD(h)$ and it will be important to choose $\tD$ appropriately in the scheme.

\begin{lemma}
\label{lem:FEM_zero}
Let $D=\width\tD$. Then the tempered FEM scheme \eqref{eq:FEM_mod} for $\width=0$ is equivalent to the formulation
\begin{equation}
\label{eq:FEM_zero}
\int_{\Omega\setminus\Gamma} \nabla \tu_h\cdotp\nabla v_h \dx +\tD\frac{h}{2}\sum_{x_i\in\Gamma}[\tu_h]_{x_i}[v_h]_{x_i}=(f,v_h),\quad\forall v_h\in V_h^\Gamma,
\end{equation}
where
\begin{equation*}
V_h^\Gamma = \{v_h\in C(\Omega\setminus\Gamma):\ v_h|_K\in P^1(K)\,\forall K\in\Th, v_h|_{\partial\Omega}=0\}
\end{equation*}
is the set of globally continuous piecewise linear functions discontinuous on $\Gamma$.
\end{lemma}

\begin{remark}
\label{rem:quadrature}
The scheme \eqref{eq:FEM_zero} is similar to the nonconsistent penalization method of Babu\v{s}ka, cf. \cite{Babuska:1973:FEM},
\begin{equation}
\label{eq:FEM_penalization}
\int_{\Omega\setminus\Gamma} \nabla u_h\cdotp\nabla v_h \dx +\tD\int_\Gamma[u_h][v_h]\dS =(f,v_h),\quad\forall v_h\in V_h^\Gamma,
\end{equation}
due to the fact that
\begin{equation}
\label{eq:FEM_penalization_quadrature}
\int_\Gamma[u_h][v_h]\dS\approx \frac{h}{2}\sum_{x_i\in\Gamma}[u_h]_{x_i}[v_h]_{x_i}
\end{equation}
is a \emph{composite trapezoidal quadrature} approximation on the partition given by the nodes $x_i\in\Gamma$, which has step size $h/2$. We note that the endpoints $x_0,x_{N+1}$ of $\Gamma$ do not come into play, since  $[v_h]_{x_0}=[v_h]_{x_{N+1}}=0$ for all $v_h\in V_h^\Gamma$. Also note that \eqref{eq:FEM_penalization_quadrature} is an approximation, not an equality, since $[u_h][v_h]$ is a piecewise quadratic function and composite trapezoidal quadratures integrate exactly only piecewise linear functions.
\end{remark}

\begin{proof}[Proof of Lemma \ref{lem:FEM_zero}]
Consider the scheme \eqref{eq:FEM_mod} in the form \eqref{eq:FEM_zero:3}. Let $\varphi_i$, $i=1,\ldots,M$ be the standard `tent' basis functions of $V_h$ corresponding to individual vertices of $\Th$. We write $\tu_h(x)=\sum_{j=1}^M U_j\varphi_j(x)$, where $U_j\in \IR$, and set $v_h:=\varphi_k$. Then, \eqref{eq:FEM_zero:3} becomes
\begin{equation}
\label{lem:FEM_zero:1}
\sum_{j=1}^M U_j\int_{\Omega\setminus\B} \nabla \varphi_j\cdotp\nabla \varphi_k \dx +\sum_{j=1}^M U_j\width\tD\int_{\B} \gradn \varphi_j \cdotp\gradn \varphi_k +\gradt \varphi_j\cdotp\gradt \varphi_k \dx=(f,\varphi_k).
\end{equation}
Now consider that for all $j$, it holds $|\gradt \varphi_j|\leq1/h$ on each $K\in\B$ and that $|\B|\leq L\width$. Therefore, for fixed $h$,
\begin{equation*}
\label{lem:FEM_zero:2}
\Big|\width\tD\int_{\B} \gradt \varphi_j\cdotp\gradt \varphi_k \dx\Big| \leq \width^2 L\tD h^{-2} \to 0,\text{ as }\width\to 0^+,
\end{equation*}
therefore these terms `disappear' from the limiting system of linear equation (\ref{lem:FEM_zero:1}) for $\width=0$.

On the other hand, on an element $K_i\in\B$ with apex $x_i$ we have
\begin{equation*}
\gradn \varphi_j=\pm\frac{\varphi_j(x_i)-\varphi_j(p_i)}{\width}n_\Gamma,
\end{equation*}
for all $j$ (the sign depends on the orientation of $K$ with respect to $n_\Gamma$). Therefore (since $|K|=\tfrac{1}{2}\width h$ for all $K\in\B$),
\begin{equation}
\nonumber
\label{lem:FEM_zero:3}
\begin{split}
&\width\tD\int_{\B} \gradn \varphi_j \cdotp\gradn \varphi_k \dx= \width\tD\sum_{K_i\in\B}\frac{1}{2}\width h \frac{\varphi_j(x_i)-\varphi_j(p_i)}{\width}
\frac{\varphi_k(x_i)-\varphi_k(p_i)}{\width}\\
&\ = \tD\frac{h}{2}\sum_{K_i\in\B} \big(\varphi_j(x_i)-\varphi_j(p_i)\big) \big(\varphi_k(x_i)-\varphi_k(p_i)\big)\longrightarrow \tD\frac{h}{2}\sum_{x_i\in\Gamma} [\varphi_j]_{x_i}[\varphi_k]_{x_i},
\end{split}
\end{equation}
for $\width\to 0^+$.

Altogether, the limiting formulation of (\ref{lem:FEM_zero:1}) for $\width=0$ is 
\begin{equation*}
\label{lem:FEM_zero:4}
\sum_{j=1}^M U_j\int_{\Omega\setminus\Gamma} \nabla \varphi_j\cdotp\nabla \varphi_k \dx +\sum_{j=1}^M U_j\tD\frac{h}{2} \sum_{x_i\in\Gamma} [\varphi_j]_{x_i}[\varphi_k]_{x_i} =(f,\varphi_k),
\end{equation*}
which is equivalent to \eqref{eq:FEM_zero}.

\end{proof}

Now that we have a limiting formulation for $\bar{h}=0$, the TFEM scheme can be analyzed. The main result is Theorem \ref{thm:FEM_zero_opt}, which basically states that if $\tD\sim\frac{1}{h^2}$, then 
\begin{equation}
|u-\tu_h|_{H^1(\Omega\setminus\Gamma)} +\sqrt{\tD}\big\|[u-\tu_h]\big\|_{L^2(\Gamma)}\leq C(u,L)h,
\end{equation}
as expected. In order not to break the flow of the paper, we move the theorem and its proof to Appendix \ref{apx:zero_tfem}. We note that similarly as in (\ref{eq:Jminopt}), the choice $\tD\sim\frac{1}{h^2}$ corresponds to simply taking $J_{\text{min}}\sim h^3$ in the code. 

\section{Extensions}
So far we've worked with a scalar Laplacian operator. 
In this section we discuss possible  extensions. The first one is to 
consider linear elasticity. The second extension deals with the potential use of the TFEM for the mortaring of non-conformal meshes, even when there is a large difference in mesh size. Third, we show that high-order elements adapt well to our approach. With a slightly modified choice of $\Jmin$, the TFEM ensures the proper convergence rate of high-order elements. The fourth and last extension studies the TFEM
capabilities for  advection type equations.
\subsection{Extension to linear elasticity}
The deformation $\mathbf{u}$ of a solid  under mechanical loads is    defined by the system of partial differential equations:
\begin{align*}
    \nabla \cdot \mathbf{\sigma}(\epsilon(\mathbf{u})) &= \mathbf{f} \hspace{2cm} \text{in }\Omega,\\
    \mathbf{u} &= \mathbf{u}_{D} \hspace{1.7cm} \text{on } \Gamma_D, \\
    \mathbf{\sigma} \cdot \mathbf{n} &= \mathbf{t}_N \hspace{1.8cm} \text{on } \Gamma_N, 
\end{align*}
where $\mathbf{\sigma}$ is the stress tensor, $\mathbf{f}$ body forces, and $\mathbf{t}_N$ normal stresses applied on the part $\Gamma_N$
of the boundary. The complementary of the boundary $\Gamma_D$ is where Dirichlet boundary conditions $\mathbf{u}_{D}$ are imposed.
The infinitesimal strain tensor $\epsilon$ 
is given by $\mathbf{\epsilon}(\mathbf{u}) = \frac{1}{2} \left( \nabla \mathbf{u} + \nabla \mathbf{u}^T \right)$.

The relationship between the stress tensor and the infinitesimal strain tensor is assumed linear (Hooke's model) :
\begin{equation*}
    \mathbf{\sigma} = \mathbf{C} : \mathbf{\epsilon},
\end{equation*}
where $\mathbf{C}$ is the elasticity tensor for plane strain in our case. 
 
In this section, we simulate the interaction between two different materials with two kinds of meshes: A regular mesh with the interface directly represented in it for the classic FEM and a mesh where the two materials are connected through a band of caps for the TFEM approach. 
The geometry of the problem is shown in Figure \ref{fig:elasticity_problem} as well as the  the von Mises stresses $\sigma_{\text{vm}} = \sqrt{\sigma_{xx}^2 + \sigma_{yy}^2 - \sigma_{xx}\sigma_{yy} + 3 \sigma_{xy}^2}$.  Results for both the classic FEM and the TFEM approach are visually very similar. \\
\definecolor{airforceblue}{rgb}{0.36,0.54,0.66}
\begin{figure}[!ht]
    \begin{tikzpicture}
        \draw[thick] (0,0) rectangle (5,5);
        \draw[thick] (0,2.5) -- (5,2.5);
        \node at (3.5, 1) [below right] {$E_1, \nu_1$};
        \node at (3.5, 3.5) [below right] {$E_2, \nu_2$};
        \node at (4, 5.5) [below right] {$\Gamma_N$};
        \draw[ultra thick, airforceblue] (0,5) -- (5,5);
        \node at (4, -0.6) [above right] {$\Gamma_D$};
        \draw[ultra thick, brown] (0,0) -- (5,0);
        \draw[->, thick] (1.5, 6) -- (1.5, 5.2);
        \draw[->, thick] (3.5, 6) -- (3.5, 5.2);
        \node at (8, 2.5) {\includegraphics[width=5cm]{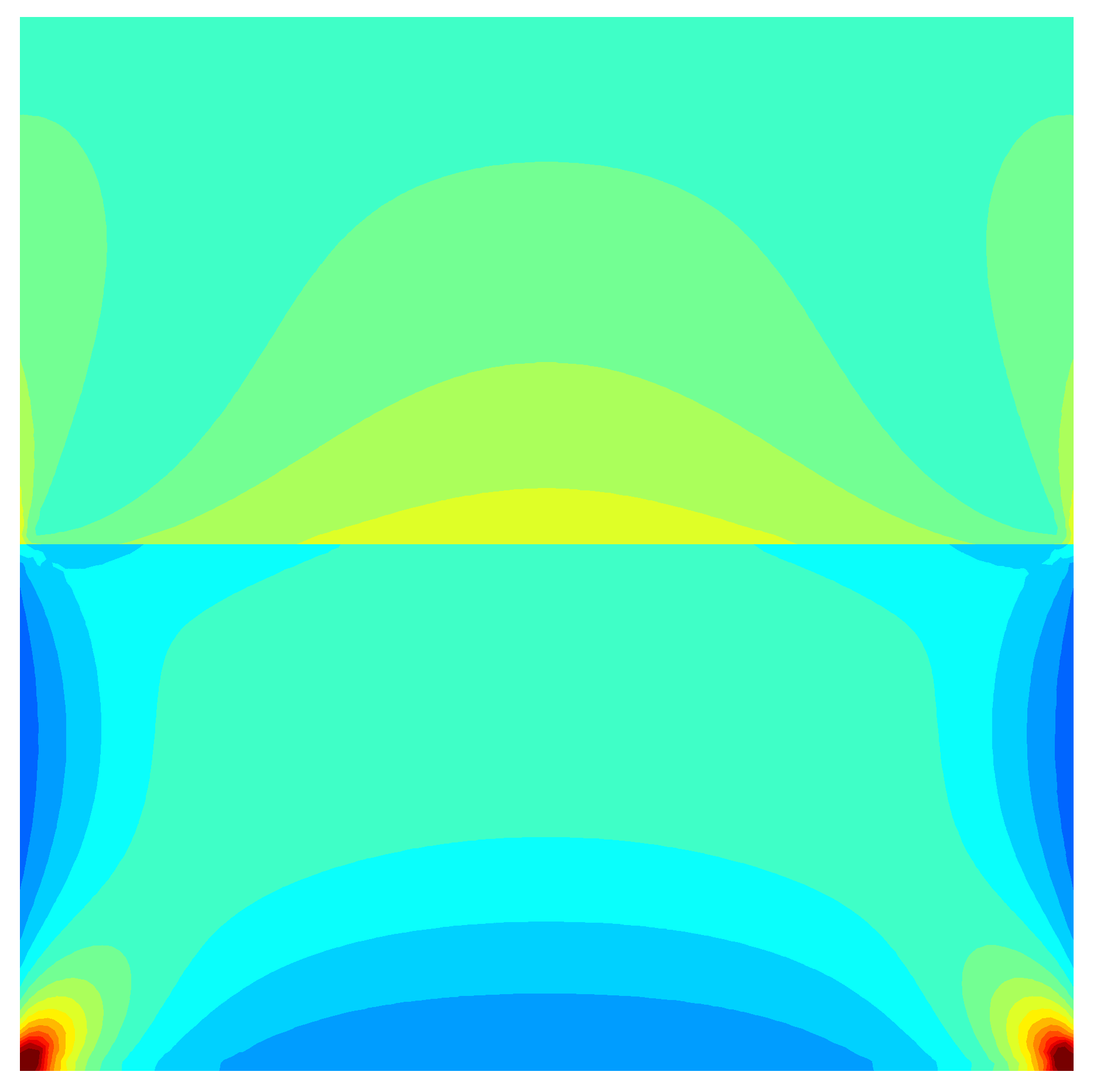}};
        \node at (13.5, 2.5) {\includegraphics[width=5cm]{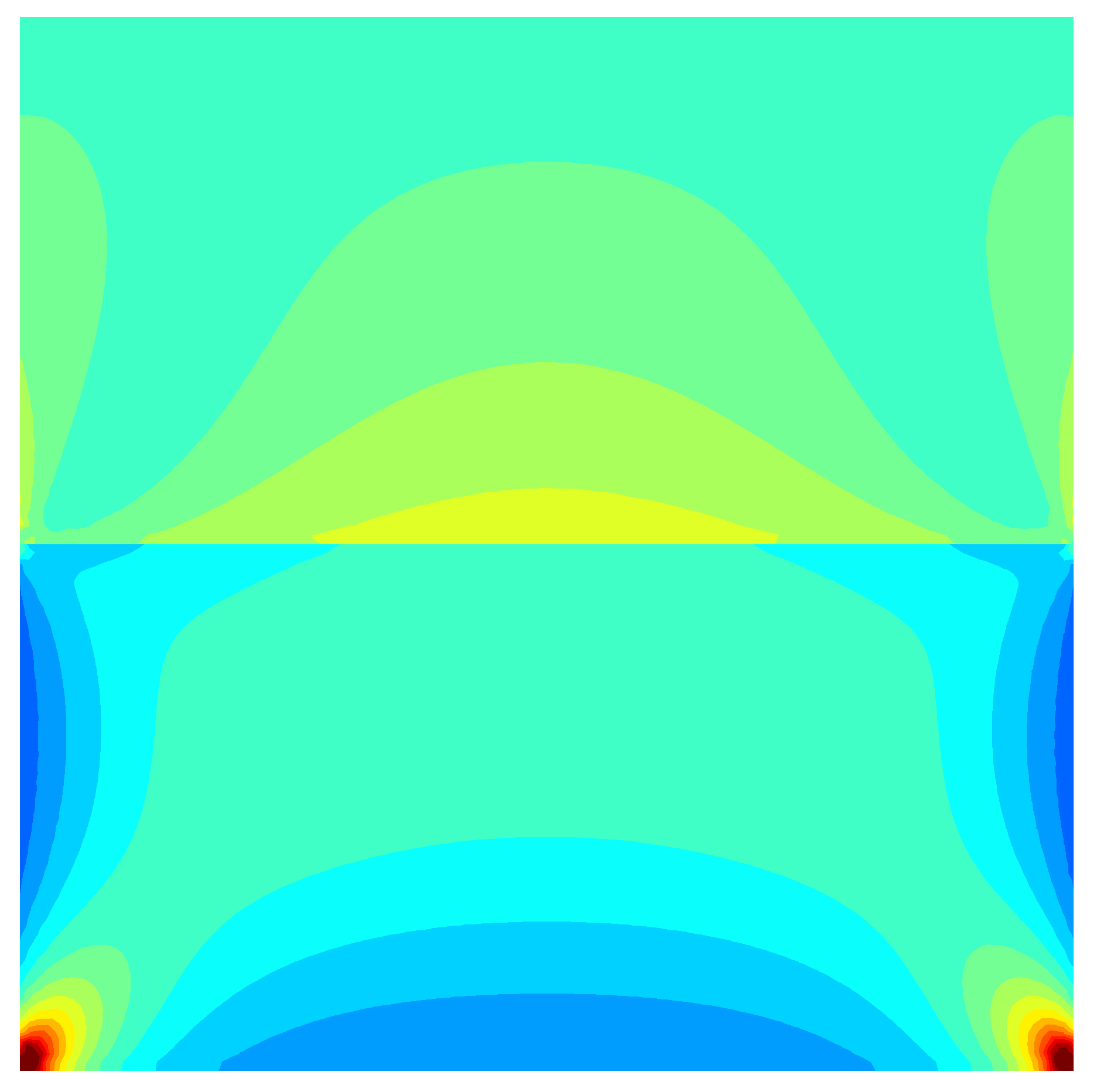}};
        \node at (10.75, -0.43) {\includegraphics[width=8cm]{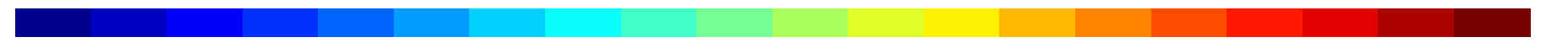}};
        \node at (6.45,-0.37) {$300$};
        \node at (15,-0.37) {$750$};
        \node at (10.77,-0.13) {$\sigma_{\text{vm}}$};
        \node at (8, 5.25) {Classic FEM};
        \node at (13.5, 5.25) {TFEM};
    \end{tikzpicture} \vspace{-0.5cm}
    \caption{Problem configuration (left) and von Mises stresses for the regular mesh (middle) and the mesh with a band of caps using TFEM (right).}
    \label{fig:elasticity_problem}
\end{figure}\\
To validate this observation we consider a problem with a manufactured solution such as the one presented in section \ref{sec:experiment}. As shown in the convergence plots at Figure \ref{fig:conv_elasticity}, the TFEM approach yields results that are comparable to those obtained using traditional finite element methods (FEM). However, the key advantage of TFEM lies in its flexibility: it allows for accurate simulation without requiring the mesh to be of good quality and can thus simplify the meshing of domains with an interface. The convergence graphs illustrate that both the traditional FEM and TFEM achieve similar convergence rates as the mesh is refined, validating the accuracy of the TFEM.
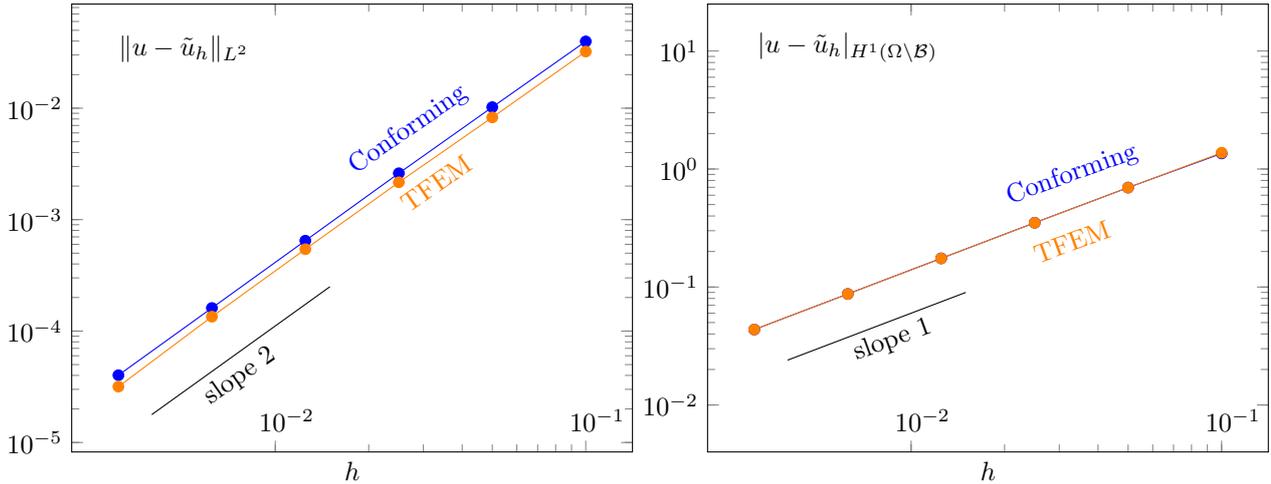
\begin{figure}[!ht]
    \hspace{-0.75cm}%
    \begin{tikzpicture}
        \begin{axis}[
            xlabel={$h$},
            ylabel={$\|u - \tilde{u}_h\|_{L^2}$},
            xtick={0.01, 0.1},
            xticklabel style={above, yshift=1ex, xshift=1ex},
            ytick={0.00001, 0.0001, 0.001, 0.01, 0.1},
            legend pos=north west,
            ymajorgrids=true,
            grid=none,
            xmode=log,
            ymode=log,
            width=0.56\textwidth,
            height=0.47\textwidth,
            x label style={at={(axis description cs:0.5,-0.0)},anchor=north},
            y label style={at={(axis description cs:0.2,0.85)},rotate=-90,anchor=south},
        ]

        \addplot[
            color=blue,
            mark=*,
            ]
            coordinates {
                (0.1, 0.03974827346829593) (0.05, 0.010217335092967733) (0.025, 0.0026056705258123165) (0.0125, 0.000647233502442667) (0.00625, 0.00016114168129659532) (0.003125, 4.0165661217557574e-05)
            }; 

        \addplot[
            color=orange,
            mark=*,
            ]
            coordinates {
                (0.1, 0.032220024516440036) (0.05, 0.008264206705749955) (0.025, 0.0021622854972983786) (0.0125, 0.0005424663723752355) (0.00625, 0.00013454673724207065) (0.003125, 3.172101081967118e-05)
            }; 

        \addplot[color=black, mark=none] 
            coordinates {(0.015, 0.00025) (0.004, 0.00001778)};
        \node [blue, rotate=35] at (rel axis cs: 0.6, 0.72) { Conforming };
        \node [orange, rotate=35] at (rel axis cs: 0.65,0.6) { TFEM };
        \node [black, rotate=35] at (rel axis cs: 0.30,0.17) { slope 2 };


        \end{axis}
        \end{tikzpicture}%
    \begin{tikzpicture}%
        \begin{axis}[
            xlabel={$h$},
            ylabel={$|u - \tilde{u}_h|_{H^1(\Omega \setminus \B)}$},
            xtick={0.01, 0.1, 1},
            ytick={0.01, 0.1, 1, 10},
            xticklabel style={above, yshift=1ex, xshift=1ex},
            ymin=0.004, 
            ymax=25,
            legend pos=north west,
            ymajorgrids=true,
            grid=none,
            xmode=log,
            ymode=log,
            width=0.56\textwidth,
            height=0.47\textwidth,
            x label style={at={(axis description cs:0.5,-0.0)},anchor=north},
            y label style={at={(axis description cs:0.25,0.85)},rotate=-90,anchor=south},
        ]
        
        \addplot[
            color=blue,
            mark=*,
            ]
            coordinates {
                (0.1, 1.3552792073254205) (0.05, 0.6965074297298216) (0.025, 0.3504682548133527) (0.0125, 0.17497862450947912) (0.00625, 0.08731890748270642) (0.003125, 0.043593499479793095)
            }; 

        \addplot[
            color=orange,
            mark=*,
            ]
            coordinates {
                (0.1, 1.3740144993207315) (0.05, 0.6955214040148843) (0.025, 0.35008069052992186) (0.0125, 0.17456221667309715) (0.00625, 0.08723099270892716) (0.003125, 0.04359312653247345)
            }; 
        
        \addplot[color=black, mark=none] 
            coordinates {(0.015, 0.090) (0.004, 0.024)};

        \node [blue, rotate=20] at (rel axis cs: 0.65,0.62) { Conforming };
        \node [orange, rotate=20] at (rel axis cs: 0.65,0.47) { TFEM };
        \node [black, rotate=20] at (rel axis cs: 0.33,0.25) { slope 1 };
            
        \end{axis}
        \end{tikzpicture}
        \caption{Convergence of the error in the $L^2$ norm (left) and $H^1$ semi-norm (right) for a manufactured elasticity problem.}
        \label{fig:conv_elasticity}
\end{figure}\\
\subsection{Mesh mortaring}
    The TFEM method is equivalent to a mortaring method in which the solution jump is penalized, cf. Section \ref{section:zero-measure}. Visually, as shown in Figure \ref{fig:meshes}, the mesh resembles a non-conforming mesh. This raises the possibility that the method could be valuable for mortaring non-conforming meshes within a standard FEM code, allowing for the simulation of specific regions of the domain with a more refined mesh without the need to ensure connectivity between different zones during mesh generation. The objective of the following numerical experiment is to compare the results between our approach and a conformal mesh. In this experiment, one half of the domain is meshed five times more finely than the other, as illustrated in Figure \ref{fig:mortar_sol}. We compare the TFEM approach to classical FEM, where a mesh with a transitional region is considered.
\begin{figure}[!ht]
    \centering
    \begin{subfigure}{0.485\textwidth}
        \centering
        \includegraphics[width=\linewidth]{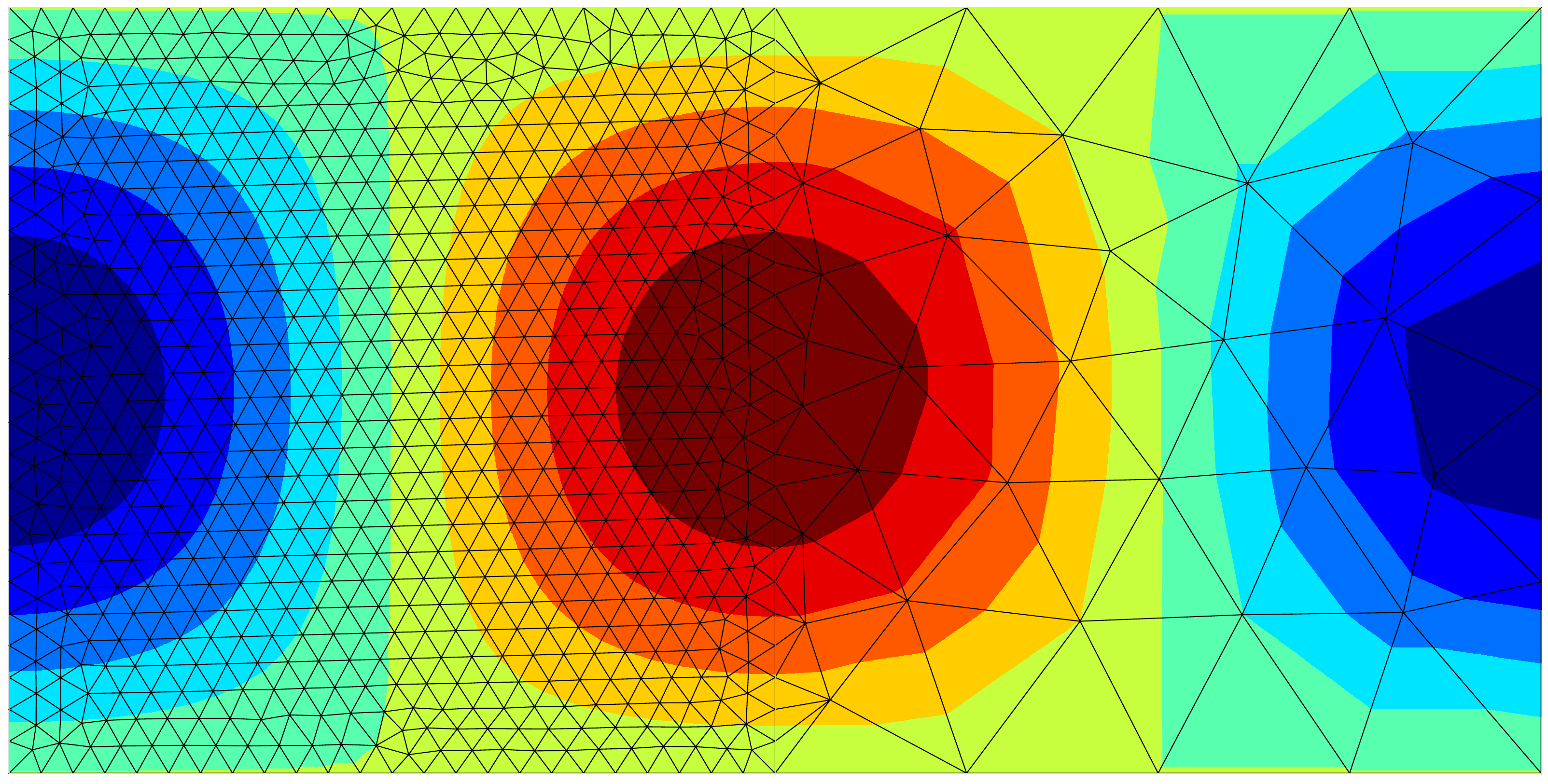}
    \end{subfigure}
    \begin{subfigure}{0.485\textwidth}
        \begin{tikzpicture}
            \node[inner sep=0pt] at (0,0) { \includegraphics[width=\linewidth]{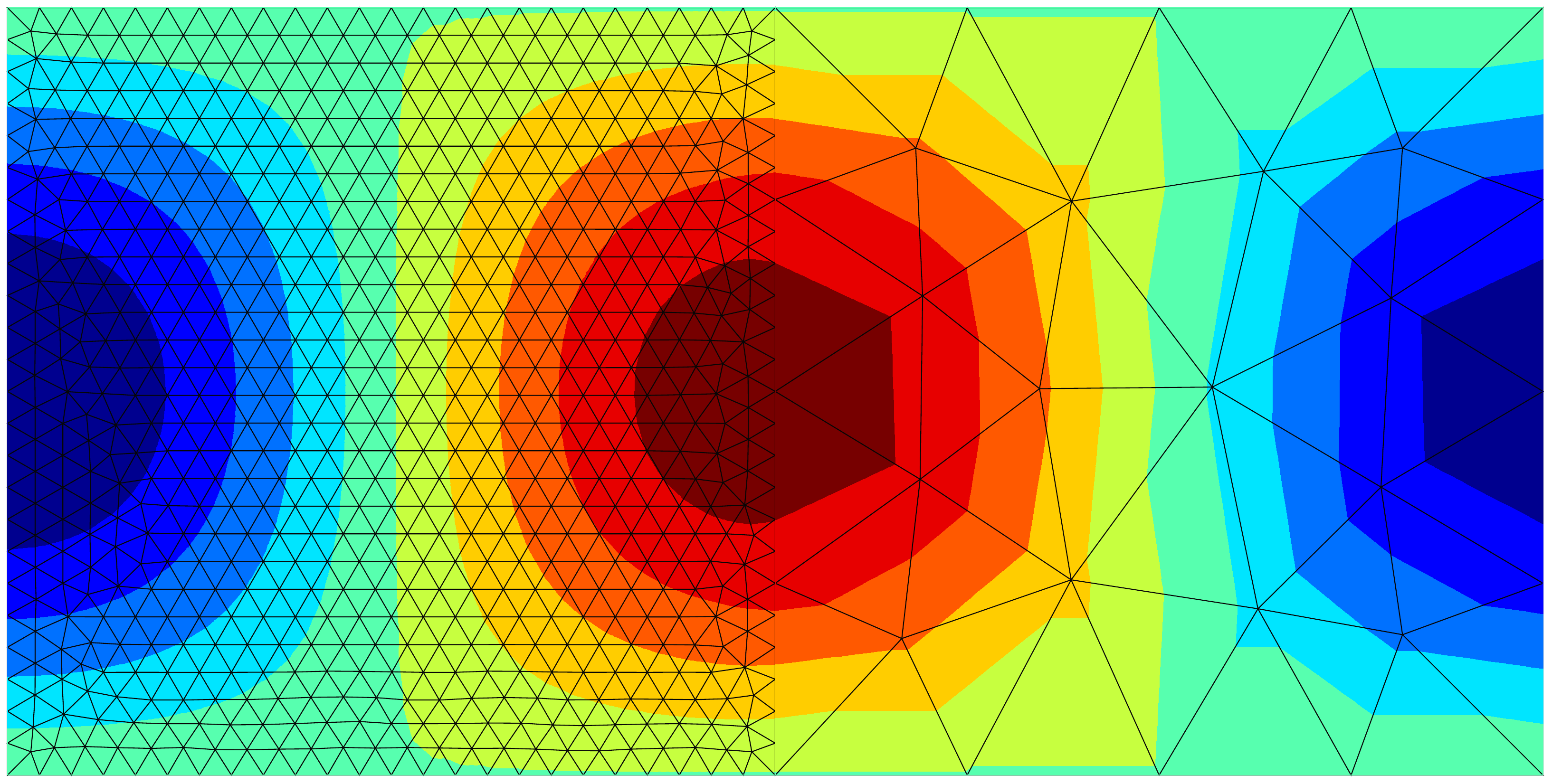}};
            \node[inner sep=0pt] at (0.42\linewidth,0.08\linewidth) { \includegraphics[width=0.12\linewidth]{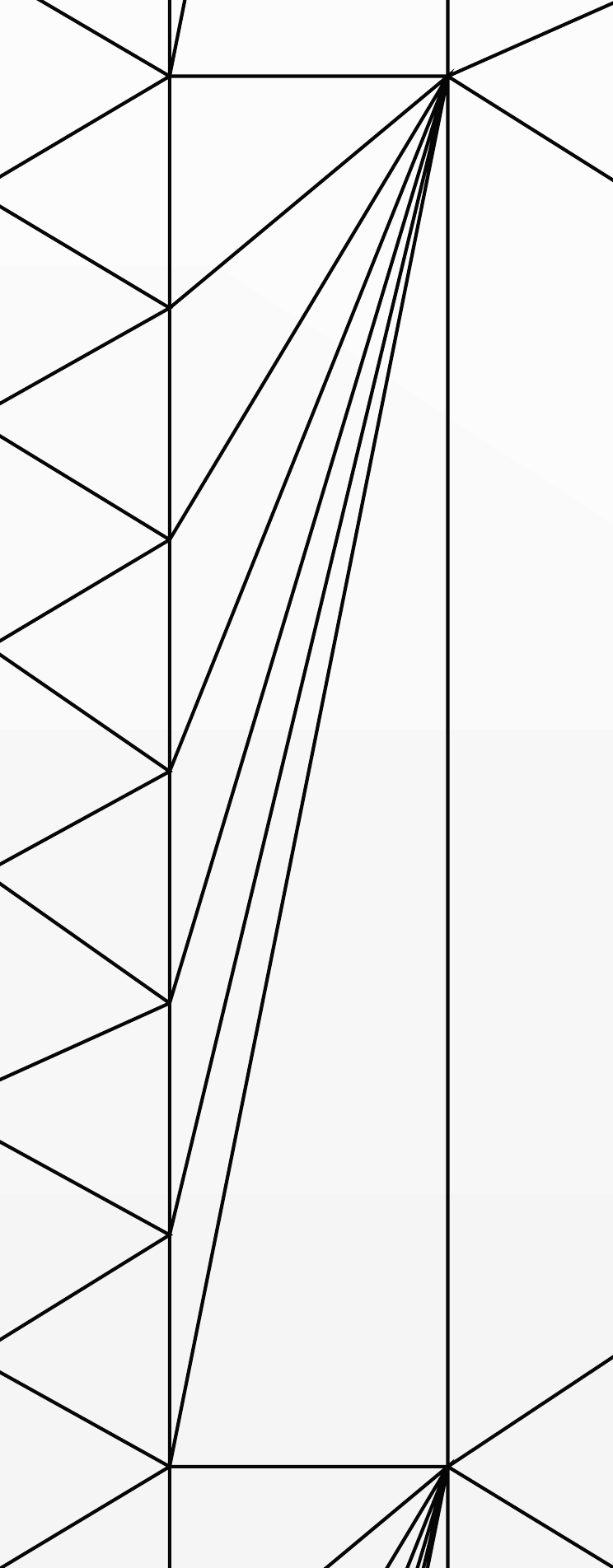}};
            
            \draw[color=white,very thick] (0.36\linewidth,-0.075\linewidth) rectangle (0.48\linewidth,0.235\linewidth);
            \draw[color=white,thick] (-0.018\linewidth,-0.01\linewidth) rectangle (0.018\linewidth, 0.13\linewidth);
            
            \draw[color=white, thick] (0.018\linewidth,-0.01\linewidth) -- (0.36\linewidth,-0.075\linewidth);
            \draw[color=white, thick] (0.018\linewidth,0.13\linewidth) -- (0.36\linewidth,0.235\linewidth);
        \end{tikzpicture}
    \end{subfigure}
    \caption{Mesh and solution for a classical FEM approach (left) and with the TFEM approach (right).}
    \label{fig:mortar_sol}
\end{figure}\\
The split view of the mesh in Figure \ref{fig:mortar_sol} (right) demonstrates how the two parts of the mesh are connected using degenerate elements. These are collapsed to have exactly zero width $\width$ and treated using TFEM. These elements are generated in a straightforward manner, ensuring simplicity of implementation.

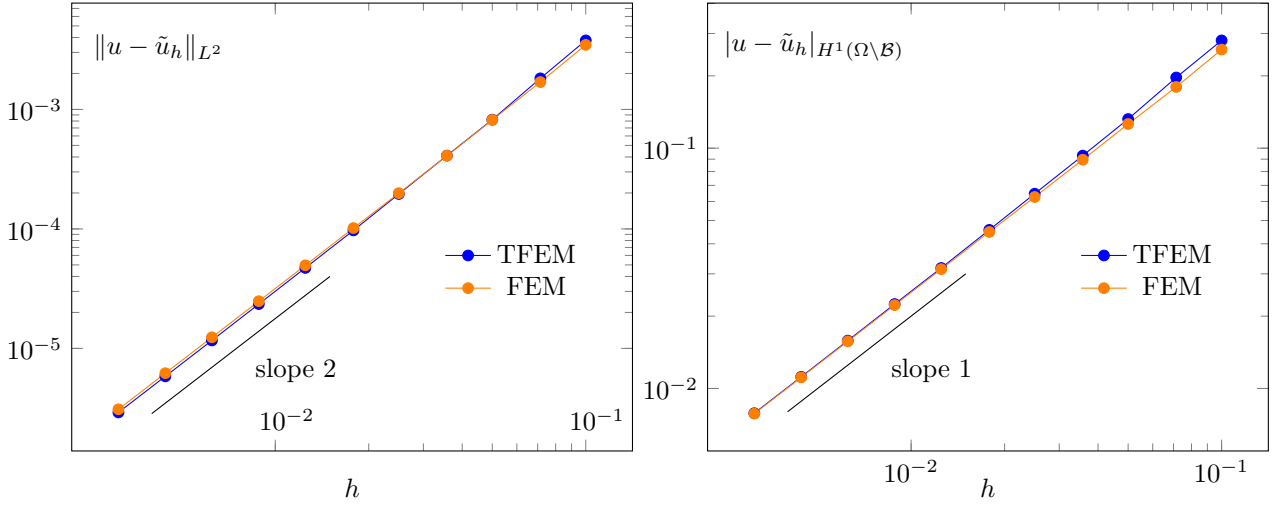
\begin{figure}[!ht]
    \hspace{-0.75cm}%
    \begin{tikzpicture}
        \begin{axis}[
            xlabel={$h$},
            ylabel={$\|u - \tilde{u}_h\|_{L^2}$},
            xtick={0.01, 0.1},
            xticklabel style={above, yshift=1ex, xshift=1ex},
            ytick={0.00001, 0.0001, 0.001, 0.01, 0.1},
            legend style={at={(0.65,0.4)},anchor=west},
            ymajorgrids=true,
            grid=none,
            xmode=log,
            ymode=log,
            width=0.56\textwidth,
            height=0.47\textwidth,
            x label style={at={(axis description cs:0.5,-0.04)},anchor=north},
            y label style={at={(axis description cs:0.155,0.85)},rotate=-90,anchor=south},
            legend style={draw=none}
        ]
        
        \addplot[
            color=blue,
            mark=*,
            ]
            coordinates {
                (0.1000000000, 0.0037821717168802357) (0.0714285714, 0.0018210556292122325) (0.0500000000, 0.0008216254348959442) (0.0357142857, 0.00041080879085906636) (0.0250000000, 0.00019604634895768828) (0.0178571429, 9.71770490938907e-05) (0.0125000000, 4.717411926376751e-05) (0.0088495575, 2.3498982556892935e-05) (0.0062500000, 1.165359157803714e-05) (0.0044247788, 5.836543130218828e-06) (0.0031250000,  2.9110517977460716e-06) 
            };
            
        \addplot[
            color=orange,
            mark=*,
            ]
            coordinates {
                (0.1000000000, 0.0034701829) (0.0714285714, 0.0016943771) (0.0500000000, 0.0008173497) (0.0357142857, 0.0004095590) (0.0250000000, 0.0001991787) (0.0178571429, 0.0001017064) (0.0125000000, 0.0000495798) (0.0088495575, 0.0000248433) (0.0062500000, 0.0000123848) (0.0044247788, 0.0000062028) (0.0031250000, 0.0000030923)
            };

        \addplot[color=black, mark=none] 
            coordinates {(0.015, 0.00004) (0.004, 0.0000028448)};
        
        \node [black, rotate=0] at (rel axis cs: 0.4,0.18) { slope 2 };


        \legend{TFEM, FEM}

        \end{axis}
        \end{tikzpicture}%
    \begin{tikzpicture}%
        \begin{axis}[
            xlabel={$h$},
            ylabel={$|u - \tilde{u}_h|_{H^1(\Omega \setminus \B)}$},
            xtick={0.01, 0.1, 1},
            ytick={0.01, 0.1, 1, 10},
            legend style={at={(0.65,0.4)},anchor=west},
            ymajorgrids=true,
            grid=none,
            xmode=log,
            ymode=log,
            width=0.56\textwidth,
            height=0.47\textwidth,
            x label style={at={(axis description cs:0.5,-0.04)},anchor=north},
            y label style={at={(axis description cs:0.19,0.85)},rotate=-90,anchor=south},
            legend style={draw=none}
        ]
        
        \addplot[
            color=blue,
            mark=*,
            ]
            coordinates {
                (0.1000000000, 0.28084802894009586) (0.0714285714, 0.19694682061423252) (0.0500000000, 0.13218668930658264) (0.0357142857, 0.09312972550440507) (0.0250000000, 0.06465135501982358) (0.0178571429, 0.04570545837010874) (0.0125000000, 0.03172976410099902) (0.0088495575, 0.022456852222826566) (0.0062500000, 0.015818497569468792) (0.0044247788, 0.011181431373388563) (0.0031250000, 0.007891771794919022) 
            };
            
        \addplot[
            color=orange,
            mark=*,
            ]
            coordinates {
                (0.1000000000, 0.2572831523) (0.0714285714, 0.1798048120) (0.0500000000, 0.1260804348) (0.0357142857, 0.0895468438) (0.0250000000, 0.0626357136) (0.0178571429, 0.0448076216) (0.0125000000, 0.0313731802) (0.0088495575, 0.0222298001) (0.0062500000, 0.0157086182) (0.0044247788, 0.0111247495) (0.0031250000, 0.0078583967)    
            };

        \addplot[color=black, mark=none] 
        coordinates {(0.015, 0.03) (0.004, 0.008)};
        
        \node [black, rotate=0] at (rel axis cs: 0.4,0.18) { slope 1 };

        \legend{TFEM, FEM}
        \end{axis}
        \end{tikzpicture}
    \caption{Convergence of the error in the $L^2$ norm (left) and the $H^1$ semi-norm (right) for the TFEM mortaring of two meshes and the classical FEM with a transitional mesh.}
    \label{fig:mortar_conv}
\end{figure}
A good convergence rate is achieved for both types of meshes, and the error of our approach is very similar to that of a conforming mesh as we can see on Figure \ref{fig:mortar_conv}. 
Unlike some mortar methods, TFEM does not require differentiating between the sides of the mortaring edges as well as a specific integation scheme on the sides. This makes mortaring using TFEM straightforward, even for complex domain divisions.

\subsection{High order elements}
    To ensure the method's versatility, it is crucial to determine whether the TFEM is still effective with high-order elements. In this section, we investigate the convergence behavior of high-order elements and examine the impact of the first and second derivatives on the optimal choice of $J_{\text{min}}$.
\input{coeff_high_order_2D.tex}\\
    Figure \ref{fig:coeff_high_order_2D} illustrates the sensitivity of the $L^2$ and $H^1$ errors to the choice of $\Jmin$ for elements of orders 1 through 4. The error behavior is consistent across different orders, with a notable plateau observed in the $H^1$ error even for higher-order elements.  Black crosses indicate a heuristically fitted choice of $\Jmin = C h^{(2.2 + 0.8 \, \text{order})}$. As we can see on these graphs, this formula provides an effective selection for $\Jmin$.\\
\input{convergence_high_order_2D.tex}\\
    As illustrated in Figure \ref{fig:conv_high_order_2D}, this choice of $\Jmin$ preserves the good convergence rate of high-order elements in 2D. The solution obtained with our TFEM approach for a degenerate mesh (solid line) is compared to the classical solution on a regular mesh (dashed line). For both the $L^2$ and $H^1$ errors, our approach performs comparably to the method on a regular mesh. The theoretical analysis of higher order TFEM will be performed in a forthcoming paper.
    
\subsection{Advection}
So far, we were focused on diffusion operators.  In this section, we briefly investigate advection. The advection equation is known to pose stability challenges when using continuous finite elements, which necessitates the implementation of stabilization techniques. To address this, we apply the Streamline Upwind/Petrov-Galerkin (SUPG) method.
The stabilized advection equation weak form is to find $\phi$ satisfying \eqref{eq:sup} for all $\psi$:
\begin{equation}
    \int_{\Omega} \mathbf{v} \cdot \nabla \phi \, \left( \psi + \tau_{\text{SUPG}} \mathbf{v} \cdot \nabla \psi \right) \dx  \label{eq:sup}
\end{equation}
Let us note that for this equation the TFEM approach only needs to be applied on the SUPG stabilization term and not the advection one. Indeed, when the advection local matrix is computed, the determinant $\J$  of the transformation between the reference element and the degenerate element appears in the denominator for the computation of the gradient of the shape function and in the numerator for the computation of the integral. They cancel out and no $\J$ is left in the final advection matrix.\\
\input{advection_problem}\\
We consider a test case with a circular velocity field, denoted as $\mathbf{v}$. 
The boundary conditions for this case correspond to a Dirichlet condition on the boundary $\Gamma_D$ defined by a smoothstep function as we can see in Figure \ref{fig:adv_problem}.\\
We test the advection problem on two types of meshes: a regular reference mesh and a mesh containing a band of caps, where the TFEM approach is applied. The convergence of the error is analyzed by plotting the error as a function of mesh characteristic length $h$. The results, shown in the convergence graph of Figure \ref{fig:conv_advection} (left), indicate that the error converges at the expected rate for both $L^2$ and $H^1$ norms, validating the accuracy and robustness of the TFEM approach when applied to advection problems.
\input{advection_conv.tex}
\newpage

\section{Conclusion}

In this work we have developed a new method called the Tempered Finite Element Method (TFEM), which increases the flexibility of finite elements with respect to meshing. The method is simple to implement within any finite element code: it simply consists to avoiding division by (almost) zero in the evaluation of FEM matrix entries. We have shown through numerical experiments that on meshes with degenerate elements the method provides an error close to the one of finite elements on regular meshes and that it converges at the same rate. We have also proven this convergence theoretically for Poisson's problem and linear elements. This method is however not only useful for solving a Laplacian. It works just as well for linear elasticity, advection, and high-order elements and advection. For exactly zero-measure elements, we show that TFEM corresponds to non-conforming mesh mortaring. We are particularly interested in solving on meshes containing a so-called band of caps. This structure mimics a finely resolved one-dimensional interface in the two-dimensional domain. The TFEM approach is able to completely resolve the locking phenomenon that standard finite elements undergo in this situation.

A crucial point of the scheme is the choice of a parameter called $\Jmin$, which is the minimal value of element Jacobians that we are willing to tolerate in denominators when evaluating FEM matrix entries. Our numerical experiments and the mathematical theory have identified the dependence of $\Jmin$ on $h$, the mesh size. However, the optimal practical choice of $C$ in the expression $\Jmin = C h^{d+1}$ in $\mathbb{R}^d$ remains to be determined. The mathematical study of the scheme provides a theoretically optimal choice of $C$. However, we have observed numerically that it was not always the best possible in practice. Nevertheless, the sensitivity analysis shows that the method is robust as the choice of $C$ could vary by at least an order of magnitude around the minimum without being too far from the classical finite element error. Although the choice of the optimal $C$ should be the subject of further research, the method can already be used. We show this by performing 100 tests on random functions consisting of a sum of linear, quadratic, sine and cosine functions of random amplitude and frequency, and a Gaussian function of random amplitude and position. The domains considered were rectangles of random sizes containing a band of caps of variable length and we considered the simple choice $\Jmin = h^3$, i.e. $C=1$. The relative errors in the $L^2$ norm and in the $H^1$ semi-norm are shown in Figure \ref{fig:random}. It can be seen that, despite the simple choice of $C=1$, the error of the TFEM method is close to the error of the FEM method on a regular mesh. In the $H^1$ semi-norm, both errors are almost identical. 

The TFEM method could save considerable time and effort in mesh generation by reducing the mesh quality required to perform finite element analysis.
\input{random_test.tex}

\newpage
\section*{Acknowledgment}
This project has received funding from the European Research Council (ERC) under the European Union’s Horizon research and
innovation program (Grant agreement No. 101 071 255).

\appendix
\section{Appendix: Study of the FEM stiffness matrix} \label{apx:stiff}
Let's study the stiffness matrix related to the Laplace operator on
a triangle (see Figure \ref{fig:cap_apx}):
$$K_{ij} = \int_{T} \nabla \phi_i \cdot \nabla \phi_j \;
dx~~~,~~~i,j=1,2,3.$$
Assume three scalar values $u_i$, $i=1,2,3$, associated to nodes located at 
$x_i$, $i=1,2,3$. The stiffness matrix, $[K]$, allows one to write the quadratic expression
of the energy:
\begin{equation}
f(u_1,u_2,u_3) = \frac12 \sum_{i=1}^3\sum_{j=1}^3K_{ij} u_i u_j.
\label{eq:quadf_bis}
\end{equation}
It turns out that  $[K]$  only depends on the 
triangle angles denoted $\theta_i$'s:
\begin{eqnarray}
    \begin{bmatrix} K \end{bmatrix}
&=& 
    \begin{bmatrix}
    \cot(\theta_2)+\cot(\theta_3)  & -\cot(\theta_3) & -\cot(\theta_2) \\ 
    -\cot(\theta_3) &\cot(\theta_1)+\cot(\theta_3) &  -\cot(\theta_1) \\ 
    -\cot(\theta_2)&  -\cot(\theta_1) &  \cot(\theta_1)+\cot(\theta_2)  
\end{bmatrix} 
=  
    \begin{bmatrix}
    c_2+c_3  & -c_3 & -c_2 \\ 
    -c_3 &c_1+c_3 &  -c_1 \\ 
    -c_2&  -c_1 &  c_1+c_2  
    \end{bmatrix}.
\label{eq:cot}
\end{eqnarray}

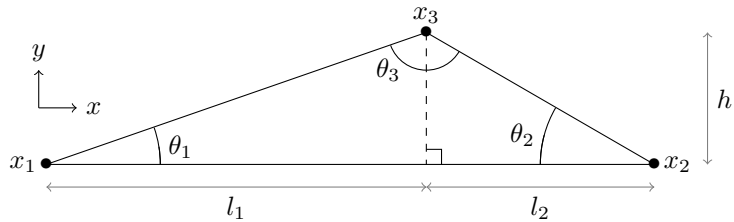
\begin{figure}[!ht]
\begin{center}
        \begin{tikzpicture}
            \draw (0,0) node {$\bullet$} ;
            \draw (5,1.75) node {$\bullet$} ;
            \draw (8,0) node {$\bullet$} ;
            \draw (5, 0.2) -- (5.2, 0.2) -- (5.2, 0);
            \draw[] (0,0)  node[left]{${x}_1$} --  (5,1.75) node[above]{${x}_3$} --  (8,0)  node[right]{${x}_2$} -- (0,0);
            \draw[->] (-0.1,0.75)--(0.4,0.75) node[right]{$x$};
            \draw[->] (-0.1,0.75)--(-0.1,1.25) node[above]{$y$};
            \draw[dashed] (5,1.75)--(5,0.) ;
            \draw[<->,color=gray] (0,-0.3)--(5,-0.3) ;
            \draw (2.5,-0.3)node[below]{$l_1$} ;
            \draw[<->,color=gray] (5,-0.3)--(8,-0.3) ;
            \draw (6.5,-0.3)node[below]{$l_2$} ;
            \draw[<->,color=gray] (8.7,0.)--(8.7,1.75) ;
            \draw (8.7,0.875)node[right]{$h$} ;
            \draw[] (1.5,0) arc (0:9:1.5) node[right]{$\theta_1$} ;
            \draw[] (1.5,0) arc (0:19:1.5) ;
            \draw[] (6.5,0) arc (180:165:1.5) node[left]{$\theta_2$} ;
            \draw[] (6.5,0) arc (180:150:1.5);
            \draw[] (5,1.25) arc (-90:-30:0.5);
            \draw[] (5,1.25) arc (-90:-160:0.5);
            \draw[] (4.5, 1.25) node{$\theta_3$};
        \end{tikzpicture}
    \end{center}
    \caption{A triangle $T$. The angle 
    $\theta_i$ is associated to the vertex $x_i$.\label{fig:cap_apx}}
\end{figure}
Without loss of generality, we  assume  that $\theta_3 \geq \theta_2 \geq \theta_1$. 
Edge $({x}_1, {x}_2)$ is thus the longest edge of T (see
Figure \ref{fig:cap_apx}) with $\|{x}_2 - {x}_1\| = l_1+l_2$. 
We define the element flatness $f$ and its symmetry  $s$ with the following expression
\begin{equation}
    f = \frac{h}{l_1+l_2}, \quad s = \frac{l_1}{l_1+l_2}, \quad f \in\left[0,
      {\sqrt{3}/2}\right], \quad s \in [ 0, 1 ].
    \label{eq:f_bis}
\end{equation}
The cotangents may be expressed as a function of $f$ and $s$:
$$    c_1 = \frac{s}{f}, \quad c_2 = \frac{1-s}{f}, \quad c_3 = f -
\frac{s(1-s)}{f}.$$
Eigenvalues $\lambda_1 \leq \lambda_2 \leq \lambda_3$ of $[K]$ can be expressed as functions of $c =
c_1+c_2+c_3 = f + \frac{ 1 - s + s^2}{f}$ as:
\begin{equation}
{    \lambda_1 = 0, \quad \lambda_2 = c - \sqrt{c^2 - 3}, \quad 
    \lambda_3 = c + \sqrt{c^2 - 3}.}\nonumber
\end{equation}
whereas the corresponding eigenvectors are
\begin{align}
    {\bf v}_1 & = [ 1, 1, 1], \nonumber \\
    {\bf v}_2 & = [ c_3 - c_1 - \sqrt{c^2 - 3}, -c_3 + c_2 + \sqrt{c^2 - 3} , c_1 - c_2], \nonumber \\
    {\bf v}_3 & = [ 1 - \alpha,  \alpha, -1], \quad \alpha = \frac{2 c_1 -c_2 - c_3   + \sqrt{c^2 - 3}}{c_1 + c_2 - 2 c_3 + 2 \sqrt{c^2 - 3}} \in [0,1].\nonumber 
\end{align}
Finally, the energy $e_i$ associated to each mode is given by $e_i =  \lambda_i \| {\bf v}_i \|^2/2 $.
\begin{align*}
e_1&= 0, \\
e_2&=\frac{(c - \beta)}{2} \left[(c_1 - c_2)^2 + (-c_1 + c_3 - \beta)^2 + (c_2 - c_3 + \beta)^2\right], \\
e_3&= \frac{(c - \beta)}{2} \left[\left((1 - \frac{2c_1 - c_2 - c_3 + \beta}{c_1 + c_2 - 2c_3 + 2\beta}\right)^2 + 1 + \left( \frac{2c_1 - c_2 - c_3 + \beta}{c_1 + c_2 - 2c_3 + 2\beta} \right)^2\right], \\
\text{with } \beta &= \sqrt{c^2 - 3}.
\end{align*}
We note that the first mode is the rigid
mode with zero energy.
We are interested in the limit case  
$f \rightarrow
0$ for which the element reaches zero measure. 
It is possible to show using a Taylor expansion that as $f \rightarrow 0$ we have
\begin{equation}{\lambda_2 \sim \frac{3}{2} \frac{f}{1-s+s^2}
    \quad \text{and} \quad \lambda_3 \sim  2 \frac{ 1 - s + s^2}{f}}. 
\label{eq:eig_bis}
  \end{equation}

\section{Appendix: Analysis of tempered FEM on zero-measure elements} \label{apx:zero_tfem}
Now we analyze the limiting scheme (\ref{eq:FEM_zero}) and prove its convergence. There will be two technical issues with this analysis:
\begin{itemize}
\item The scheme is not consistent and we must deal with a consistency error.
\item As noted in Remark \ref{rem:quadrature}, the second term in (\ref{eq:FEM_zero}) is essentially a quadrature approximation of an integral, while the non-consistency term will be in integral form and this discrepancy must be dealt with.
\end{itemize}

First, we detail the equation that the exact solution $u$ satisfies on $\Omega\setminus\Gamma$, if we use $v_h\in V_h^\Gamma$ as a test function.

\begin{lemma}
\label{lem:cont_sol}
The exact solution $u$ of (\ref{eq:problem}) satisfies
\begin{equation}
\label{lem:cont_sol_eq}
\int_{\Omega\setminus\Gamma}\nabla u\cdotp\nabla v_h \dx -\int_\Gamma\nabla u\cdotp n_\Gamma[v_h]\dS +\tD\frac{h}{2}\sum_{x_i\in\Gamma}[u]_{x_i}[v_h]_{x_i} =(f,v_h),
\end{equation}
for all $v_h\in V_h^\Gamma$.
\end{lemma}
\begin{proof}
We multiply problem (\ref{eq:problem}) by $v_h$, integrate over $\Omega\setminus\Gamma$ and apply Green's theorem. Since $v_h$ is in general discontinuous on $\Gamma$, we obtain
\begin{equation}
\int_{\Omega\setminus\Gamma}\nabla u\cdotp\nabla v_h \dx -\int_\Gamma\nabla u\cdotp n_\Gamma[v_h]\dS =(f,v_h).
\end{equation}
Since $[u]=0$ on $\Gamma$, we can add the last left-hand side term in (\ref{lem:cont_sol_eq}) to the formulation as it is equal to zero.
\end{proof}

Now we address the second problem.

\begin{lemma}
\label{lem:quadrature}
Let $f:\Gamma\to\IR$ be a piecewise linear function on the partition $x_i, i=0,\ldots,N+1,$ such that $f(x_0)=f(x_{N+1})=0$. Then
\begin{equation}
\label{lem:quadrature_est}
\int_\Gamma f^2 \dS\leq \frac{h}{2}\sum_{x_i\in\Gamma} f^2(x_i).
\end{equation} 
\end{lemma}
\begin{proof}
For simplicity, we can consider $\Gamma$ to be the one-dimensional interval $[0,L]$ with corresponding nodes $x_i$. We consider two neighboring points $x_i, x_{i+1}$ of the partition and estimate the integral over this segment of $\Gamma$, which has length $h/2$. Since $f$ is linear on this segment, we can write $f(x)=f(x_i)\phi_i(x)+f(x_{i+1})\phi_{i+1}(x)$, where $\phi_i(x)=\frac{2}{h}(x_{i+1}-x)$ and $\phi_{i+1}(x)=\frac{2}{h}(x-x_{i})$ are the standard Lagrange basis functions in 1D, for which the necessary integrals can be easily evaluated. Therefore
\begin{equation}
\begin{split}
\label{lem:quadrature_1}
&\int_{x_i}^{x_{i+1}} f^2(x) \dx \\ &= \int_{x_i}^{x_{i+1}} f^2(x_i)\phi_i^2(x) +2 f(x_i)f(x_{i+1})\phi_i(x)\phi_{i+1}(x) +f^2(x_{i+1})\phi^2_{i+1}(x)\dx\\
&=\frac{h}{6}f^2(x_{i})+ 2\frac{h}{12}f(x_{i})f(x_{i+1}) +\frac{h}{6}f^2(x_{i+1})\leq \frac{h}{4}f^2(x_{i}) +\frac{h}{4}f^2(x_{i+1}),
\end{split}
\end{equation}
since $2ab\leq a^2+b^2$. Summing (\ref{lem:quadrature_1}) over all $i=0,\ldots,N$ and taking into account $f(x_0)=f(x_{N+1})=0$ gives us (\ref{lem:quadrature_est}).
\end{proof}

We will eventually apply Lemma \ref{lem:quadrature} to the function $f=[\xi_h]$. For this it is important to realize that $[\xi_h]$ is a piecewise linear function on the partition given by the nodes $x_i$. Indeed, $[\xi_h]=\xi_h^--\xi_h^+$, and $\xi_h^-$ is linear on the segments between $x_i, x_{i+2}$ for some $i$, while $\xi_h^+$ is linear on the segment between $x_{i-1}, x_{i+1}$, cf. Figure \ref{fig:Band_zero_measure}. Therefore, their difference $[\xi_h]$ is linear on the intersection, i.e. on the segment between $x_i, x_{i+1}$. Furthermore, $[\xi_h]_{x_0}=[\xi_h]_{x_{N+1}}=0$, since $\xi_h$ is continuous at the endpoints of $\Gamma$.

Next, we need an estimate of the interpolation error on $\Gamma$.

\begin{lemma}
\label{lem:interp_error_gamma}
It holds that
\begin{align}
\label{lem:interp_error_gamma_est1}
\frac{h}{2}\sum_{x_i\in\Gamma}[u-\Pi_h u]_{x_i}^2 &\leq \frac{1}{64}L h^4 \|\gradt^2 u\|_{L^\infty(\Gamma)}^2,\\
\label{lem:interp_error_gamma_est2}
\int_\Gamma[u-\Pi_h u]^2\dS &\leq \frac{1}{16}L h^4 \|\gradt^2 u\|_{L^\infty(\Gamma)}^2,
\end{align}
where $\gradt^2 u$ is the second derivative of $u$ in the direction parallel to $\Gamma$.
\end{lemma}
\begin{proof}
\textit{First estimate:} Since each $x_i$ is the vertex of an element neighboring $\Gamma$, and $\Pi_h u$ interpolates $u$ at vertices, we have either $(u-\Pi_h u)^-(x_i)=0$ or $(u-\Pi_h u)^+(x_i)=0$. Let e.g. the latter case hold. The point $x_i$ is the midpoint of the edge of an adjacent element $K^-$ and $(u-\Pi_h u)^-(x_i)$ is simply the interpolation error of the 1D linear Lagrange interpolation of $u$ on the edge of $K^-$ with endpoints $x_{i-1},x_{i+1}$. Cauchy's theorem for the remainder of linear Lagrange interpolation on an interval can be applied, cf. \cite{Davis}, which guarantees the existence of a point $\zeta$ on the edge, such that
\begin{equation}
\big|(u-\Pi_h u)^-(x_i)\big| = \frac{1}{2} \big|\gradt^2 u(\zeta) (x_i-x_{i-1})(x_i-x_{i+1})\big| \leq \frac{1}{2} \|\gradt^2 u\|_{L^\infty(\Gamma)}\frac{h}{2} \frac{h}{2}.
\end{equation}
Using this estimate, we get 
\begin{equation}
\begin{split}
&\frac{h}{2}\sum_{x_i\in\Gamma}[u-\Pi_h u]_{x_i}^2 \leq \frac{h}{2}\sum_{x_i\in\Gamma} \big(\tfrac{1}{8}h^2\|\gradt^2 u\|_{L^\infty(\Omega)}\big)^2\\ &\quad =\frac{1}{64}h^4\|\gradt^2 u\|_{L^\infty(\Omega)}^2 \sum_{x_i\in\Gamma}\frac{h}{2}  =\frac{1}{64}L h^4 \|\gradt^2 u\|_{L^\infty(\Omega)}^2,
\end{split}
\end{equation}
since $\sum_{x_i\in\Gamma}\frac{h}{2}=L$. This gives us (\ref{lem:interp_error_gamma_est1}).

\medskip
\textit{Second estimate:} 
We have
\begin{equation}
\begin{split}
\label{lem:interp_error_gamma_est:3}
\int_\Gamma[u-\Pi_h u]^2\dS &=\int_\Gamma\big((u-\Pi_h u)^--(u-\Pi_h u)^+\big)^2\dS\\
&\leq 2\int_\Gamma\big((u-\Pi_h u)^-\big)^2\dS +2\int_\Gamma\big((u-\Pi_h u)^+\big)^2\dS.
\end{split}
\end{equation}
The right-hand side integrals can be written as the sum over edges of neighboring elements that lie on $\Gamma$. On these edges, $u-\Pi_h u$ is simply the error of the 1D linear Lagrange interpolation of $u$ on the edge. Similarly as in the proof of Lemma \ref{lem:quadrature}, we can consider $\Gamma$ to be the one-dimensional interval $[0,L]$ with corresponding nodes $x_i$. An edge of an element neighboring $\Gamma$ is an interval $[x_i,x_{i+2}]$ of length $h$, and Cauchy's theorem for the remainder of linear Lagrange interpolation gives for each $x$ a $\zeta_x$ such that
\begin{equation}
\begin{split}
\label{lem:interp_error_gamma_est:4}
\int_{x_i}^{x_{i+2}}\big(u-\Pi_h u)^2\dS &= \int_{x_i}^{x_{i+2}} \bigg(\frac{1}{2} \gradt^2 u(\zeta_x) (x-x_{i})(x-x_{i+2})\bigg)^2\dS\\
&\leq \frac{1}{4} \|\gradt^2 u\|_{L^\infty(\Gamma)}^2 \bigg(\frac{h}{2}\bigg)^4 \int_{x_i}^{x_{i+2}} \dS,
\end{split}
\end{equation}
since $\big((x-x_{i})(x-x_{i+2})\big)^2$ attains its maximum at $x=x_{i+1}$ with a value of $(h/2)^4$. When we sum (\ref{lem:interp_error_gamma_est:4}) over all elements neighboring $\Gamma$, the last integral in (\ref{lem:interp_error_gamma_est:4}) sums up to $L$. Using this estimate in (\ref{lem:interp_error_gamma_est:3}) gives us (\ref{lem:interp_error_gamma_est2}).
\end{proof}

We note that the constant in (\ref{lem:interp_error_gamma_est2}) can be improved, e.g. by evaluating the integral of $\big((x-x_{i})(x-x_{i+2})\big)^2$ in (\ref{lem:interp_error_gamma_est:4}) exactly, instead of estimating it. This gives an improved constant of $1/30$ in  (\ref{lem:interp_error_gamma_est2}). This is purely cosmetic, and only leads to `uglier' fractions in Theorems (\ref{thm:FEM_zero}) and (\ref{thm:FEM_zero_opt}).

\begin{theorem}
\label{thm:FEM_zero}
Let $u\in \WIIi(\Omega)$ and $\tD>0$. Then the solution of (\ref{eq:FEM_zero}) satisfies
\begin{equation}
\begin{split}
\label{thm:FEM_zero:est}
&|u-\tu_h|^2_{H^1(\Omega\setminus\Gamma)} +\tD\int_{\Gamma}[u-\tu_h]^2\dS \\
&\quad\leq 4C_I^2h^2|u|^2_{H^2(\Omega\setminus\Gamma)} + \frac{4}{\tD}L |u|_{\WIi(\Gamma)}^2 +\frac{3}{16}\tD L h^4 \|\gradt^2 u\|^2_{L^\infty(\Gamma)}.
\end{split}
\end{equation}
\end{theorem}
\begin{proof}
We subtract (\ref{lem:cont_sol_eq}) and (\ref{eq:FEM_zero}) with a test function $v_h\in V_h^\Gamma$:
\begin{equation}
\label{thm:FEM_zero:1}
\int_{\Omega\setminus\Gamma}(\nabla u-\nabla\tu_h)\cdotp\nabla v_h \dx -\int_\Gamma\nabla u\cdotp n_\Gamma[v_h]\dS +\tD\frac{h}{2}\sum_{x_i\in\Gamma}[u-\tu_h]_{x_i}[v_h]_{x_i} =0.
\end{equation}
We define $\xi_h=\Pi_h u-\tu_h\in V_h$ and set $v_h=\xi_h$ in (\ref{thm:FEM_zero:1}). Rearranging gives us
\begin{equation}
\label{thm:FEM_zero:3}
\begin{split}
&\int_{\Omega\setminus\Gamma} |\nabla \xi_h|^2\dx +\tD\frac{h}{2}\sum_{x_i\in\Gamma}[\xi_h]_{x_i}^2\\
&= \underbrace{\int_\Gamma\nabla u\cdotp n_\Gamma[\xi_h]\dS}_{(a)}  -\underbrace{\int_{\Omega\setminus\Gamma}(\nabla u-\nabla\Pi_h u)\cdotp\nabla \xi_h \dx}_{(b)} -\underbrace{\tD\frac{h}{2}\sum_{x_i\in\Gamma}[u-\Pi_h u]_{x_i}[\xi_h]_{x_i}}_{(c)}.
\end{split}
\end{equation}
Now we estimate the individual right-hand side terms.

Term $(a)$: Young's inequality (\ref{eq:Young}) with $r=\tD/2$ gives us
\begin{equation}
\label{thm:FEM_zero:4}
(a)\leq \frac{1}{\tD}\int_\Gamma|\nabla u|^2\dS +\frac{\tD}{4}\int_\Gamma[\xi_h]^2\dS \leq \frac{1}{\tD}L\|\nabla u\|^2_{L^\infty(\Gamma)} +\tD\frac{h}{8}\sum_{x_i\in\Gamma}[\xi_h]_{x_i}^2,
\end{equation}
where we have used Lemma \ref{lem:quadrature} with $f=[\xi_h]$. 

Term $(b)$: Young's inequality (\ref{eq:Young}) with $r=1$ and Lemma \ref{lem:Max_angle_cond} give us
\begin{equation}
\begin{split}
\label{thm:FEM_zero:4a}
(b)&\leq \frac{1}{2}\int_{\Omega\setminus\Gamma}|\nabla u-\nabla\Pi_h u|^2 \dx +\frac{1}{2}\int_{\Omega\setminus\Gamma}|\nabla \xi_h|^2 \dx\\ &\leq \frac{1}{2}C_I^2h^2|u|^2_{H^2(\Omega\setminus\Gamma)} +\frac{1}{2}|\xi_h|_{H^1(\Omega\setminus\Gamma)}^2.
\end{split}
\end{equation}

Term $(c)$: Young's inequality (\ref{eq:Young}) with $r=1/2$ along with Lemma \ref{lem:interp_error_gamma} give us
\begin{equation}
\begin{split}
\label{thm:FEM_zero:5}
(c)&\leq \tD\frac{h}{2}\sum_{x_i\in\Gamma}[u-\Pi_h u]_{x_i}^2 +\tD\frac{h}{8}\sum_{x_i\in\Gamma}[\xi_h]_{x_i}^2\\ 
& \leq\tD\frac{1}{64}L h^4 \|\gradt^2 u\|_{L^\infty(\Gamma)}^2 +\tD\frac{h}{8}\sum_{x_i\in\Gamma}[\xi_h]_{x_i}^2.
\end{split}
\end{equation}
Applying the last three estimates to (\ref{thm:FEM_zero:3}) gives us
\begin{equation}
\label{thm:FEM_zero:6}
\begin{split}
&\frac{1}{2} |\xi_h|^2_{H^1(\Omega\setminus\Gamma)} +\tD\frac{h}{4}\sum_{x_i\in\Gamma}[\xi_h]_{x_i}^2\\
&\leq \frac{1}{\tD}L\|\nabla u\|^2_{L^\infty(\Gamma)}  +\frac{1}{2}C_I^2h^2|u|^2_{H^2(\Omega\setminus\Gamma)} +\tD\frac{1}{64}L h^4 \|\gradt^2 u\|_{L^\infty(\Gamma)}^2.
\end{split}
\end{equation}

Now we estimate the error $u-\tu_h$ similarly as in (\ref{thm:FEM_modified:11}) using the facts that $u-\tu_h =(u-\Pi_h u)+\xi_h$ and $(a+b)^2\leq 2a^2+2b^2$:
\begin{equation}
\begin{split}
\label{thm:FEM_zero:7}
&|u-\tu_h|^2_{H^1(\Omega\setminus\Gamma)} +\tD\int_{\Gamma}[u-\tu_h]^2\dS =|u-\tu_h|^2_{H^1(\Omega\setminus\Gamma)} +\tD\|[u-\tu_h]\|_{L^2(\Gamma)}^2\\ &\leq \big(|u-\Pi_h u|_{H^1(\Omega\setminus\Gamma)} + |\xi_h|_{H^1(\Omega\setminus\Gamma)}\big)^2 +\tD\big(\|[u-\Pi_h u]\|_{L^2(\Gamma)}+\|[\xi_h]\|_{L^2(\Gamma)}\big)^2\\
&\leq 2|u-\Pi_h u|_{H^1(\Omega\setminus\Gamma)}^2 + 2|\xi_h|_{H^1(\Omega\setminus\Gamma)}^2 +2\tD\|[u-\Pi_h u]\|_{L^2(\Gamma)}^2 +\tD h\sum_{x_i\in\Gamma}[\xi_h]_{x_i}^2,
\end{split}
\end{equation}
where we have used inequality (\ref{lem:quadrature_est}) with $f=[\xi_h]$. The $\xi_h$-terms in (\ref{thm:FEM_zero:7}) are estimated using (\ref{thm:FEM_zero:6}). The remaining terms are standard:
\begin{equation}
\label{thm:FEM_zero:8}
|u-\Pi_h u|_{H^1(\Omega\setminus\Gamma)}^2\leq C_I^2h^2|u|^2_{H^2(\Omega\setminus\Gamma)},
\end{equation}
and for the third right-hand side term in (\ref{thm:FEM_zero:7}), we use (\ref{lem:interp_error_gamma_est2}). Altogether, we get from (\ref{thm:FEM_zero:7}) 
\begin{equation}
\begin{split}
\label{thm:FEM_zero:9}
&|u-\tu_h|^2_{H^1(\Omega\setminus\Gamma)} +\tD\frac{h}{2}\sum_{x_i\in\Gamma}[u-\tu_h]_{x_i}^2\\ 
& \leq 4 C_I^2h^2|u|^2_{H^2(\Omega\setminus\Gamma)} +\frac{4}{\tD}L\|\nabla u\|^2_{L^\infty(\Gamma)}  +\tD\frac{3}{16}L h^4 \|\gradt^2 u\|_{L^\infty(\Gamma)}^2,
\end{split}
\end{equation}
which is our desired estimate.
\end{proof}

Similarly as in Theorem \ref{thm:FEM_modified_band_opt}, we can now  optimize the value of $\tD$ in order to minimize the right-hand side of (\ref{thm:FEM_zero:est}).

\begin{theorem}
\label{thm:FEM_zero_opt}
Let $u\in \WIIi(\Omega)$. Let $\Th$ contain a zero-measure band $\Gamma$ of length $L$. In (\ref{eq:FEM_zero}) set 
\begin{equation}
\label{thm:FEM_zero_opt:D}
\tD=\frac{1}{h^2}\frac{8}{\sqrt{3}}\frac{|u|_{\WIi(\Gamma)}}{\|\gradt^2 u\|_{L^\infty(\Gamma)}},\end{equation}
then the tempered FEM solution $\tu_h$ satisfies
\begin{equation}
\begin{split}
\label{thm:FEM_zero_opt:est}
&|u-\tu_h|_{H^1(\Omega\setminus\Gamma)} +\sqrt{\tD}\big\|[u-\tu_h]\big\|_{L^2(\Gamma)}\leq C(u,L)h,
\end{split}
\end{equation}
where
\begin{equation}
\begin{split}
\label{thm:FEM_zero_opt:est2}
C(u,L)=2\sqrt{2}C_I|u|_{H^2(\Omega\setminus\B)} + \sqrt{2}\sqrt[4]{3}\sqrt{L} \sqrt{|u|_{\WIi(\Gamma)}\|\gradt^2 u\|_{L^\infty(\Gamma)}}
\end{split}
\end{equation}
is a constant independent of $h$.
\end{theorem}
\begin{proof}
We optimize the last two terms in (\ref{thm:FEM_zero:est}) with respect to $\tD$. To this end, we denote them as
\begin{equation}
\label{thm:FEM_zero_opt:1}
g(\tD)= \frac{4}{\tD}L |u|_{\WIi(\Gamma)}^2 +\frac{3}{16}\tD L h^4 \|\gradt^2 u\|^2_{L^\infty(\Gamma)}.
\end{equation}
The application of Lemma \ref{lem:opt} gives us the optimal choice of $\tD$:
\begin{equation}
\label{thm:FEM_zero_opt:2}
\tD_\opt=\sqrt{\frac{4L |u|_{\WIi(\Gamma)}^2}{\tfrac{3}{16} L h^4 \|\gradt^2 u\|^2_{L^\infty(\Gamma)}}} =\frac{8}{\sqrt{3}}\frac{1}{h^2}\frac{|u|_{\WIi(\Gamma)}}{\|\gradt^2 u\|_{L^\infty(\Gamma)}},
\end{equation}
which is (\ref{thm:FEM_zero_opt:D}). Moreover,
\begin{equation}
\begin{split}
\label{thm:FEM_zero_opt:3}
g(\tD_\opt)&= 2\sqrt{4L |u|_{\WIi(\Gamma)}^2  \tfrac{3}{16} L h^4 \|\gradt^2 u\|^2_{L^\infty(\Gamma)}}\\ &=\sqrt{3}L h^2 |u|_{\WIi(\Gamma)}\|\gradt^2 u\|_{L^\infty(\Gamma)}.
\end{split}
\end{equation}
With the choice $\tD=\tD_\opt$, estimate (\ref{thm:FEM_zero:est}) therefore becomes
\begin{equation}
\begin{split}
\label{thm:FEM_zero_opt:4}
&|u-\tu_h|^2_{H^1(\Omega\setminus\Gamma)} +\tD\|[u-\tu_h]\|_{L^2(\Gamma)}^2 \\
&\quad\leq 4C_I^2h^2|u|^2_{H^2(\Omega\setminus\Gamma)}  +\sqrt{3}L h^2 |u|_{\WIi(\Gamma)}\|\gradt^2 u\|_{L^\infty(\Gamma)}.
\end{split}
\end{equation}
We take the square root of (\ref{thm:FEM_zero_opt:4}) by using (\ref{thm:FEM_modified_band_opt:5}):
\begin{equation}
\begin{split}
\label{thm:FEM_zero_opt:6}
&|u-\tu_h|_{H^1(\Omega\setminus\Gamma)} +\sqrt{\tD}\big\|[u-\tu_h]\big\|_{L^2(\Gamma)}\\
&\quad\leq 2\sqrt{2}C_Ih|u|_{H^2(\Omega\setminus\Gamma)}  +\sqrt{2}\sqrt[4]{3}\sqrt{L} h \sqrt{|u|_{\WIi(\Gamma)}\|\gradt^2 u\|_{L^\infty(\Gamma)}}.
\end{split}
\end{equation}
This is the desired estimate (\ref{thm:FEM_zero_opt:est}) with the constant $C(u,L)$ defined by (\ref{thm:FEM_zero_opt:est2}).
\end{proof}

We could also state Theorem \ref{thm:FEM_zero_opt} in the `quadrature form':

\begin{theorem}
\label{thm:FEM_zero_opt2}
Let $u\in \WIIi(\Omega)$. Let $\Th$ contain a zero-measure band $\Gamma$ of length $L$. In (\ref{eq:FEM_zero}) set 
\begin{equation}
\label{thm:FEM_zero_opt2:D}
\tD=\frac{1}{h^2}\frac{|u|_{\WIi(\Gamma)}}{\|\gradt^2 u\|_{L^\infty(\Gamma)}}8\sqrt{\frac{2}{3}}\end{equation}
then the tempered FEM solution $\tu_h$ satisfies
\begin{equation}
\begin{split}
\label{thm:FEM_zero_opt2:est}
&|u-\tu_h|_{H^1(\Omega\setminus\Gamma)} +\bigg(\tD\frac{h}{2}\sum_{x_i\in\Gamma}[u-\tu_h]_{x_i}^2 \bigg)^{1/2}\leq C(u,L)h,
\end{split}
\end{equation}
where
\begin{equation}
\begin{split}
\label{thm:FEM_zero_opt2:est2}
C(u,L)=2\sqrt{2}C_I|u|_{H^2(\Omega\setminus\B)} + \sqrt[4]{6}\sqrt{L} \sqrt{|u|_{\WIi(\Gamma)}\|\gradt^2 u\|_{L^\infty(\Gamma)}}
\end{split}
\end{equation}
is a constant independent of $h$.
\end{theorem}
\begin{proof}
The proof follows exactly the proofs of Theorems \ref{thm:FEM_zero} and \ref{thm:FEM_zero_opt}, the only difference being the application of (\ref{lem:interp_error_gamma_est1}) instead of (\ref{lem:interp_error_gamma_est2}) in the corresponding version of (\ref{thm:FEM_zero:7}).
\end{proof}
\vspace{0.5cm}

\section*{Declaration of generative AI and AI-assisted technologies in the writing process}
During the preparation of this work the author(s) used DeepL in order to translate and improve the readability and language of the manuscript. After using this tool/service, the author(s) reviewed and edited the content as needed and
take(s) full responsibility for the content of the published article. \\ \\
During the preparation of this work the author(s) used ChatGPT in order to improve the readability and language of the manuscript. After using this tool/service, the author(s) reviewed and edited the content as needed and
take(s) full responsibility for the content of the published article.
\bibliographystyle{abbrv}
\bibliography{Notes} 

 \end{document}